\documentclass[11pt,reqno,a4paper]{amsart}
\usepackage{fullpage} 

\usepackage{amssymb,amsmath,amsthm,graphicx,xcolor,mathtools,enumerate,mathrsfs,tabularx,bbm,tikz,url}
\usepackage[shortlabels]{enumitem} 
\usepackage[british]{babel}
\usepackage[utf8]{inputenc}
\usepackage[T1]{fontenc}
\allowdisplaybreaks

\usepackage{booktabs}
\usepackage{array}
\usepackage{tabularx}

\usepackage[backref=page]{hyperref}
\hypersetup{colorlinks, linkcolor={red!50!black}, citecolor={green!50!black}, urlcolor={blue!50!black}}

\usepackage{tikz}
\usetikzlibrary{math}
\usetikzlibrary{calc,decorations.pathmorphing,through}
\usetikzlibrary{shapes,decorations}
\usetikzlibrary{graphs,graphs.standard}
\usetikzlibrary{decorations.markings}
\usetikzlibrary{positioning, intersections, backgrounds}

\usetikzlibrary{shapes.geometric, arrows}
\usetikzlibrary{decorations.pathmorphing}
\usetikzlibrary{decorations.markings}

\usetikzlibrary{calc,shadings}
\usepackage{pgfplots}

\definecolor{g1}{RGB}{72,100,94}
\definecolor{g2}{RGB}{154,175,140}
\definecolor{g3}{RGB}{199,222,163}

\definecolor{b1}{RGB}{64,100,132}
\definecolor{b2}{RGB}{98,151,179}
\definecolor{b3}{RGB}{157,198,239}

\definecolor{br1}{RGB}{154,82,73}
\definecolor{br2}{RGB}{204,149,116}
\definecolor{br3}{RGB}{230,181,124}

\definecolor{r1}{RGB}{156,49,66}
\definecolor{r2}{RGB}{205,101,57}
\definecolor{r3}{RGB}{232,188,107}

\pgfplotsset{compat=1.18}

\usepackage{caption}
\usepackage{subcaption}

\usepackage[mathlines]{lineno}
\usepackage{etoolbox} 

\newcommand*\linenomathpatch[1]{%
\expandafter\pretocmd\csname #1\endcsname {\linenomath}{}{}%
\expandafter\pretocmd\csname #1*\endcsname{\linenomath}{}{}%
\expandafter\apptocmd\csname end#1\endcsname {\endlinenomath}{}{}%
\expandafter\apptocmd\csname end#1*\endcsname{\endlinenomath}{}{}%
}
\newcommand*\linenomathpatchAMS[1]{%
\expandafter\pretocmd\csname #1\endcsname {\linenomathAMS}{}{}%
\expandafter\pretocmd\csname #1*\endcsname{\linenomathAMS}{}{}%
\expandafter\apptocmd\csname end#1\endcsname {\endlinenomath}{}{}%
\expandafter\apptocmd\csname end#1*\endcsname{\endlinenomath}{}{}%
}

\expandafter\ifx\linenomath\linenomathWithnumbers
\let\linenomathAMS\linenomathWithnumbers
\patchcmd\linenomathAMS{\advance\postdisplaypenalty\linenopenalty}{}{}{}
\else
\let\linenomathAMS\linenomathNonumbers
\fi

\linenomathpatchAMS{gather}
\linenomathpatchAMS{multline}
\linenomathpatchAMS{align}
\linenomathpatchAMS{alignat}
\linenomathpatchAMS{flalign}
\linenomathpatch{equation}

\usepackage{thmtools}

\usepackage{cleveref}

\theoremstyle{plain}
\newtheorem{theorem}{Theorem}[section]
\crefname{theorem}{Theorem}{Theorems}

\newtheorem{proposition}[theorem]{Proposition}
\crefname{proposition}{Proposition}{Propositions}

\newtheorem{corollary}[theorem]{Corollary}
\crefname{corollary}{Corollary}{Corollaries}

\newtheorem{lemma}[theorem]{Lemma}
\crefname{lemma}{Lemma}{Lemmas}

\newtheorem{conjecture}[theorem]{Conjecture}
\crefname{conjecture}{Conjecture}{Conjectures}

\crefname{problem}{Problem}{Problem}

\newtheorem{claim}[theorem]{Claim}
\crefname{claim}{Claim}{Claims}

\newtheorem{observation}[theorem]{Observation}
\crefname{observation}{Observation}{Observations}

\crefname{setup}{Setup}{Setups}

\newtheorem{fact}[theorem]{Fact}
\crefname{fact}{Fact}{Facts}

\crefname{algorithm}{Algorithm}{Algorithms}

\crefname{remark}{Remark}{Remarks}

\crefname{example}{Example}{Examples}

\theoremstyle{definition}
\newtheorem{definition}[theorem]{Definition}
\crefname{definition}{Definition}{Definitions}

\newtheorem{construction}[theorem]{Construction}
\crefname{construction}{Construction}{Constructions}

\crefname{question}{Question}{Questions}

\crefname{tab}{Table}{Tables}

\numberwithin{equation}{section}

\crefname{section}{Section}{Sections}
\crefname{appendix}{Appendix}{Appendix}

\newcommand{\rf}[1]{\cref{#1} (\nameref*{#1})}

\newenvironment{proofclaim}[1][Proof of the claim]{\begin{proof}[#1]}{\end{proof}}

\newcommand*{\dittoclosing}{------}

\newcommand{\bvec}[1]{\mathbf{#1}}

\def\COMMENT#1{}

\usepackage{comment}

\newcommand{\vecb}{\mathbf}
\newcommand{\es}{\emptyset}
\newcommand{\eps}{\varepsilon}
\renewcommand{\rho}{\varrho}
\newcommand{\sm}{\setminus}
\renewcommand{\subset}{\subseteq}
\newcommand{\NATS}{\mathbb{N}}

\newcommand{\INTS}{\mathbb{Z}}
\newcommand{\REALS}{\mathbb{R}}

\DeclareMathOperator{\Hom}{{Hom}}
\renewcommand{\hom}[2]{\Hom(#1;#2)} 
 
\newcommand{\Exp}{\mathbb{E}}
\newcommand{\vn}{\mathbf{1}}

\DeclareMathOperator{\cover}{{cov}}
\DeclareMathOperator{\res}{{res}}

\newcommand\restr[2]{{
	\left.\kern-\nulldelimiterspace 
	#1 
	\vphantom{\big|} 
	\right|_{#2} 
}}

\newcommand{\cB}{\mathcal{B}}

\newcommand{\cD}{\mathcal{D}}

\newcommand{\cF}{\mathcal{F}}
\newcommand{\cG}{\mathcal{G}}

\newcommand{\cL}{\mathcal{L}}
\newcommand{\cM}{\mathcal{M}}

\newcommand{\cS}{\mathcal{S}}

\newcommand{\cU}{\mathcal{U}}
\newcommand{\cV}{\mathcal{V}}
\newcommand{\cW}{\mathcal{W}}

\let\th\relax
\let\div\relax
\let\P\relax

\DeclareMathOperator{\P}{\mathsf{P}} 
 
\DeclareMathOperator{\cQ}{\mathsf{Q}} 
 
\DeclareMathOperator{\til}{\mathsf{Til}} 
\DeclareMathOperator{\mat}{\mathsf{Mat}}
\DeclareMathOperator{\cov}{\mathsf{Cov}}
\DeclareMathOperator{\spa}{\mathsf{Spa}}
\DeclareMathOperator{\div}{\mathsf{Div}}

\newcommand{\SpaF}[1]{\mathsf{\spa}_{#1}}

\DeclareMathOperator{\Deg}{\mathsf{MinDeg}}
\newcommand{\DegF}[2]{\mathsf{\Deg}_{#1,\,#2}}

\DeclareMathOperator{\Del}{\mathsf{Del}}

\DeclareMathOperator{\otil}{\mathsf{O-Til}}
\DeclareMathOperator{\ocov}{\mathsf{O-Cov}}
\DeclareMathOperator{\ospa}{\mathsf{O-Spa}}
\DeclareMathOperator{\odiv}{\mathsf{O-Div}}
\DeclareMathOperator{\Ord}{\mathsf{Ord}}

\DeclareMathOperator{\rtil}{\mathsf{T-Til}}
\DeclareMathOperator{\rcov}{\mathsf{T-Cov}}
\DeclareMathOperator{\rspa}{\mathsf{T-Spa}}
\DeclareMathOperator{\rdiv}{\mathsf{T-Div}}
\DeclareMathOperator{\rmix}{\mathsf{C-Cov}}

\DeclareMathOperator{\pdiv}{\mathsf{P-Div}}

\DeclareMathOperator{\th}{\delta}
\DeclareMathOperator{\con}{\mathsf{Con}}
\DeclareMathOperator{\crit}{cr}

\DeclareMathOperator{\multi}{m}

\DeclareMathOperator{\den}{\mathsf{UniDen}}
\newcommand{\DenF}[2]{\mathsf{\den}_{#1,\, #2}}

\DeclareMathOperator{\dist}{dist}

\newcommand{\PG}[3]{{P^{(#3)}}(#1;#2)}

\newcommand{\rdeg}{\overline{\deg}}

\usepackage[symbols, nogroupskip, sort=def]{glossaries-extra}

\glsxtrnewsymbol[description={homomorphisms from $F$ to~$G$}]
{hom}{\ensuremath{\hom{F}{G}}}

\glsxtrnewsymbol[description={hypergraphs with perfect $F$-tilings}]
{til-hypergraph}{\ensuremath{\til_F}}

\glsxtrnewsymbol[description={\dittoclosing\,  perfect fractional $\Hom(F)$-tilings}]
{spa-hypergraph}{\ensuremath{\spa_F}}

\glsxtrnewsymbol[description={\dittoclosing\,  perfect integral $\Hom(F)$-tilings}]
{div-hypergraph}{\ensuremath{\div_F}}

\glsxtrnewsymbol[description={hypergraphs that are $\Hom(F)$-covered}]
{cov-hypergraph}{\ensuremath{\cov_F}}

\glsxtrnewsymbol[description={minimum $d$-degree of $G$}]
{mindeg}{\ensuremath{\delta_d(G)}}

\glsxtrnewsymbol[description={minimum $d$-degree threshold for perfect $F$-tilings}]
{md-threshold-tiling}{\ensuremath{\th_d^{}(\til_F)}}

\glsxtrnewsymbol[description={\dittoclosing\, perfect fractional $\Hom(F)$-tilings}]
{md-threshold-frac-tiling}{\ensuremath{\th_d^{}(\spa_F)}}

\glsxtrnewsymbol[description={\dittoclosing\, perfect integral $\Hom(F)$-tilings}]
{md-threshold-int-tiling}{\ensuremath{\th_d^{}(\div_F)}}

\glsxtrnewsymbol[description={\dittoclosing\, $\Hom(F)$-covering}]
{md-threshold-cover}{\ensuremath{\th_d^{}(\cov_F)}}

\glsxtrnewsymbol[description={multiplicity of element $v$ in set $e$}]
{multi}{\ensuremath{\multi(v,e)}}

\glsxtrnewsymbol[description={digraph encoding the standard setting}]
{HFG-standard}{\ensuremath{H(F;G)}}

\glsxtrnewsymbol[description={blow-up of a (reduced) digraph $R$ with clusters $\cV$}]
{blow-up-digraph}{\ensuremath{R(\cV)}}

\glsxtrnewsymbol[description={digraphs with perfect matchings}]
{mat-digraph}{\ensuremath{\mat}}

\glsxtrnewsymbol[description={property $s$-graph}]
{property-graph}{\ensuremath{\PG{H}{\P}{s}}}

\glsxtrnewsymbol[description={digraphs with perfect fractional matchings}]
{spa-digraph}{\ensuremath{\SpaF{}}}

\glsxtrnewsymbol[description={digraphs with perfect integral matchings}]
{div-digraph}{\ensuremath{\div}}

\glsxtrnewsymbol[description={covered digraphs}]
{cov-digraph}{\ensuremath{\cov}}

\glsxtrnewsymbol[description={hypergraphs with relative minimum $d$-degree $\delta$}]
{min-deg}{\ensuremath{\DegF{d}{\delta}}}

\glsxtrnewsymbol[description={digraphs with $\rho$-perfect fractional matchings}]
{spa-defect-digraph}{\ensuremath{\SpaF \rho}}

\glsxtrnewsymbol[description={complete $s$-partite $s$-graph with parts of size $b$}]
{complete-partite-graph}{\ensuremath{K_s^{(s)}(b)}}

\glsxtrnewsymbol[description={blow-up of a (reduced) digraph $R$ with clusters of uniform size $b$}]
{blow-up-b}{\ensuremath{R(b)}}

\glsxtrnewsymbol[description={indicator vector of a set or tuple $e$}]
{indicator-vector}{\ensuremath{\vn_e}}

\glsxtrnewsymbol[description={lattice of the digraph $H$}]
{lattice}{\ensuremath{\cL(H)}}

\glsxtrnewsymbol[description={transferral}]
{transferral}{\ensuremath{\vn_v - \vn_u}}

\glsxtrnewsymbol[description={digraphs satisfying $\P$ after vertex deletion}]
{deletion-robustness}{\ensuremath{\Del_{d}(\P)}}

\glsxtrnewsymbol[description={blow-up of a digraph $R$ rooted in $X$}]
{rooted-blow-up}{\ensuremath{R(b;X)}}

\glsxtrnewsymbol[description={partite divisibility property}]
{div-partite}{\ensuremath{\pdiv_{}(\cU)}}

\glsxtrnewsymbol[description={chromatic number}]
{chromatic}{\ensuremath{\chi(F)}}

\glsxtrnewsymbol[description={relative size of a smallest part}]
{tau-F}{\ensuremath{\tau(F)}}

\glsxtrnewsymbol[description={critical chromatic number}]
{critical-chromatic}{\ensuremath{\chi_{\crit}(F)}}

\glsxtrnewsymbol[description={greatest common divisor}]
{gcd}{\ensuremath{\gcd(F)}}

\glsxtrnewsymbol[description={min. $d$-degree threshold for transversal perfect $F$-tilings}]
{thransversal-threshold}{\ensuremath{\th_d^{}(\rtil_F)}}

\glsxtrnewsymbol[description={\dittoclosing\, for colour covering}]
{colour-cover}{\ensuremath{\th_d^{}(\rmix_F)}}

\glsxtrnewsymbol[description={digraph encoding the vertex-ordered setting}]
{digraph-vertex-ordered}{\ensuremath{H^{\Ord}(F;G)}}

\glsxtrnewsymbol[description={uniformly $(\eps,d)$-dense hypergraphs}]
{uni-dense}{\ensuremath{\DenF{\eps}{d}}}

\glsxtrnewsymbol[description={edge density forcing an $F$-tiling of relative size $\beta$}]
{edge-turan}{\ensuremath{\th_0^{}(\SpaF{F};\beta)}}

\makenoidxglossaries

\title{Tiling dense hypergraphs}

\author[R.~Lang]{Richard Lang}

\address[R.~Lang]{
	Departament de Matemàtiques,
	Universitat Politècnica de Catalunya,
	Barcelona, Spain and
	Centre de Recerca Matemàtica, Barcelona, Spain
}

\begin{document}

\begin{abstract}
	The goal in the perfect tiling problem is to cover the vertices of a hypergraph~$G$ with pairwise vertex-disjoint copies of a hypergraph $F$.
	Prior work has identified three necessary conditions for perfect tilings, which correspond to barriers in space, divisibility and covering.
	It is natural to ask for which families of hypergraphs these conditions are also asymptotically sufficient.
	
	Our main result confirms this for all families that are approximately closed under subsampling.
	Among others, this includes families described by minimum degrees and uniform density, which have been studied extensively in this area.
	For instance, we characterise the minimum $d$-degree threshold for perfect $F$-tilings in terms of the thresholds to overcome the space, divisibility and covering barriers:
	\begin{equation*}
		\th_d^{}(\til_F) = \max \left\{	\th_d^{}(\SpaF{F}),\,	\th_d^{}(\div_F),\,	\th_d^{}(\cov_F) \right\}\,.
	\end{equation*}
	As an application, we recover and extend a series of well-known results for perfect tilings in hypergraphs and related settings involving vertex-orderings and transversal structures.
\end{abstract}

\maketitle
\vspace{-0.8cm}

\section{Introduction} \label{sec:introduction}

A basic question in combinatorics is whether a discrete object on a ground set of vertices contains a particular substructure that spans all vertices.
The corresponding decision problems are often computationally intractable, which means that we do not have a `simple' characterisation of the appearance of these substructures.
In extremal combinatorics, research has therefore focused on identifying easily verifiable sufficient conditions, a classic example being minimum degrees in the graph setting.
Over the past decades, a comprehensive literature has evolved around these problems, yet many questions remain wide open~\cite{Kee18,KO09a,SS19,Zha16}.

More recently, efforts have increasingly been dedicated to formulating an axiomatic approach.
The idea is to identify relaxations of the original problem, which describe natural obstacles to a solution.
One then aims to show that robustly solving the relaxations leads to a solution of the original problem.
Important milestones in this direction are due to Keevash and Mycroft~\cite{KM15}, Han~\cite{Han21} (perfect matchings) and Keevash~\cite{Kee14} (designs) in the hypergraph setting.
For graphs, analogous results have been obtained by Kühn, Osthus and Treglown~\cite{KOT2010} (Hamilton cycles), Knox and Treglown~\cite{KT13}, Lang and Sanhueza-Matamala~\cite{LS23} (easily separable graphs) as well as Hurley, Joos and Lang~\cite{HJL25} (perfect mixed tilings).

We continue this line of investigation by introducing a framework for perfect tilings in dense hypergraphs.
Prior research has identified three natural barriers that prevent perfect tilings in this setting, which correspond to obstructions in space, divisibility and covering.
Freschi and Treglown~\cite{FT22} raised the `meta question' of whether these three barriers already include all relevant obstacles.
We confirm this for families of hypergraphs that are approximately closed under subsampling.
Our main contribution, \cref{thm:framework}, states that any hypergraph which robustly overcomes each of the above obstructions must already contain a perfect tiling.
As an application, we recover the central results for perfect tilings under minimum degree conditions in graphs~\cite{KO09,Myc16} and hypergraphs~\cite{HZZ17,Myc16} as well as recent breakthroughs in the ordered setting~\cite{BLT22,FT22}, the quasirandom setting~\cite{DHS+22} and the transversal setting~\cite{CHW+23,MMP22}.

\subsection*{Threshold decompositions}

The implications of our work are best illustrated by the example of minimum degree thresholds.
To formalise this, we introduce some terminology.
A \emph{$k$-uniform hypergraph} (or \emph{$k$-graph} for short) $G$ consists of vertices~$V(G)$ and edges $E(G)$, where each edge $e$ is a set of $k$ vertices.
We often identify $G$ with its edges, writing $e \in G$.
A \emph{perfect matching} of $G$ is a partition of its vertex set into edges.
This is generalised to tilings by replacing the edges with copies of a fixed graph.
More precisely, for a $k$-graph $F$ (called \emph{tile}), a \emph{perfect $F$-tiling} of $G$ is a set of copies of $F$ in $G$ such that every vertex of $G$ is covered exactly once.
Our building blocks for perfect $F$-tilings are the \emph{homomorphisms} from $F$ to~$G$, denoted by \gls{hom}, which are the functions $\phi\colon  V(F) \to V(G)$ that map edges of $F$ to edges of $G$.
A weighting $w\colon \hom{F}{G} \to \REALS$ is called a \emph{perfect $\Hom(F)$-tiling} of $G$ if $\sum_{\phi \in \hom{F}{G}} w(\phi) \cdot |\phi^{-1}(v)| = 1$ for every $v \in V(G)$.
We sometimes refer to a perfect $F$-tiling as \emph{binary}, since it is equivalent to a perfect $\Hom(F)$-tiling whose image is binary.
Denote by \gls{til-hypergraph} the set of hypergraphs with perfect $F$-tilings.
Next, we introduce the aforementioned necessary conditions for perfect tilings.

\subsubsection*{Space} 

We call a perfect $\Hom(F)$-tiling \emph{fractional} if its image is restricted to the unit interval.
Denote by \gls{spa-hypergraph} the set of hypergraphs with perfect fractional $\Hom(F)$-tilings.
A \emph{space barrier} is any hypergraph outside of $\spa_F$.
For instance, we obtain a space barrier in  the $2$-uniform matching setting by taking a  clique on $n$ vertices and removing all edges inside a set $S$ of size above $n/2$. 
Since every vertex in $S$ requires on average one vertex outside of $S$ to be covered, there is not enough space for a perfect fractional matching.
The study of the relation between fractional and integral matchings dates back to the work of Lov\'asz~\cite{Lov74} and Füredi~\cite{Fur81}.
In the minimum degree setting for matchings, the space barrier appears in the work of Rödl, Ruciński and Szemerédi~\cite{RRS09b} as well as Keevash and Mycroft~\cite{KM15}.

\subsubsection*{Divisibility} 

A perfect $\Hom(F)$-tiling is \emph{integral} if its image is restricted to the integers.
Denote by \gls{div-hypergraph} the set of hypergraphs with perfect integral $\Hom(F)$-tilings.
Hypergraphs not covered by $\div_F$ are called \emph{divisibility barriers}.
For example, the disjoint union of two odd cliques is a divisibility barrier in the $2$-uniform matching setting due to the parity obstruction.
Divisibility barriers can be traced back to designs~\cite{GJ73,Wil73}, which can be viewed as perfect matchings in sparse hypergraphs.
In the dense setting, the obstacle was studied systematically for matchings by Keevash and Mycroft~\cite{KM15} in terms of lattice completeness (see \cref{sec:lattice-completeness}).

\subsubsection*{Covering} 

A final obstruction is derived from the trivial fact that in a perfect $F$-tiling every vertex is on a copy of $F$.
Motivated by this, we say that a $k$-graph $G$ is \emph{$\Hom(F)$-covered} if for every vertex $v \in V(G)$, there is a homomorphism $\phi \in \hom{F}{G}$ with $|\phi^{-1}(v)| =1$.
Denote by \gls{cov-hypergraph} the set of $\Hom(F)$-covered hypergraphs.
A hypergraph outside of $\cov_F$ is called a \emph{covering barrier}.
The covering barrier has recently emerged as a third obstacle in the study of various tiling problems~\cite{FT22,HZZ17,MMP22}.
{Somewhat surprisingly, this simple requirement is sometimes the bottleneck, see for instance~\cref{thm:k-partite} or the discussion in~\cref{sec:transversal}}.

\subsubsection*{Thresholds} 

Turning to degree conditions, for $0 \leq d < k$, we define the \emph{minimum $d$-degree} of a $k$-graph~$G$, written \gls{mindeg}, as the largest~$\ell$ such that every $d$-set in $V(G)$ is contained in at least~$\ell$ edges.
Note that for $d=0$, this is just the number of edges in $G$.
For $d\geq 1$ and a $k$-graph~$F$ of order $m=v(F)$, let \gls{md-threshold-tiling} be the infimum over all $\delta \in [0,1]$ such that for all $\mu>0$ and large enough~$n$ divisible by $m$, every $n$-vertex $k$-graph~$G$ with $\delta_d(G) \geq (\delta +\mu) \tbinom{n-d}{k-d}$ admits a perfect $F$-tiling.
Analogously, we define the thresholds \gls{md-threshold-frac-tiling}, \gls{md-threshold-int-tiling} and \gls{md-threshold-cover} for admitting perfect fractional and integral $\Hom(F)$-tilings as well as being $\Hom(F)$-covered, respectively.

\subsubsection*{Decomposition} 

Clearly, $\th_d^{}(\til_F)$ is bounded from below by $\th_d^{}(\spa_F)$, $\th_d^{}(\div_F)$ and $\th_d^{}(\cov_F)$.
As a corollary of our main result, we find that one of these inequalities is always tight.
In other words, all relevant obstacles are already captured by the space, divisibility and covering barriers.

\begin{theorem}\label{thm:minimum-degree-thresholds}
	For every $k$-graph $F$ and $1\leq d < k$, we have
	\begin{equation*}
		\th_d^{}(\til_F) = \max \big\{	\th_d^{}(\SpaF{F}),\,	\th_d^{}(\div_F),\,	\th_d^{}(\cov_F) \big\}\,.
	\end{equation*}
\end{theorem}

Analogous decompositions of the minimum degree thresholds have been obtained by Glock, Kühn, Lo, Montgomery and Osthus~\cite{GLM+19}, (very recently) Delcourt, Henderson, Lesgourgues and Postle~\cite{DHLP25} and Henderson and Postle~\cite{HP25} for designs as well as Han and Treglown~\cite{HT20a} for algorithmic aspects of tilings.
We discuss the consequences of \cref{thm:minimum-degree-thresholds} and past work on perfect tilings in more detail in \cref{sec:applications}.

\subsection*{Overview}

{The remainder of the paper is organised into two parts.
In the first part, we present our general framework.
The main result (\cref{thm:framework}) together with two variants concerning exceptional vertices (\cref{thm:framework-exceptional})  and stability (\cref{cor:stability}) is formulated in \cref{sec:main-results}.
Their proofs can be found in \cref{sec:proof-main-result}.
The second part is dedicated to a series of applications.
In each setting, we determine the barriers to perfect tilings (Table~\ref{tab:barriers}) and then continue to derive concrete outcomes (Table~\ref{tab:recovering}), recovering much of the past work.
The details of these applications are discussed in \cref{sec:applications}, and their proofs can be found in \cref{sec:dirac-hyper-proofs,sec:transversal-proof,sec:ordered-graphs-proofs,sec:quasirandomness-proofs}, respectively.
We conclude the paper with a discussion of open problems in \cref{sec:conclusion}.
The appendix contains a collection of technical results that have been included for the sake of completeness.}

\begin{table}[ht]
	\centering
	\caption{Identifying barriers to perfect tilings}
	\label{tab:barriers}
	\renewcommand{\arraystretch}{1.4}
	\begin{tabularx}{\textwidth}{
			>{}l       
			X                   
			l                   
		}
		\toprule
		Result &   Setting   & Proof \\
		\midrule
		\cref{thm:minimum-degree-thresholds} & $k$-uniform min. $d$-degree thresholds for perfect tilings  (standard) & \cref{sec:main-results} \\
		\cref{thm:transversal-tiling}   &  \dittoclosing\, for transversal  perfect tilings  & \cref{sec:transversal-proof} \\ 
		\cref{cor:threshold-decomposition-ordered}    &  \dittoclosing\, for  vertex-ordered perfect tilings &  \cref{sec:ordered-graphs-proofs}  \\ 
		\cref{cor:uniformly-dense}    & uniformly dense $k$-graphs & \cref{sec:quasirandomness-proofs}\\ 
		\bottomrule
	\end{tabularx}
\end{table}

\begin{table}[ht]
	\centering
	\caption{Applications}
	\label{tab:recovering}
	\renewcommand{\arraystretch}{1.4}
	\begin{tabularx}{\textwidth}{
			>{}l       
			l                   
			X                   
			l                   
		}
		\toprule
		Result & Setting & Recovers/extends work of & Proof \\
		\midrule
		\cref{pro:komlos-kuhn-osthus-thresholds} &    standard  & Kühn--Osthus~\cite{KO09} &  \cref{sec:dirac-hyper-proofs} \\
		\cref{pro:mycroft-thresholds} &        & Mycroft~\cite{Myc16} &    \\
		\cref{thm:k-partite} &        & Han, Zang and Zhao~\cite{HZZ17} &   \\
		\cref{thm:montgomery-muyesser-pehova} &  transversal    & Montgomery, M{\"u}yesser and Pehova~\cite{MMP22} &  \cref{sec:transversal-proof} \\ 
		\cref{thm:cheng-han-wang-wang} &        & Cheng, Han, Wang and Wang~\cite{CHW+23} &    \\ 
		\cref{pro:thresholds-ordered} &   vertex-ordered   & Freschi and Treglown~\cite{FT22}&  \cref{sec:ordered-graphs-proofs}   \\ 
		\cref{lem:quasirandom-divisibility} & quasirandom    & Ding, Han, Sun, Wang and Zhou~\cite{DHS+22} & \cref{sec:quasirandomness-proofs} \\ 
		\bottomrule
	\end{tabularx}
\end{table}

\section{A general framework}\label{sec:main-results}
 
Given $k$-graphs~$F$ and $G$, we aim to find a perfect $F$-tiling of $G$ using the homomorphisms from~$F$ to~$G$ as building blocks.
The set $\hom{F}{G}$ can be viewed as a directed hypergraph~$H$ on $V(G)$, where each homomorphism is encoded as a tuple of vertices.
(The tuples may have repeated vertices.)
Roughly speaking, $H$ tracks the positions of admissible tiles, which for our purposes is the relevant information.
Our main result is therefore formulated in the language of directed hypergraphs.
From a theoretical perspective, this gives a clear view on the general assumptions that lead to perfect tilings.
In practice, it allows us to address various combinatorial settings without mentioning their idiosyncratic features (such as orders, colours and so on).
This includes our applications in \cref{sec:applications}, but also potential future work.

\subsection*{Directed hypergraphs}\label{sec:directed-hypergraphs}

For a set $V$ and a tuple $e \in V^m$, we write $\gls{multi}={|\{i \colon e(i) = v\}|}$ for the \emph{multiplicity} of $v$ in $e$.
A tuple is \emph{{injective}} if it contains no repeated vertices.
An \emph{${m}$-uniform directed hypergraph} (\emph{${m}$-digraph} for short) $H$ is a pair consisting of vertices~$V(H)$ and edges $E(H) \subset {V(H)}^m$.
So edges are allowed to contain repeated vertices, which is non-standard but essential to encode (non-injective) homomorphisms.
We often identify $H$ with its edge set~$E(H)$, writing $e \in H$ instead of $e \in E(H)$.
Unless specified otherwise, we assume that $m\geq 2$.
We write $v(H)=|V(H)|$ for the \emph{order} of $H$ and $e(H) = |E(H)|$ for its \emph{size}.
For convenience, we refer to {(directed) subhypergraphs} of $H$ simply as \emph{subgraphs}.
Given $S \subset V(H)$, we write $H[S]$ for the $m$-digraph \emph{induced by $S$}, which contains all edges of $H$ with entries in~$S$.
We denote by $H-S$  the subgraph of $H$ induced by $V(H) \sm S$.

A \emph{matching} ${\cM \subset H}$ is a subgraph of {injective} edges such that no vertex is covered twice.
Moreover,~$\cM$ is \emph{perfect} if every vertex of $H$ is covered.
We write \gls{mat-digraph} for the set of uniform digraphs with a perfect matching.
We can formulate tiling problems in terms of matchings in digraphs, which is formalised as follows for the standard setting (uniform hypergraphs).

\begin{definition}[Standard setting]\label{def:hom-graph}
	For $k$-graphs $F$ and $G$ with $V(F)=\{1,\dots,m\}$, let \gls{HFG-standard} be the $m$-digraph on vertex set $V(G)$ with an edge $(\phi(1),\dots,\phi(m))$ for every $\phi \in \hom{F}{G}$.
\end{definition}

\COMMENT{We remark that the labelling of $F$ can be taken implicitly in most applications.}
Note that an {injective} edge in $H(F;G)$ corresponds to a copy of $F$ in $G$.
So a perfect $F$-tiling in $G$ is equivalent to a perfect matching in $H(F;G)$.
One can define analogous digraphs for other combinatorial settings such as ordered hypergraphs or transversal problems (see \cref{sec:applications}).

\subsection*{Fissility}\label{sec:fissility}

We now turn to the notion of fissility, which allows us to find {injective} edges in (directed) blow-ups.
To motivate this concept, consider a $k$-graph $G$ that contains a (complete) blow-up of a homomorphic image of a $k$-graph $F$ with  large enough cluster sizes.
In this case, the blow-up and thus $G$ trivially contain a copy of $F$.
For instance, if $F$ is a $6$-cycle and $G$ is a graph containing the blow-up of an edge with clusters of size $3$, then $G$ contains a copy of~$F$.
This allows us to `split' vertices that were `merged' under a homomorphism.
The digraph $H(F;G)$ inherits this property, resulting in an injective edge (copy of $F$) hosted by every large enough blow-up of an edge (homomorphic image of $F$).
This is an important property for our purposes, and we will call abstract digraphs satisfying it `fissile'.
Let us formalise this discussion as follows.

Consider an $m$-digraph $R$ and a family of pairwise disjoint non-empty sets $\cV=\{V_x\}_{x \in V(R)}$.
For an edge $f \in R$, an $m$-tuple $e\in (\bigcup \cV)^m$ is called \emph{$f$-preserving}, if $e(i) \in V_{f(i)}$ for every $1 \leq i \leq m$.
Moreover, $e$ is \emph{strictly $f$-preserving} if in addition $e(i) = e(j)$ whenever $f(i) = f(j)$.
Note that a strictly $f$-preserving edge $e$ is {injective} if and only if $f$ is {injective}.
The \emph{blow-up} \gls{blow-up-digraph} with \emph{reduced graph} $R$ and \emph{clusters} $\cV$ is the $m$-digraph on $\bigcup \cV$, which contains for every edge $f \in R$, all strictly $f$-preserving $m$-edges.
We remark that if $\cV'$ is obtained from $\cV$ by removing the vertices of an edge $e \in R(\cV)$ and $\cV'$ has no empty clusters, then $R(\cV')$ is still a blow-up.
Moreover, $R(\cV)$ contains an {injective} edge if and only if $R$ does.
We call an edge $f \in R$ \emph{ample} if $|V_{f(i)}| \geq \multi(f(i),f) = |\{j \colon f(j) = f(i)\}|$ for each $1 \leq i \leq m$, which is a trivial requirement for the existence of an $f$-preserving {injective} edge hosted by a digraph $H$ containing $R(\cV)$.

\begin{definition}[Fissility]\label{def:fissility}
	An $m$-digraph $H$ is called \emph{fissile} if for every blow-up $R(\cV) \subset H$ and ample $f \in R$, there is an $f$-preserving {injective} edge in $H$.
\end{definition}

Note that in particular $H$ is fissile if all its edges are {injective}.
Since fissility is a somewhat abstract concept, we illustrate it with the example of the standard setting.

\begin{observation}\label{obs:hom-graph-fissile}
	Let $F$ and $G$ be $k$-graphs with $V(F) = \{1,\dots,m\}$.
	Then $H(F;G)$ is fissile.
\end{observation}

\begin{proof}
	Set $H = H(F;G)$, and consider a blow-up $R(\cV) \subset H$ with a family $\cV=\{V_x\}_{x \in V(R)}$ and an ample $m$-edge $f \in R$.
	We show that $H$ contains an $f$-preserving {injective} edge.
	Let $e$ be a strictly $f$-preserving edge in $R(\cV)$.
	So $e$ corresponds to a homomorphism $\phi_e \in \hom{F}{G}$ by definition of $H$.
	Let $\phi \in \hom{F}{R}$ be the corresponding homomorphism, which maps each vertex $v \in V(F)$ to the vertex $x \in V(R)$ for which $\phi_e(v) \in V_x$.
	
	Now consider an edge $v_1\cdots v_k \in F$, and let $\phi(v_i) = x_i$ for $1 \leq i \leq k$.
	Note that $x_1,\dots,x_k$ are pairwise distinct.
	We claim that every set $\{w_1,\dots,w_k\}$ with $w_i \in V_{x_i}$ is an edge in $G$.
	To see this, define a function $\phi' \colon V(F) \to V(G)$ by setting $\phi'(v) = w_i$ if $\phi_e(v) \in V_{x_i}$ and $\phi'(v) = \phi_e(v)$ otherwise.
	Note that $\phi'$ is in fact a homomorphism, since its corresponding $m$-tuple is strictly $f$-preserving in $R$ and $R(\cV) \subset H$.
	Since homomorphisms map edges to edges, it follows that $w_1 \cdots w_k \in G$ as claimed.
	
	Since $f$ is ample, we  can thus select for each $x = f(j)$ with $1 \leq j \leq m$, a number of $|\phi^{-1}(x)| = \multi(x,f)  \leq |V_{x}|$ vertices in each cluster $V_{x}$ to find a copy of $F$ in $G$ on exactly these vertices.
	This results in an $f$-preserving {injective} edge in $H$.
\end{proof}

\subsection*{Obstacles to perfect matchings}\label{sec:necessary-conditons}

Given these preliminaries, we shall focus on perfect matchings in {fissile} digraphs.
In the following, we recover the obstacles and underlying relaxations of \cref{sec:introduction} in this setting.
Consider an $m$-digraph $H$ with a weight function $w \colon H \rightarrow \REALS$ such that for all $v\in V(H)$, we have $\sum_{e \in H} w(e) \multi(v,e) = 1$.
Note that $w$ is equivalent to a perfect matching if its image is binary.
We call $w$ a \emph{perfect fractional  matching} if its image is in the unit interval.
Moreover, $w$  is an \emph{integral perfect matching} if its image is in the integers.
We say that an $m$-digraph $H$ is \emph{covered} if for every vertex $v \in V(H)$, there is an $m$-edge $e \in H$ in which~$v$ has multiplicity one.
Note that $H$ is covered if it has a perfect matching, since {injective} edges do not have repeated entries. 
Denote by \gls{spa-digraph} and \gls{div-digraph} the sets of digraphs which have perfect fractional and integral matchings, respectively.
Let \gls{cov-digraph} be the set of digraphs that are covered.
We can relate these digraph properties to the standard setting (\cref{def:hom-graph}) as follows.

\begin{observation}\label{obs:digraph-property-equivalence}
	For a $k$-graph $F$, a $k$-graph $G$ satisfies $\til_F$, $\spa_F$, $\div_F$ and $\cov_F$ if and only if $H(F;G)$ satisfies $\mat$, $\spa$, $\div$ and $\cov$, respectively.
\end{observation}

A \emph{space}, \emph{divisibility} or \emph{covering barrier} is any digraph not satisfying $\SpaF{}$, $\div$ or $\cov$, respectively.
We emphasise once more the necessity of each of these properties.

\begin{observation}\label{obs:matching-to-space-divisibility-cover}
	If a digraph $H$ satisfies $\mat$, then it also satisfies $\spa$, $\div$ and $\cov$. 
\end{observation}

\subsection*{Robustness}\label{sec:robustness}

To state our main result, we introduce a notion of robustly satisfying a digraph property.
This is formalised in terms of an auxiliary hypergraph, which encodes the positions where a given digraph inherits the property.

\begin{definition}[Property graph]\label{def:property-graph}
	For an $m$-digraph $H$ and a family of $m$-digraphs $\P$, the \emph{property graph}, denoted by \gls{property-graph}, is the $s$-graph on vertex set $V(H)$ with an edge $S \subset V(H)$ whenever the induced subgraph $H[S]$ \emph{satisfies}~$\P$, that is $H[S] \in \P$.
\end{definition}

A digraph $H$ robustly satisfies a property $\P$ if the property graph $\PG{H}{\P}{s}$ is locally dense:

\begin{definition}[Robustness]\label{def:robustness-detailed}
	For a family of $m$-digraphs $\P$, an $n$-vertex $m$-digraph $H$ satisfies \emph{$(\delta,r,s)$-robustly} $\P$ if the minimum $r$-degree of the property $s$-graph $\PG{H}{\P}{s}$ is at least $\delta  \tbinom{n-r}{s-r}$.
	Moreover, we just write \emph{$s$-robustly} in the case when $\delta = 1-1/s^2$ and $r=1$.
\end{definition}

{We note that the value $1-1/s^2$ guarantees the existence of a perfect matching in the property $s$-graph $P$  (\cref{thm:DH81}).}
Better bounds are available, but not required in our context, since in practice $P$ has a much larger minimum degree anyway due to the inheritance principle (see below).

\subsection*{Sufficient conditions}\label{sec:sufficient-matching}

Our main result inverts the implication of \cref{obs:matching-to-space-divisibility-cover}.
It states that every directed hypergraph which robustly overcomes the space, divisibility and covering barriers has a perfect matching:

\begin{theorem}[Main result]\label{thm:framework}
	For all $2\leq m \leq s$, there is $n_0 \geq s$ such that every fissile $m$-digraph $H$ on $n \geq n_0$ divisible by $m$ vertices that $s$-robustly satisfies $\spa \cap \div \cap \cov$ has a perfect matching.
\end{theorem}

The proof of this result is given in \cref{sec:proof-main-result}.
In  the rest of this section, we discuss applications, related work and variants involving exceptional vertices and stability.

\subsection*{Example application}\label{sec:example-application}

To see how \cref{thm:framework} is used, consider a digraph family $\P$ with the ambition to show that its elements $H$ contain perfect matchings, that is $\P \subset \mat$.
Suppose that~$\P$ is approximately closed under subsampling, a phenomenon that is captured by an {Inheritance Lemma} (such as \cref{lem:inheritance-minimum-degree,lem:inheritance-uniformly-dense}).
Informally, this translates to the property $s$-graph $\PG{H}{\tilde \P}{s}$ being locally dense for a slightly perturbed property $\tilde \P = \{\tilde R\colon R \in \P\}$.
Thanks to \cref{thm:framework}, the task of showing that $\P \subset \mat$ is then reduced to showing that $\tilde \P$ satisfies $\spa$, $\div$ and $\cov$, which is a considerably simpler task.

More concretely, let us illustrate a typical application of \cref{thm:framework} by deriving \cref{thm:minimum-degree-thresholds}.
The \emph{relative minimum $d$-degree} of an $n$-vertex $k$-graph $G$ is $\delta_d(G)/\binom{n-d}{k-d}$.
For $\delta \in [0,1]$, let \gls{min-deg} be the union of all $k$-graphs with $k>d$ and relative minimum $d$-degree at least $\delta$.
It is not hard to see that minimum degree conditions are approximately closed under subsampling.

\begin{lemma}[Inheritance Lemma for minimum degree]\label{lem:inheritance-minimum-degree}
	For $0 \leq d \leq k-1$, $q \geq 1$ and $\mu > 0$, there is~$s_0=s_0(k,q,\mu)$ such that for every $s\geq s_0$ the following holds.
	Let $G$ be a $k$-graph on $n\geq s$ vertices with $\delta_d(G) \geq (\delta + \mu) \tbinom{n-d}{k-d}$ where $\delta \in [0,1]$.
	Then the property $s$-graph $P =\PG{G}{\DegF{d}{\delta+\mu/2}}{s}$ satisfies $\delta_{q}(P) \geq  (1-e^{-\sqrt{s}}    )  \tbinom{n-q}{s-q}$.
\end{lemma}

The proof of \cref{lem:inheritance-minimum-degree} follows from a standard concentration analysis and can be found in \cref{sec:inheritance-lemma-minimum-degree}.
Now we are ready to derive the threshold decomposition from \cref{thm:framework}.

\begin{proof}[Proof of \cref{thm:minimum-degree-thresholds}]
	Given $\mu > 0$ and $V(F) = [m]$, choose $s_0$ as in \cref{lem:inheritance-minimum-degree} (applied with $q=1$) and $s\geq s_0$ large enough such that $e^{- \sqrt{s}} \leq 1/s^2$.
	Suppose that $n_0$ satisfies \cref{lem:inheritance-minimum-degree,thm:framework}.
	Set $\delta = \max \big\{	\th_d^{}(\SpaF{F}),\,	\th_d^{}(\div_F),\,	\th_d^{}(\cov_F) \big\}$.
	Let $G$ be a $k$-graph on $n\geq n_0$ divisible by $m$ vertices with $\delta_d(G) \geq (\delta+\mu) \binom{n-d}{k-d}$.
	In light of \cref{obs:digraph-property-equivalence}, our goal is to show that the $m$-digraph $H = H(F;G)$ has a perfect matching.
	
	It follows by \cref{lem:inheritance-minimum-degree} that the property $s$-graph $P=\PG{G}{\DegF{d}{\delta+\mu/2}}{s}$ satisfies 
	$\delta_1(P) \geq (1-1/s^2) \tbinom{n-1}{s-1}$.
	So by definition of $\delta$ and \cref{obs:digraph-property-equivalence}, the $m$-digraph $H[S]$ satisfies $\spa \cap \div \cap \cov$ for every $S \in P$.
	Thus $\spa \cap \div \cap \cov$ is $s$-robustly satisfied by $H$.
	Finally, note that $H$ is fissile by \cref{obs:hom-graph-fissile}.
	Therefore $H$ contains a perfect matching by \cref{thm:framework}.
\end{proof}

\subsection*{Related work}

{Keevash and Mycroft~\cite{KM15} as well as Han~\cite{Han21} investigated similar phenomena in the setting of $k$-graphs.
These results differ in their notion of robustness and in their proof techniques.
In particular, Keevash and Mycroft~\cite{KM15} introduced the concept of completeness for lattices (see \cref{sec:lattice-completeness}) and used it to find a suitable allocation for the Hypergraph Blow-up Lemma.
Independently, Lo and Markström~\cite{LM15} developed an absorption-based approach to hypergraph tiling using a (more restrictive) form of lattice completeness.
Han~\cite{Han21} combined and extended these ideas to give a simpler proof of the Keevash--Mycroft Theorem avoiding the (Strong) Hypergraph Regularity Lemma.}

Our framework contributes to this line of research in two ways.
Firstly, the interface is simple but practical.
For host graph families that are approximately closed under subsampling, \cref{thm:framework} essentially decomposes the problem of finding perfect tilings into verifying the space, divisibility and cover properties independently (as illustrated by \cref{thm:minimum-degree-thresholds}), which can greatly simplify the analysis.
The fact that $m$-digraphs are allowed to have repeated vertices adds significant flexibility to this approach.
In combination, we obtain short and insightful proofs of many old and new results.

The second noteworthy point about \cref{thm:framework} is that the proof itself is quite short.
The argument is self-contained, after discounting classic insights from combinatorics, and it does not involve the Regularity Lemma.
The techniques can easily be extended to other setups concerning exceptional vertices and the partite setting.
Finally, our framework can also be used to derive stability results via the theory of property testing (see below).

{Since the announcement of this paper, further applications of our framework and the introduced methods have appeared.
We discuss this in more detail in \cref{sec:conclusion}.}

\subsection*{Exceptional vertices} 

The following variant of our main result can be applied when the space property is satisfied up to a small defect.
The \emph{size} of a {fractional matching} $w \colon H \rightarrow [0,1]$  in an $m$-digraph $H$ on $n$ vertices is $\sum_{e \in H} w(e) \leq n/m$.
Note that a perfect fractional matching has size exactly $n/m$.
We say that $w$ is \emph{$\rho$-perfect} if its size is at least $(1-\rho)n/m$.
For $0 \leq \rho \leq 1$, let \gls{spa-defect-digraph} be the set of uniform digraphs with a $\rho$-perfect fractional matching.

\begin{theorem}
	\label{thm:framework-exceptional}
	For all $2\leq m \leq r$ and $\alpha >0$ there is $\rho > 0$ such that for every $s \geq m$ there is an $n_0$  with the following property.
	Let $H$ be a fissile $m$-digraph on $n \geq n_0$ divisible by $m$ vertices. 
	Suppose that $H$ satisfies $r$-robustly  $\div \cap \cov$, and $H-A$ satisfies $s$-robustly $\SpaF \rho$ for every set $A \subset V(H)$ of at most $\alpha n $ vertices.
	Then $H$ has a perfect matching.
\end{theorem}

\COMMENT{For simplicity, we state the degree condition on $\PG{H-A}{  \SpaF \rho}{s}$ in terms of $n$.}

We remark that the assumption on $H-A$ is required for applications where $\alpha$ is much larger than $1/s$, which is usually the case.
The proof of \cref{thm:framework-exceptional} can be found in \cref{sec:proof-main-result}.

\subsection*{Stability}\label{sec:stability}

When studying exact (as opposed to approximate) minimum degree conditions for perfect tilings, one typically proceeds via a stability analysis.
The idea is to show that the graph in question either has a perfect tiling or must be close to an extremal construction.
One then continues with a more careful analysis of the extremal cases.
Techniques based on the Regularity Lemma typically allow one to carry out such a stability analysis (see for instance the work of Kühn and Osthus~\cite{KO09}).
We can combine our framework with results from property testing to obtain similar structural insights.
Recall that an Inheritance Lemma states that a property $\cQ$ is approximately preserved under subsampling, that is, under taking typical induced subgraphs of constant order.
In property testing, the focus is on the inverse statement:
given that a typical induced subgraph of constant order satisfies $\cQ$, does the original graph satisfy $\cQ$ approximately?

A \emph{hypergraph property} is a set of hypergraphs that is closed under taking isomorphisms.
For an $n$-vertex $k$-graph $G$ and a $k$-graph property $\cQ$, let $\dist(G,\cQ)$ be the minimum of the (edit) distance $\dist(G,G') =|E(G) \bigtriangleup E(G')|/\binom{n}{k}$ over all $G' \in \cQ$ with $V(G') = V(G)$.
We denote by $B_{\mu}(\cQ)$ the set of all $k$-graphs $G$ with $\dist(G,\cQ) \leq \mu$.
We say that a $k$-graph property $\cQ$ is \emph{testable} if for every $\mu > 0$, there is a positive integer $s$ such that if we select for an $n$-vertex $k$-graph $G$ with $n \geq s$ and $\dist(G,\cQ) > \mu$ an $s$-set $S \subset V(G)$ uniformly at random then the probability that $\dist(G[S],\cQ) \leq \mu $ is at most $\mu$.
The systematic study of testable properties goes back to the work of Goldreich, Goldwasser and Ron~\cite{GGR98}.
It was shown by Fisher and Newman~\cite{FN07} and Alon, Fischer, Newman and Shapira~\cite{AFNS2009} that a graph property is testable if and only if it can be expressed by a `regularity instance', a notion closely related to the Regularity Lemma.
The analogous result for hypergraphs was proved by Fischer, Matsliah and Shapira~\cite{FMS2010}.
See also the work of Joos, Kim, Kühn, Osthus~\cite[Theorem 3.1]{JKKO23}, where the connection with property testing and random sampling in hypergraphs is made explicit.
We remark that many `natural' properties can be tested without the use of (Hypergraph) Regularity Lemmas~\cite{GGR98,LR25}.

To draw the connection between property testing and our work, we derive the following corollary for the standard setting, which transfers stability results for space, divisibility and cover properties to stability results for perfect tilings.
The proof can be found in \cref{sec:proof-main-result}.

\begin{corollary}
	\label{cor:stability}
	Let $F$ be a $k$-graph on $m$ vertices.
	Let $\cQ$ be a testable $k$-graph property.
	Then, for every $\mu > 0$, there are $s,n_0>0$ with the following property.
	Let $G$ be an $n$-vertex $k$-graph~$G$ with $n \geq n_0$ divisible by $m$.
	Let $H = H(F;G)$, $\P = \spa \cap \div \cap \cov$ and $P = \PG{H}{\P}{s}$.  
	Let $Q$ be $s$-graph on $V(G)$ with an edge $S$ if $G[S] \in B_{\mu/4}(\cQ)$.
	Suppose that $\delta_1 \big(P \cup Q \big) \geq  \left(1-1/s^2 + \mu\right)  \tbinom{n-1}{s-1}$.
	Then $G$ has a perfect $F$-tiling or satisfies $B_\mu(\cQ)$.
\end{corollary}

The purpose of \cref{cor:stability} is to turn stability results for space, divisibility and covering into stability statements for perfect tilings.
We illustrate  this  with the following (trivial) application in the setting of $2$-uniform graphs.
(For a more consequential application of these ideas, see the work of Letzter and Ranganathan~\cite{LR25}.)
Let $\cQ_1$ be the family of graphs obtained from the union of two disjoint cliques of equal size.
Let $\cQ_2$ be the family of complete  balanced bipartite graphs.
Note that these graphs correspond (roughly) to the extremal constructions (space and divisibility barrier) for perfect matchings under minimum degree conditions.
It follows from the work of Goldreich, Goldwasser and Ron~\cite{GGR98} that $\dist(\cdot,\cQ)$ is testable for $\cQ = \cQ_1 \cup \cQ_2$.

Consider the family $\cG_\eps$ of graphs of minimum relative degree at least $1/2-\eps$.
Let $F$ be the graph consisting of a single edge.
Given $\mu > 0$ and choosing $\eps > 0$ small enough, one can show that for a large enough graph $R \in \cG_{2\eps}$, we have $R \in \spa_F \cap \div_F \cap \cov_F$ or we have $R \in B_{\mu/4}(\cQ)$.
This corresponds to a stability statement for space, divisibility and covering.

Now let $G$ be a sufficiently large graph in $\cG_\eps$.
Let $P' =\PG{G}{\cG_{2\eps}}{s}$, $P = \PG{H}{\P}{s}$ and let~$Q$ be $s$-graph on $V(G)$ with an edge $S$ if $G[S] \in B_{\mu/4}(\cQ)$.
By the above, we have $P' \subset P \cup Q$ for $s$ large enough with respect to $\eps$ and $\mu$.
Moreover, it follows that $\delta_1 \big(P' \big) \geq  \left(1-1/s^2 + \mu\right)  \tbinom{n-1}{s-1}$ by \cref{lem:inheritance-minimum-degree}.
Hence, we may apply \cref{cor:stability} to find that $G$ contains a perfect matching or $G \in B_\mu(\cQ)$.
This gives a stability statement for perfect matchings.

\section{Proof of the main results}\label{sec:proof-main-result}

In this section, we prove  \cref{thm:framework,thm:framework-exceptional,cor:stability}.
We begin with an informal discussion on how the assumptions of \cref{thm:framework} are used.

\subsection*{Proof outline}

Consider an $n$-vertex $m$-digraph $H$ as in \cref{thm:framework}.
Our task is to find a perfect matching in $H$.
For convenience, we ignore divisibility conditions on the order (such as~$m$ dividing $n$) in the following discussion.

We begin with a simplified scenario.
Assume that~$H$ contains a large blow-up $R(\cV)$ with parts of equal size such that $R$ is an $s$-vertex $m$-digraph in $\spa \cap  \div \cap \cov $.
In this case, it is possible to show that $H$ contains a matching that spans the vertices of $R(\cV)$.\footnote{We do not verify this claim, but it follows immediately from the ideas used in the proofs of \cref{obs:fractional-to-integral,lem:absorber}. 
	A formal proof appeared in subsequent work~\cite[Proposition 9.5]{LS24a}.}
We informally call embedding into a blow-up the `allocation' of a perfect matching.
The properties $\spa$, $\cov$ and $\div$ allow us to find such an allocation.

So if $H$ is spanned by a single (suitable) blow-up, we are done.
Of course, this is not always the case in the context of \cref{thm:framework}.
However, as it turns out, the assumption on the minimum degree of the property $s$-graph $\PG{H}{\spa \cap  \cov \cap \div}{s}$ allows us to find a collection of pairwise vertex-disjoint blow-ups $R_1(\cV_1),\dots,R_\ell(\cV_\ell)$ that together span the vertices of $H$ and whose reduced graphs $R_1,\dots,R_\ell$ are $s$-vertex $m$-digraphs in $\spa \cap \div \cap \cov$.\footnote{As before, we do not show this statement formally, but it follows from the arguments used in the proofs of \cref{lem:covering-with-blow-ups,lem:local-blow-ups}. See also subsequent work~\cite[Proposition 6.3]{LS24a}.}
Given this, we can find a perfect matching of $H$ by picking a perfect matching in each of the blow-ups.
In practice, we implement these ideas using the Absorption Method (see below), which adds additional flexibility that is required for the proof of \cref{thm:framework-exceptional}.

We note that this approach bears some similarities to the embedding procedure when using a Regularity Lemma.
In this context, one works with a single `reduced' $s$-vertex $m$-digraph~$R$ and a partition $\cV$ of $V(H)$ whose parts correspond to the vertices of $R$.
As above, one can allocate a perfect matching provided that $R(\cV) \subset H$.
However, the latter inclusion is usually not given.
Instead, one can find a `quasirandom blow-up' of $R$ on the clusters of $\cV$ in $H$, meaning that the edges of $R$ are replaced with quasirandom subgraphs of $H$.
Given this, one can turn an allocation of a perfect matching into the embedding of a perfect matching via technical instruments for quasirandomness such as the Hypergraph Blow-up Lemma~\cite{KM15}.
Since in our approach, we identify the blow-ups directly as subgraphs of $H$, we are able to avoid this machinery, which simplifies the proofs considerably.

\subsection*{Details}

The actual proof of \cref{thm:framework,thm:framework-exceptional} is split into two parts.
First we match most of the vertices, and then we integrate the remaining ones.
This process is called the Absorption Method and was popularised (in its modern form) by Rödl, Ruciński and Szemerédi~\cite{RRS06}.
The first step is formalised as follows.

\begin{proposition}[Almost perfect matching]\label{prp:almost-perfect-matching}
	For all $2\leq m \leq s$, $\eta >0$ and  $\rho \geq 0$, there is an $n_0>0$ such that every fissile $m$-digraph $H$ on $n \geq n_0$ vertices that $(\delta,1,s)$-robustly satisfies~$\SpaF \rho$  for $\delta=1-1/s^{1.5}$ has a matching covering all but at most $(\eta+\rho)n$ vertices.
\end{proposition}

In the second step, we take care of the leftover vertices using a special `absorption structure' that is provided by the following result.

\begin{proposition}[Absorption]\label{prp:absorption}
	For all $2\leq m \leq s$ and $\alpha >0$, there are $\eta ,n_0>0$ such that 
	every fissile $m$-digraph $H$ on $n \geq n_0$ vertices that $s$-robustly satisfies $\div \cap \cov$ admits a set $A \subset V(H)$ with $|A| \leq \alpha n$ divisible by $m$ such that $H[A \cup L]$ has a perfect matching for every set $L \subset V(H-A)$ with $|L| \leq \eta n$ divisible by $m$. 
\end{proposition}

We remark that the above propositions do not require $n$ to be divisible by $m$.
The proofs of \cref{prp:almost-perfect-matching,prp:absorption} can be found in \cref{sec:absorption,sec:almost-perfect-matching}.
Given this setup, it is easy to derive our main result.

\begin{proof}[Proof of \cref{thm:framework}]  
	Given $2\leq m \leq s$, choose $\alpha >0$ sufficiently small.
	Let $\eta > 0$ be small enough to apply \cref{prp:absorption}.
	Choose $n_0$ large enough so that \cref{prp:almost-perfect-matching,prp:absorption} are satisfied, where the former is applied with $\rho=0$.
	For $n \geq 2n_0$ divisible by $m$, let $H$ be an $m$-digraph as in the statement.
	We obtain $A \subset V(H)$ by an application of \cref{prp:absorption}.
	Let $P = \PG{H-A}{\spa}{s}$, and note that $\delta_1\left(P\right) \geq  (1-1/s^{1.5})  \binom{\bar n-1}{s-1}$ by choice of $\alpha$, where $\bar n =n-|A|$.
	This allows us to apply \cref{prp:almost-perfect-matching} and find a matching $\cM_1$ in $H-A$ covering all but at most $\eta n$ vertices.
	Denote by $L$ the set of uncovered vertices.
	By choice of~$A$, there is a perfect matching $\cM_2$ of $H[A\cup L]$.
	So $\cM_1 \cup \cM_2$ is a perfect matching of~$H$.
\end{proof}
 
\COMMENT{The proof of \cref{thm:framework-exceptional} follows analogously, so we only sketch the differences.
	We obtain~$\eta$ by from \cref{prp:absorption} with $r$ playing the rôle of $s$.
	Next, we apply \cref{prp:absorption} with~$\eta/2$ playing the rôle of both $\eta$ and $\rho$.
	This allows us to pick $A \subset V(H)$ of size at most $\alpha n$.
	We then use \cref{prp:almost-perfect-matching} on $H-A$ to cover all but at most $\eta n$ vertices and finish as before.}

The proof of \cref{thm:framework-exceptional} follows analogously.
We include the details for the sake of completeness.
\begin{proof}[Proof of \cref{thm:framework-exceptional}]
	Given $2\leq m \leq r$ and $\alpha > 0$, let $\eta > 0$ be small enough to apply \cref{prp:absorption} with $r$ playing the rôle of $s$.
	Set $\rho = \eta/2$, and consider $s \geq m$.
	Choose $n_0$ large enough so that \cref{prp:almost-perfect-matching,prp:absorption} are satisfied, where the former is applied with $\eta/2$ playing the rôle of $\eta$, while the latter is applied with $\eta/2$ and $r$ playing the rôles of $\eta$ and $s$.
	For $n \geq 2n_0$ divisible by $m$, let $H$ be an $m$-digraph as in the statement.
	We obtain $A \subset V(H)$ by an application of \cref{prp:absorption}.
	In particular, $|A| \leq \alpha n$.
	So by assumption $H-A$ satisfies $s$-robustly $\SpaF{\rho}$.
	This allows us to apply \cref{prp:almost-perfect-matching} and find a matching $\cM_1$ in $H-A$ covering all but at most $(\eta/2 + \rho) n = \eta n$ vertices.
	Denote by $L$ the set of uncovered vertices.
	By choice of~$A$, there is a perfect matching $\cM_2$ of $H[A\cup L]$.
	So $\cM_1 \cup \cM_2$ is a perfect matching of~$H$.
\end{proof}

It is also not hard to see how stability can be derived from \cref{thm:framework}:

\begin{proof} [Proof of \cref{cor:stability}]
	Consider an $m$-vertex $k$-graph $F$, a testable $k$-graph property $\cQ$ and $\P = \spa \cap \div \cap \cov$.
	Given $\mu  > 0$, let $s$ and $n$ be sufficiently large with $s \leq n$ and $n$ divisible by~$m$.
	Let $G$ be a $k$-graph on $n$ vertices with $H = H(F;G)$ as in \cref{def:hom-graph} and $P = \PG{H}{\P}{s}$ as in \cref{def:property-graph}.
	Let $Q$ be $s$-graph on $V(G)$ with an edge~$S$ if $G[S] \in B_{\mu/4}(\cQ)$.
	Suppose that $\delta_1 \big(P \cup Q\big) \geq  \left(1-1/s^2 + \mu\right)  \tbinom{n-1}{s-1}$.
	
	If $\delta_1 \big( P \big) \geq  \left(1-1/s^2\right)  \tbinom{n-1}{s-1},$
	then $G$ has a perfect $F$-tiling by \cref{obs:digraph-property-equivalence,thm:framework}, and we are done.
	So suppose otherwise.
	But then there is a vertex $v \in V(G)$ such that $\deg_{Q}(v) \geq \mu  \cdot \binom{n-1}{s-1}$.
	By the definition of $Q$, this means that $G[S \cup \{v\}]$ satisfies $B_{\mu/4}(\cQ)$ with probability at least $\mu$ when choosing an $(s-1)$-set $S \subset V(G-v)$ uniformly at random.
 	Since $s$ is large enough, $G[S]$ also satisfies $B_{\mu/2}(\cQ)$ in this case.
	Put differently, with probability at least $\mu$ we have $\dist(G[S], \cQ) \leq \mu/2$ when selecting an $(s-1)$-set $S\subset V(G-v)$ uniformly at random.
	
	We may assume that $\dist(G,\cQ) > \mu \geq \mu/2$, since we are done otherwise.
	As $\cQ$ is testable, it follows that the probability that $\dist(G[S],\cQ) \leq \mu/2$ is at most $\mu/2$ for a uniformly chosen $(s-1)$-set $S \subset V(G-v)$.
	So with probability strictly above $1-\mu/2$, we have that $\dist(G[S],\cQ) > \mu/2$.
	Combined with the lower bound of $\mu$ for the complementary event, this is a contradiction.
\end{proof}

The remainder of this section is dedicated to the proofs of \cref{prp:absorption,prp:almost-perfect-matching} and thus concludes the proof of the main result.
We prepare ourselves with a few preliminary results in \cref{sec:preliminaries}.
This is followed by the proof of \cref{prp:almost-perfect-matching} in \cref{sec:almost-perfect-matching}.
We discuss an auxiliary boosting process in \cref{sec:boosting}.
\cref{sec:lattice-completeness} is dedicated to the concept of lattice completeness.
Finally, we show \cref{prp:absorption} in \cref{sec:absorption}.

\subsection{Preliminaries}\label{sec:preliminaries}

In the following, we introduce a few general definitions and tools.

\subsubsection*{General terminology}
For an integer $n\geq1$, we let $[n]=\{1,\dots,n\}$.
For $x,y,z \geq 0$, we write $x = y \pm z$ to mean $y - z \leq x \leq y+z$.
We express some of the constant hierarchies in standard $\gg$-notation.
	To be precise, we write $y \gg x$ to mean
	that for any $y \in (0, 1]$ there exists an $x_0 \in (0,1)$
	such that for all $x \in (0,x_0]$ the subsequent statements
	hold.  Hierarchies with more constants are defined in a
	similar way and are to be read from left to right following the order that the constants are chosen.
Moreover, we tend to ignore rounding errors if the context allows~it.

\subsubsection*{Auxiliary results}

Recall from the proof outline above that we plan to use blow-ups $R(\cV)$ whose reduced graph $R$ satisfies $\spa$, $\div$ and $\cov$.
We find these blow-ups using the hypergraph analogue of the Kővári--Sós--Turán Theorem due to Erd{\H o}s~\cite{Erd64} in combination with the supersaturation phenomenon discovered by Erd{\H o}s and Simonovits~\cite{ES83}.

For a set $V$ with a partition $\cV$, we say that $e \subset V$ is \emph{$\cV$-partite} if $e$ has at most one element in each part of $\cV$.
An $s$-graph $P$ is called \emph{$\cV$-partite} if all of its edges are $\cV$-partite.
We say that~$P$ is \emph{complete $\cV$-partite} if $\cV$ partitions $V(P)$ and $P$ contains all $\cV$-partite edges.
If we do not wish to specify the underlying partition, we just say that $P$ is \emph{(complete) $r$-partite} to mean that there exists a partition~$\cW$ with $r$ parts such that $P$ is (complete) $\cW$-partite.
We say that $P$ is \emph{balanced} if the underlying partition has parts of uniform size.
We write \gls{complete-partite-graph} for the balanced complete $s$-partite $s$-graph with parts of size $b$.
Given this, we can state the aforementioned result of {Erd{\H o}s~\cite{Erd64}}.

\begin{theorem}\label{thm:erd-original}
	For all $s\geq 2$, $b \geq 1$ and $\gamma > 0$, there is $n_0 > 0$ such that every $s$-graph $P$ on $n\geq n_0$ vertices with $e(P) \geq \gamma n^s$ contains a copy of $K_s^{(s)}(b)$.
\end{theorem}

We use the following supersaturated variant of this result, whose proof follows the exposition of Keevash~\cite{Kee11}.

\begin{corollary}\label{thm:erd64}
		For all $s\geq 2$, $b \geq 1$ and $\gamma > 0$, there are $\beta,n_0>0$ such that every $s$-graph $P$ on $n\geq n_0$ vertices with $e(P) \geq \gamma n^s$ contains at least $\beta n^{sb}$ copies of $K_s^{(s)}(b)$.
\end{corollary}
 
\begin{proof}
	Fix $t$ such that \cref{thm:erd-original} applies with $\gamma/2$, $t$ playing the rôles of $\gamma$, $n$.
	For $n$ large enough, consider an $s$-graph $P$ on $n$ vertices with $e(P) \geq \gamma n^s \geq \gamma \binom{n}{s}$.
	There are at least $(\gamma/2) \binom{n}{t}$ sets $T \subset V(P)$ of $t$ vertices such that $e(P[T]) \geq (\gamma/2) \binom{t}{s}$.
	Otherwise, we would have $\sum_T e(P[T]) < \gamma    \binom{t}{s} \binom{n}{t}$.
	However, we also have $\sum_T e(P[T]) = \binom{n-s}{t-s}  e(P)  \geq \gamma \binom{n-s}{t-s}  \binom{n}{s} = \gamma    \binom{t}{s} \binom{n}{t}$, which gives a contradiction.
	By \cref{thm:erd-original}, each $P[T]$ with $e(P[T]) \geq (\gamma/2) \binom{t}{s}$ contains a copy of $K = K_s^{(s)}(b)$.
	So the number of copies of $K$ in $P$ is at least $(\gamma/2) \binom{n}{t}/ \binom{n-sb}{t-sb} = (\gamma/2) \binom{n}{sb}/ \binom{t}{sb}$.
	We may therefore finish with $\beta = \gamma/    t^{2sb} $.
\end{proof}

Next, we record (without proof) the observation that minimum degrees are monotone.

\begin{observation}\label{fct:monotone-degrees}
	For $0 \leq d \leq d' \leq s$, every $n$-vertex $s$-graph $P$ satisfies 
	\begin{equation*}
		\frac{\delta_d(P) }{\binom{n-d}{s-d}} \geq \frac{\delta_{d'}(P)}{\binom{n-d'}{s-d'}}   \,.
	\end{equation*}
\end{observation}

We  also require a simple bound on the minimum degree that forces a perfect matching to kickstart our machinery, which was first obtained by Daykin and Häggkvist~\cite{DH81}.

\begin{theorem}\label{thm:DH81}
	For every $s\geq 2$ and $\mu > 0$, there is $n_0$ such that every $s$-graph $P$ on $n\geq n_0$ divisible by $s$ vertices and $\delta_1(P) \geq (1-1/s + \mu ) \binom{n-1}{s-1}$ has a perfect matching.
\end{theorem}

This result allows us to find an (almost) perfect matching in the property graphs of \cref{prp:almost-perfect-matching}.
With some additional work, we shall obtain a stronger outcome in this context (\cref{lem:blow-up-matching}).
For the partite setting, Daykin and Häggkvist~\cite{DH81} proved the following analogue:

\begin{theorem}\label{thm:DH81-partite}
	For every $s \geq 2$ and $\mu > 0$, there is $n_0$ such that every $\cV$-partite $s$-graph $P$ with~$\cV$ having $s$-parts of size $n\geq n_0$ each and $\delta_1(P) \geq (1-1/s + \mu ) n^{s-1}$ has a perfect matching.
\end{theorem}

\subsection{Almost perfect matchings}\label{sec:almost-perfect-matching}

Informally, \cref{prp:almost-perfect-matching} transforms a robust (almost) perfect fractional matching into an almost perfect matching.
The usual approach to proving such a result relies on the (Strong Hypergraph) Regularity Lemma, which involves a somewhat technical setup.
A proof without regularity was given by Alon, Frankl, Huang, Rödl, Ruciński and Sudakov~\cite{AFH+12} for the special setting of $k$-graphs under minimum degree conditions.
It is however not clear whether this argument generalises to perfect tilings.
So for instance, we cannot use it to recover classic results such as Koml\'os' theorem (\cref{thm:komlos}) without further ideas.
In what follows, we present a new approach that works in our general setting and also avoids the use of Regularity Lemmas.

Let us begin with a basic step of turning a fractional matching into an integral one.
We write~$\gls{blow-up-b}$ for the blow-up of an $m$-digraph $R$ with parts of uniform size~$b$.
Recall the definition of fissility and ample edges (\cref{def:fissility}).

\begin{observation}\label{obs:fractional-to-integral}
	Let $R$ be an $s$-vertex $m$-digraph with a $\rho$-perfect fractional matching.
	Suppose that~$H$ is a fissile $m$-digraph on $sb$ vertices with $R(b) \subset H$.
	Then $H$ contains a matching covering at least $(1-\rho) sb - ms^m$ vertices.
\end{observation}

\begin{proof} Consider a $\rho$-perfect fractional matching $w \colon R \to [0,1]$ of $R$.
	Since $H$ is fissile, we may pick a matching $\cM$ in $H$ by greedily selecting $\lfloor b \cdot w(f) \rfloor$ many $f$-preserving {injective} edges for every edge $f \in R$.
	Indeed, by fissility there is an $f$-preserving {injective} edge in $H$ for every ample edge $f \in R$.
	Moreover, we may avoid the vertices of already picked edges, since subgraphs of~$R(b)$ are still blow-ups and ampleness is guaranteed by choice of $w$.
	Then $\cM$ covers at least $(1-\rho) sb - ms^m$ vertices, where $ms^m$ accounts for the rounding error at every edge of $R$.
\end{proof}

Our strategy for showing \cref{prp:almost-perfect-matching} is then to first cover the vertices of $H$ with blow-ups~$R(b)$ where $R \in \spa$ and then to conclude by \cref{obs:fractional-to-integral}.
The next lemma allows us to carry out the first part.

\begin{lemma}[Almost perfect blow-up-tiling]\label{lem:covering-with-blow-ups}
	For all $2\leq m \leq s$, $b\geq 1$ and $\eta >0$, there is an $n_0>0$ such that the following holds for every $m$-digraph property $\P$ and $m$-digraph $H$ on $n \geq n_0$ vertices  
	that $(\delta,1,s)$-robustly satisfies $\P$  for $\delta=1-1/s^{1.5}$.
	All but $\eta n$ vertices of $H$ may be covered with pairwise vertex-disjoint blow-ups $R_1(b),\dots,R_\ell(b)$, where each $R_i$ is an $s$-vertex $m$-digraph satisfying $R_i \in \P$.
\end{lemma}

We remark that \cref{lem:covering-with-blow-ups} is interesting in its own right when combined with the inheritance principle (see \cref{cor:almost-perfect-blow-up-tiling}).
Before we come to the proof of \cref{lem:covering-with-blow-ups}, let us derive the main result of this subsection.

\begin{proof}[Proof of \cref{prp:almost-perfect-matching}]
	Introduce $b$ such that $1/s,\,\eta \gg  1/b \gg 1/n$.
	We apply \cref{lem:covering-with-blow-ups} to cover all but $(\eta/2) n$ vertices of $H$ with pairwise vertex-disjoint blow-ups $R_1(b),\dots,R_\ell(b)$, where each $R_i$ is an $s$-vertex $m$-digraph satisfying $R_i \in \SpaF \rho$.
	Next, we use \cref{obs:fractional-to-integral} to find in each of these blow-ups a matching of order at least $(1-\rho) sb- ms^m$.
	Since $\ell ms^m \leq ( ms^m/(bs)) n \leq (\eta/2) n$, this gives a matching that covers all but $(\eta + \rho)n$ vertices of $H$.
\end{proof}

It remains to show \cref{lem:covering-with-blow-ups}.
The proof is based on the following fact, which tiles a uniform hypergraph $P$ almost perfectly with complete partite graphs of constant order, provided that the minimum degree of $P$ forces a perfect matching.

\begin{lemma}\label{lem:blow-up-matching}
	For $\mu,\, 1/b,\, 1/s \gg 1/n$, every $n$-vertex $s$-graph $P$ with $\delta_1(P) \geq \left(1-1/s+\mu \right) \binom{n-1}{s-1}$ contains a $K_s^{(s)}(b)$-tiling missing at most $\mu n$ vertices.
\end{lemma}

Before we come to the proof of \cref{lem:blow-up-matching}, let us show how it implies \cref{lem:covering-with-blow-ups}.
The idea is simple: we first apply \cref{lem:blow-up-matching} to cover most vertices of the property $s$-graph with complete partite graphs.
Then we use \cref{thm:erd64} to cover most vertices of each of the partite graphs with the desired blow-ups.

\begin{proof}[Proof of \cref{lem:covering-with-blow-ups}]
	Introduce $b'$ with $1/s,\, 1/b,\,\eta \gg 1/b' \gg 1/n$ and assume without loss of generality that $1/s \gg \eta$.
	Set $P = \PG{H}{\P}{s}$ as in \cref{def:property-graph}.
	So $P$ is the $s$-graph on $V(H)$ with an $s$-edge~$S$ whenever $H[S]$ satisfies $\P$.
	Moreover, $\delta_1(P) \geq (1-1/s^{1.5}) \binom{n-1}{s-1}$ by assumption.
	We apply \cref{lem:blow-up-matching} with $\eta/2$ playing the rôle of $\mu$ to obtain a $K_s^{(s)}(b')$-tiling in $P$ that covers all but at most $(\eta/2) n$ vertices.
	Fix a copy $K$ of $K_s^{(s)}(b')$ in this tiling.
	Note that each edge of $K$ corresponds to an $s$-vertex $m$-digraph satisfying $\P$.
	To finish, we have to cover all but at most $(\eta/2) b's$ vertices of $K$ with blow-ups $R_1(b),\dots,R_{\ell'}(b)$, where each $R_i$ is an $s$-vertex $m$-digraph with $R_i \in \P$.
	
	We colour each edge $Y$ of $K$ by one of at most $2^{s^{m}}$ colours corresponding to the $m$-digraph~$H[Y]$ labelled by the part indices.
	Applying \cref{thm:erd-original} with $K$, $sb'$ and $s^{-s}2^{-s^{m}}$ playing the rôles of $P$, $n$ and $\gamma$ to the colour with the most edges, we identify a complete $s$-partite subgraph $K' \subset K$ with parts of size $b$ whose edges all correspond to the same $s$-vertex $m$-digraph~$R$ with $R \in \P$.
	In other words, $K'$ is isomorphic to the blow-up $R(b)$.
	We take out the vertices of $K'$ and repeat this process until all but at most $(\eta/2) b's$ vertices of~$K$ are covered, as desired.
\end{proof}

It remains to show \cref{lem:blow-up-matching}.
The existence of the tiling follows by a straightforward application of the Weak Hypergraph Regularity Lemma.
Here we provide an alternative (possibly simpler) proof.
More precisely, we derive \cref{lem:blow-up-matching} by repeatedly applying the following result.

\begin{lemma}\label{lem:larger-matching}
	For $1/s,\, \mu \gg 1/\ell \gg 1/b \gg 1/n$, let $P$ be an $s$-graph on $n$ vertices with $\delta_1(P) \geq \left(1-1/s+ \mu \right) \binom{n-1}{s-1}$.
	Set $B = K_{s}^{(s)}(b)$ and $L= K_{s}^{(s)}(\ell)$.
	Suppose that $P$ contains a $B$-tiling $\cB$ on $\lambda n$ vertices with $0 \leq \lambda \leq 1-\mu/8$.
	Then $P$ contains an $L$-tiling on at least $(\lambda + 2^{-6}(\mu/s)^2 ) n$ vertices.
\end{lemma}

We defer the proof of \cref{lem:larger-matching} for a moment and continue by deriving \cref{lem:blow-up-matching}.
The argument proceeds iteratively.
In each step, we turn a tiling $\cB$ of `big' tiles $B$ into a tiling $\cL$ of `little' tiles $L$ such that $v(\cL) \geq 2^{-6} (\mu/s)^2n+v(\cB)$.
So we arrive at the outcome of \cref{lem:blow-up-matching} after at most $2^{7} (s/\mu)^2$ steps.

\begin{proof}[Proof of \cref{lem:blow-up-matching}]
	Given $s$, $b$ and $\mu$, introduce $b_t,\dots,b_1$ with $t= \lceil 2^{7}/(\mu/s)^{2} \rceil$
	and
	\begin{equation*}
		\mu,\,\tfrac{1}{b} \gg \tfrac{1}{b_t} \gg \dots \gg \tfrac{1}{b_1} \gg \tfrac{1}{n}
	\end{equation*}
	 such that  \cref{lem:larger-matching} can be applied with $b_{i}$ and $b_{i-1}$ playing the rôles of $\ell$ and $b$ for each $2\leq i \leq t$.
	Let $B_i =  K_{s}^{(s)}(b_i)$ and $\lambda_i = (i-1) \cdot 2^{-6}(\mu/s)^2$ for $1 \leq i \leq t$.
	First, suppose that there is $1 \leq {j} \leq t$ such that $P$ contains a $B_{j}$-tiling $\cB_{j}$ on at least $(1-\mu/8) n$ vertices.
	We then tile all but $sb$ vertices of each copy of $B_{j}$ in $\cB_j$ into copies of $B=K_{s}^{(s)}(b)$.
	This gives a $B$-tiling of all but $\mu n/8 + sbn/(sb_{j}) \leq \mu n$ vertices as desired.
	So let us assume for the sake of contradiction that no such $j$ exists.
	
	We show by induction that  $P$ contains a $B_i$-tiling on at least $\lambda_i n$ vertices for every $1 \leq i \leq t$.
	Note that this is trivial for $i=1$ as $\lambda_1 = 0$.
	For $i \geq 2$, let $\cB$ be a $B_{i-1}$-tiling on $\lambda n$ vertices with $  \lambda_{i-1} \leq \lambda \leq   1-\mu/8 $.
	By \cref{lem:larger-matching}, $P$ contains a $B_{i}$-tiling on at least $(\lambda + 2^{-6}(\mu/s)^2 ) n \geq \lambda_i n$ vertices as desired.
	But then we obtain a contradiction to the fact that $\lambda_t > 1$.
\end{proof}

To complete the proof of \cref{lem:blow-up-matching}, it remains to show \cref{lem:larger-matching}.

\begin{proof}[Proof of \cref{lem:larger-matching}]
	For convenience, we assume that $n$ is divisible by $bs$.
	We note that if $\lambda = 0$, then $\cB$ is empty.
	Let $\cU$ be a   partition of $V(P)$ into parts of size $b$, which is obtained by taking the $s$ parts of each copy of $B$ in~$\cB$ and the other parts in an arbitrary fashion.
	Let $r$ be the number of parts in $\cU$, and note that $r = n/b$.
	Set $\gamma = \mu/(8s)$.
	Let $R$ be an $s$-graph with vertex set~$\cU$ and an edge $X$ if $P$ contains at least $4\gamma s b^s$ $X$-partite edges.
	{Informally, $R$ plays the rôle of a reduced graph in the context of a Regularity Lemma.}
	As usual in this setting, the degrees of $P$ are inherited to $R$.
	
	\begin{claim}\label{cla:reduced}
	We have $\delta_1(R) \geq \left( 1-1/s +\mu/4 \right) \binom{r-1}{s-1}$.
	\end{claim}
	
	\begin{proofclaim} 
		Fix a vertex $x$ in $R$.
		There are at most $r^{s-2} b^{s} \leq (\mu/4) b^{s} \binom{r-1}{s-1}$ edges of $P$, which have one vertex in the part of~$\cU$ that contains $x$ and at least two vertices in some part of~$\cU$.
		Put differently, there are at least $b^s \cdot (1-1/s+(3/4)\mu) \binom{r-1}{s-1}$ edges in $P$, which are $\cU$-partite and have a vertex in the part of $x$.
		The parts of every edge of $R$ incident to $x$ can host at most~$b^{s}$ of these edges.
		The parts of every $s$-set $f \subset V(R)$ with $x \in f$, but $f \notin R$ can host at most $4\gamma s b^s = (\mu/2) b^s$ of these edges.
		Thus
		\begin{equation*}
			b^{s} \deg_R(x) + (\mu/2) b^s \binom{r-1}{s-1}  \geq b^s (1-1/s+(3/4)\mu) \binom{r-1}{s-1}\,.
		\end{equation*}
		{Solving for $\deg_R(x)$ yields the desired bound.}
	\end{proofclaim}
	
	It follows by \cref{thm:DH81} that there is a perfect matching $\cM$ in~$R$.
	For each edge $X$ of~$\cM$, we repeatedly use \cref{thm:erd64} to find an $L$-tiling on at least $2\gamma sb$ vertices in the $X$-partite subgraph of $P$ induced by the parts of $X$.
	{(This is possible, since an $L$-tiling of order at most $2\gamma sb$ intersects with at most $2\gamma sb^s$ edges whose vertices are covered by~$X$.)}
	Denote the union of these `fresh' $L$-tilings by $\cL_{\text{fresh}}$.
	So $\cL_{\text{fresh}}$ covers at least $2\gamma (1-\lambda)n \geq 2\gamma^2 n$ vertices outside of $V(\cB)$.
	Moreover, $\cL_{\text{fresh}}$ covers the same number of vertices in each part of $\cU$.
	This allows us to pick a maximal $L$-tiling $\cL_{\text{rec}} \subset P$ in $V(\cB) \sm V(\cL_{\text{fresh}})$ to `recycle' what is left of $\cB$.
	So~$\cL_{\text{rec}}$ leaves at most $\ell-1$ vertices uncovered in each of the parts of $\cB$.
	Since $\cB$ is a $B$-tiling on $\lambda n$ vertices with $B = K_{s}^{(s)}(b)$, there are $\lambda n/b$ of these parts.
	Hence the tiling $\cL_{\text{rec}}$ covers all but $(\ell-1)\lambda n/b \leq \gamma^2 n$ vertices of $V(\cB) \sm V(\cL_{\text{fresh}})$.
	It follows that $\cL_{\text{fresh}} \cup \cL_{\text{rec}}$ covers at least $2\gamma^2 n + \lambda n - \gamma^2 n \geq (\lambda + \gamma^2) n$ vertices.
\end{proof}

We remark that the bounds on the order $n$ in the above argument are quite large (tower-type).
A more efficient approach is discussed in \cref{sec:conclusion}.

We conclude this subsection by recording a partite version of \cref{prp:almost-perfect-matching}:

\begin{proposition}[Almost perfect partite matching]\label{prp:almost-perfect-matching-partite}
	For all $2\leq m \leq s$ and $\mu,\eta >0$ and $\rho \geq 0$ there is an $n_0>0$ such that the following holds.
	Let $H$ be a fissile $m$-digraph with a partition~$\cV$ of $V(H)$ into $s$ parts of uniform size $n \geq n_0$.
	Suppose that $P \subset \PG{H}{\SpaF{\rho}}{s}$ is a $\cV$-partite subgraph on $V(H)$ such that
	$\delta_1(P) \geq   \left(1-1/s^{1.5} \right) n^{s-1}.$
	Then $H$ has a matching covering all but at most $(\eta + \rho) sn$ vertices.
\end{proposition}

The proof of \cref{prp:almost-perfect-matching-partite} follows \emph{mutatis mutandis}\footnote{‘Changing what must be changed’, that is, with the obvious modifications to the new setting.} along the lines of the proof of \cref{prp:almost-perfect-matching}.
The main difference is that \cref{thm:DH81-partite} is used instead of \cref{thm:DH81} in the proof of the partite analogue of \cref{lem:larger-matching}.
We omit the details.

\subsection{Lattice completeness}\label{sec:lattice-completeness}

In practice, it is helpful to reformulate the divisibility property in terms of lattice completeness, which was used in the work of Keevash and Mycroft~\cite{KM15} on perfect matchings with origins in the design setting~\cite{GJ73,Wil73}.

For a set $V$ and a tuple $e \in V^m$, we denote by $\gls{indicator-vector} \in \INTS^V$ the \emph{indicator vector}, which takes value $\multi(v,e)$ at index $v \in V$.
({When writing $\INTS^V$, we implicitly fix an ordering of $V$. But this does not matter for the arguments.})
We set $\vn_v = \vn_e$ for $e=(v)$.
The \emph{lattice} of an $m$-digraph $H$ is the additive subgroup $\gls{lattice} \subset \INTS^{V(H)}$ generated by the vectors $\vn_e$ with $e \in H$.
We say that $\cL(H)$ is \emph{complete} if it contains all vectors $\vecb b$ with $\sum_{v \in V(H)} \vecb b (v)$ divisible by $m$, meaning that there are   $c_e \in \INTS$ such that $\vecb b = \sum_{e \in H} c_e \vn_e$.
Our proof requires a simple bound on the values of~$|c_e|$.
Using an old result of Steinitz~\cite{steinitz1913bedingt}, one can show that $|c_e| \leq 2 ||\vecb b||_1(2n)^{n}$ where $n=v(H)$ and  $||\cdot ||_1$ is the $L_1$-norm.
\COMMENT{See also the expository paper of Bárány~\cite{barany2008power}.}
However, for our purposes it suffices that $|c_e|$ is trivially bounded by some function of $n$ and $\vecb b$.

To connect lattice completeness with the divisibility property, note that a perfect integral matching $w$ in an $m$-digraph $H$ can be viewed as a vector $\vecb b = \sum_{e \in H} w(e) \vn_e \in \cL(H)$.
Hence lattice completeness implies the divisibility property.
It turns out that the converse is also true under slightly stronger assumptions:

\begin{lemma}\label{obs:robust-int-matching-gives-lattices-completeness}
	Let $H$ be an $m$-digraph on $n\geq m+1$ divisible by $m$ vertices such that $H - D \in \div$ for every $m$-set $D \subset V(H)$.
	Then $\cL(H)$ is complete. 
\end{lemma}

To show \cref{obs:robust-int-matching-gives-lattices-completeness}, we need to shift weight between vertices.
Formally, a \emph{transferral} is a vector \gls{transferral} that moves weight $1$ from a vertex $u$ to a vertex $v$ in $H$.
We can characterise lattice completeness as follows:

\begin{observation}\label{obs:lattice-completeness}
	The lattice $\cL(H)$ of an $m$-digraph $H$ is complete if and only if $\cL(H)$ contains $\vn_v - \vn_u$ for every $u,v \in V(H)$.
\end{observation}

\begin{proof}
	Clearly, a complete lattice contains all transferrals.
	So we focus on the other direction.
	Suppose that $\cL(H)$ contains all transferrals.
	Consider $\vecb b \in  \INTS^{V(H)}$ with $\sum_{v \in V(H)} \vecb b (v) = \ell \cdot m$ for $\ell \in \INTS$.
	We have to show that $\vecb b  \in \cL(H)$.
	Fix an edge $e \in H$, which is possible as $\cL(H)$ is not trivial.
	In particular, $\ell\cdot \vn_e \in \cL(H)$.
	Set $\vecb b' = \vecb b - \ell\cdot \vn_e $, which means $\sum_{v \in V(H)} \vecb b' (v) = 0$.
	If~$\vecb b'$ is all-zeroes, we are done as $\cL(H)$ contains this vector.
	Otherwise, there are $u,v \in V(H)$ with $\vecb b'(u) > 0$ and $\vecb b'(v) < 0$.
	We update $\vecb b'$ to $\vecb b' + \vn_v - \vn_u$, and iterate this process until~$\vecb b'$ is all-zeroes.
	Note that this process does terminate, since the sum of absolute values $||\vecb b'||_1$ decreases in each step.
	It follows that $\vecb b - \ell \cdot \vn_e \in \cL(H)$, and hence $\vecb b  \in \cL(H)$.
\end{proof}

We finish by deriving the original claim.

\begin{proof}[Proof of \cref{obs:robust-int-matching-gives-lattices-completeness}]
	Given arbitrary $u,v \in V(H)$, let $D' \subset V(H)-u-v$ be a set of $m-1$ vertices.
	Let $D_u = D' \cup \{u\}$ and $D_v = D' \cup \{v\}$.
	By assumption there is a perfect integral matching in $H - D_v$, which corresponds to a vector $\vecb b_u \in  \cL(H)$.
	Similarly, let $\vecb b_v \in  \cL(H)$ correspond to a perfect integral matching in $H - D_u$.
	Thus $\cL(H)$ contains the transferral $\vn_v - \vn_u = \vecb b_v - \vecb b_u$.
	It follows that $\cL(H)$ is complete by \cref{obs:lattice-completeness}.
\end{proof}

\subsection{Boosting}\label{sec:boosting}

It is often convenient to work with properties that survive the deletion of a few vertices.
To formalise this, we need to take into account a few trivial exceptions.
We say that a digraph family $\P$ is \emph{$(m,n)$-trivial}, if $\P$ contains no $m$-digraphs of order~$n$.
For example, $\mat$ and $\div$ are $(m,n)$-trivial whenever $m$ does not divide $n$.\textbf{}
On the other hand, $\spa$ and $\cov$ are not $(m,n)$-trivial for all $m$ and $n$.
Denote by \gls{deletion-robustness} the set of $n$-vertex $m$-digraphs $H$ such that $H-D \in \P$ for every set $D \subset V(H)$ of $d' \leq d$ vertices such that $\P$ is not $(m,n-d')$-trivial.
For instance, $H$ satisfies $\Del_m(\div)$ in the context of \cref{obs:robust-int-matching-gives-lattices-completeness}.
As it turns out, one can harden the robustness conditions of \cref{thm:framework,thm:framework-exceptional} against vertex deletion, while simultaneously boosting its degree.
Recall that for $\delta \in [0,1]$, we denote by $\DegF{d}{\delta}$ the union of all $k$-graphs with $k>d$ and relative minimum $d$-degree at least $\delta$.

\begin{lemma}[Booster]\label{lem:booster}
	For $1/m,\, 1/d,\, 1/r, \,1/p \gg 1/s \gg 1/n$ with $r$ divisible by $m$ and $\rho \in [0,1]$, let $H$ be an $m$-digraph on $n$ vertices that $r$-robustly satisfies $\P \in \{\spa_\rho,\,\div,\,\cov\}$.
	Then $H$ $(\delta,p,s)$-robustly satisfies $\Del_{d}(\P)$ for $\delta = 1-e^{-\sqrt{s}}$.
\end{lemma}

\begin{proof}
	Let $\delta' = 1-1/r$ and $\mu = 1/r-1/r^2 >0$ and $P = \PG{H}{\P}{r}$ as in \cref{def:property-graph}.
	So~$P$ is the $r$-graph on $V(H)$ with an $r$-edge $R$ whenever $H[R]$ satisfies $\P$.
	By assumption, every vertex in~$P$ has degree at least $(\delta'+\mu) \tbinom{n-1}{r-1}$.
	Now set $Q = \PG{P}{\DegF{1}{\delta'+\mu/2}}{s}$.
	So $Q$ is the $s$-graph on $V(P)=V(H)$ with an $s$-edge $S$ whenever the $r$-graph $P[S]$ has minimum vertex degree at least $\left(\delta'+\mu/2\right) \binom{s-1}{r-1}$.
	By \cref{lem:inheritance-minimum-degree} applied with $r,1,p$ playing the rôle of $k,d,q$, it follows that $\delta_{p}(Q) \geq  \delta  \tbinom{n-p}{s-p}$.
	Let $P' = \PG{H}{\Del_{d}\left(\P\right)}{s}$.
	So $H$ satisfies $(\delta,p,s)$-robustly $\Del_{d}(\P)$, if $\delta_p  (P' ) \geq \delta_{p}(Q)$.
	To prove this inequality, fix an $s$-edge $S' \in Q$.
	Let $D \subset S'$ be an arbitrary set of $d' \leq d$ vertices.
	Write $S = S' \sm D$.
	We may assume that $\P$ is not $(m,|S|)$-trivial, since otherwise there is nothing to show.
	(This only concerns the case $\P = \div$.)
	We have to show that $H[S]$ satisfies $\P$.
	
	Recall that $\P \in \{\spa_\rho,\,\div,\,\cov\}$.
	We begin with the space property.
	Note that the $r$-graph~$P[S]$ has a perfect fractional matching by \cref{thm:DH81}.
	This is evident if $|S|$ is divisible by $r$.
	Otherwise, we may consider a perfect fractional matching of $P[S]-T$ for every subset $T\subset S$ such that $|T| \leq r-1$ and $r$ divides $|S|-|T|$.
	Taking the normalised sum of these matchings results in the desired perfect fractional matching of $P[S]$.
	By definition of $P = \PG{H}{\spa_\rho}{r}$, the $m$-graph $H[R]$ has a $\rho$-perfect fractional matching for every $R \in P[S]$.
	Linearly combining this shows that $H[S]$ satisfies $\spa_\rho$ as desired.
	
	Next, we consider the divisibility property.
	Note that $|S|$ is divisible by $m$, since $\div$ is $(m,|S|)$-trivial otherwise.
	To find a perfect integral matching in $H[S]$, it suffices to show that the lattice $\cL=\cL(H[S])$ is complete as discussed in \cref{sec:lattice-completeness}.
	Fix vertices $x, y\in S$.
	By \cref{obs:lattice-completeness} we may conclude by showing that $\cL$ contains $\vn_x - \vn_y$.
	Since $P[S]$ has minimum vertex degree at least $\left(1/2+\mu/2\right) \binom{s-1}{r-1}$, it follows that there is an $(r-1)$-set $Z \subset S$ such that $X = \{x\} \cup Z$ and $Y = \{y \} \cup Z$ are edges in $P[S]$.
	As before, the $r$-graphs~$P[X]$ and $P[Y]$ have perfect integral matchings $\cM_x$ and $\cM_y$ by \cref{thm:DH81}.
	Fix an edge $M \in \cM_x \cup \cM_y$.
	By definition of $P = \PG{H}{{\div}}{r}$, the $m$-graph $H[M]$ has a perfect integral matching.
	Combining these matchings for the edges of $\cM_x$ and $\cM_y$, this gives perfect integral matchings of $H[X]$ and $H[Y]$, respectively.
	We may then take the difference between the two to generate the transferral $\vn_x - \vn_y$  in the lattice $\cL$, as desired.
	
	Finally, the cover property follows simply because every vertex $v \in S$ is on an edge $R \in P[S]$, and $H[R] \in \cov$ if $\P = \cov$.
	So $H[S]$ satisfies $\cov$ in this case, and we are done.
\end{proof}

\subsection{Absorption}\label{sec:absorption}

In this section, we show \rf{prp:absorption}.
We construct the `absorption structure' $A$ by combining many of the following local gadgets:

\begin{definition}[Absorber]\label{def:absorber}
	Let $H$ be an $m$-digraph, and let $X \subset V(H)$ be an $m$-set.
	An \emph{$X$-absorber} is the union of two matchings $\cM_1,\,\cM_2$ in $H$ such that $V(\cM_1) = V(\cM_2) \cup X$ and $V(\cM_2) \cap X = \es$.
	Its \emph{order} is $v(\cM_2)$.
\end{definition}

The following lemma tells us that after boosting the assumptions of \cref{prp:absorption} with \cref{lem:booster}, every $m$-set $X \subset V(H)$ has many $X$-absorbers.

\begin{lemma}[Absorber Lemma]\label{lem:absorber}
	For $m\geq 2$, $s \geq 2m+1$ and $1/s,\,\mu \gg 1/q \gg \beta \gg 1/n$ with~$q$ divisible by $m$, let $H$ be a fissile $n$-vertex $m$-digraph and $P =\PG{H}{\Del_{2m}(  \div_{}^{}  \cap \cov_{}^{})}{s}$.
	Suppose that an $m$-set $X \subset V(H)$ satisfies $\deg_P(X) \geq \mu n^{s-m}$.
	Then $H$ has at least $\beta n^{q}$ $X$-absorbers of order~$q$.
\end{lemma}

Given \cref{lem:absorber}, we can derive \cref{prp:absorption} by constructing the desired set $A$ in a semi-random process.
For the sake of completeness, we spell out the details of this routine argument.

\begin{proof}[Proof of \cref{prp:absorption}]
	For $2\leq m \leq s$ and $\alpha >0$, introduce $r,\, q,\,  \beta,\, \eta,\, n$ such that
	\[\tfrac{1}{m},\,\tfrac{1}{s},\, \alpha\gg \tfrac{1}{r} \gg \tfrac{1}{q} \gg \beta \gg \eta \gg \tfrac{1}{n}\]
	with $s$ divisible by $m$ and $q$ divisible by $m$.
	Let $H$ be a fissile $m$-digraph  on $n \geq n_0$ vertices with $\delta_1 \big(\PG{H}{\div \cap \cov}{s} \big) \geq   (1-1/s^2 )  \tbinom{n-1}{s-1}$ as in \cref{def:robustness-detailed}.
	Write $\P = \Del_{2m}(  \div  \cap \cov)$.
	By \cref{lem:booster} applied with $2m,s,m,r$ playing the rôles of $d,r,p,s$ (once with $\P=\div$ and once with $\P=\cov$), we have ${\delta_m(\PG{H}{\P}{r}) \geq \mu n^{r-m}}$ for $\mu =  {1}/{(2(r-m)!)}$.
	It follows by \cref{lem:absorber} that $H$ has at least $\beta n^{q}$ $X$-absorbers of order $q$ for every $m$-set $X \subset V(H)$.
	
	To construct the set $A$, set $p= (4\eta /  \beta) n^{-q+1}$.
	Let ${Q}$ be a random set containing each $q$-tuple in ${V(H)}^q$ independently with probability $p$.
	Since $\Exp[|{Q}|] = p n^{q} =  (4\eta/\beta)  n \leq (\alpha/(2q)) n$, Markov's inequality gives	$\Pr \left(|{Q}| > ({\alpha}/{q}) n \right) \leq  {1}/{2}.$
	
	We say that two distinct $q$-tuples \emph{overlap} if there is a vertex occurring in both.
	Note that there are at most $q^2 n^{2q-1}$ pairs of overlapping $q$-tuples.
	Let $W$ be the random variable that contains all overlapping pairs whose tuples are both in ${Q}$.
	We have $\Exp[|W|] \leq q^2 n^{2q-1} p^2 = q^2 (4\eta / \beta)^2 n$.
	Thus a further application of Markov's inequality reveals 
	$\Pr(|W| > \eta n) \leq  ({q^2 }/{\eta}) \left( {4\eta}/{\beta}\right)^2 \leq  {1}/{4}.$
	
	For an $m$-set $X \subset V(H)$, let ${Q}_X$ be the set of $X$-absorbers that are contained in ${Q}$.
	It follows that $|{Q}_X|$ is binomially distributed. 
	By the above, we have $\Exp[|{Q_X}|] \geq  p \beta n^{q} = 4\eta n$.
	Using {Chernoff's bound},
	we deduce that
	$\Pr(|{Q}_X| \leq 3 \eta n) \leq   {1}/{(8 n^m)} .$
	
	In combination, the above bounds on the probabilities imply that there is a (deterministic) outcome $Q$ such that
	\begin{itemize}[-]
		\item the tuples of $Q$ contain in total at most $\alpha n$ vertices,
		\item $Q$ has at most $\eta n$ overlapping pairs,
		\item $Q$ contains at least $3\eta n$ $X$-absorbers for every $m$-set $X \subset V(H)$.
	\end{itemize}
	
	Fix such a set $Q$.
	Note that some tuples of $Q$ might overlap, while others might not correspond to any absorber.
	We adjust this as follows.
	First, we delete one tuple from every overlapping pair in $Q$.
	Since we remove at most $\eta n$ tuples this way, it follows that~$H[A]$ still contains at least~$\eta n$ pairwise disjoint $X$-absorbers for every $m$-set $X \subset V(H)$.
	Next, we only keep those tuples of~$Q$ that actually correspond to an $X$-absorber for some $m$-set $X \subset V(H)$.
	Let $A \subset V(H)$ consist of the union of the vertices in the tuples of~$Q$ after these deletions.
	
	We claim that $A$ has the desired absorbing properties.
	To see this, consider a set $L \subset V(H)$ of size at most $\eta n$ and divisible by $m$.
	By the above, we may partition $L$ arbitrarily into $m$-sets and greedily match each such set $X$ with an $X$-absorber.
	We then obtain a perfect matching of $H[A\cup L]$ by covering $L$ with the active states ($\cM_1$ in \cref{def:absorber}) of the corresponding absorbers and covering the rest of $A$ using the passive states ($\cM_2$ in \cref{def:absorber}) of the corresponding absorbers.
\end{proof}

It remains to show \cref{lem:absorber}.
The idea of the proof is to find a suitable blow-up $R(b) \subset H$, which covers $X$ and such that $R$ satisfies $\div$ and $\cov$.
We then use the divisibility and cover property together with the fact that $H$ is fissile to construct the desired absorber in $R(b)$.
Unfortunately, such a cover might not exist, since the vertices in $X$ may exhibit an exceptional position in $H$.
We address this by finding a rooted blow-up $R(b;X) \subset H$ in which each vertex of~ $X$ is covered by a singleton cluster and all other clusters are of size $b$.
In addition, we ensure that $R$ satisfies the property $\Del_{2m}(  \div_{}^{}  \cap \cov_{}^{})$, which allows us to find the desired absorber for~$X$ using the structure of~$R$.

Let us formalise this discussion as follows.
For an $m$-digraph $R$ and a set $X$ disjoint from $V(R)$, we denote by \gls{rooted-blow-up} the blow-up $R(\cV)$ where $\cV$ consists of the singletons $\{x\}$ for $x \in X$ and additional $v(R)-|X|$ disjoint sets each of size~$b$.
We say that an $m$-digraph $H$ with $X \subset V(H)$ contains a \emph{rooted} copy of~$R(b;X)$ if there is an injective homomorphism from~$R(b;X)$ to $H$, which is the identity on $X$.

\begin{lemma}[Local blow-ups]\label{lem:local-blow-ups}
	For $2 \leq m \leq s$, $0 \leq \ell \leq s$ and $\mu,\, 1/b,\, 1/s \gg \beta \gg 1/n$, let~$H$ be an $n$-vertex $m$-digraph with an $\ell$-set $X \subset V(H)$.
	Let $\P$ be a family of $m$-digraphs and $P \subset \PG{H}{\P}{s}$.
	Suppose that $\deg_P(X) \geq \mu n^{s-\ell}$.
	Then there is an $s$-vertex $m$-digraph $R \in  {\P}$ such that $H$ contains at least ${\beta} n^{(s-\ell)b}$ rooted copies of $R(b;X)$.
\end{lemma}

\begin{proof}
	Introduce $b'$ with $1/b,\,1/s \gg   1/b' \gg \beta$, and set $r=s-\ell$.
	The number of $m$-digraphs on a fixed set of $s$ vertices is $2^{s^m}$.
	So averaging reveals an $s$-vertex $m$-digraph $R \in \P$ such that the $s$-graph $P$ has at least $\mu 2^{-s^m} n^{r}$ edges $S \cup X$ such that $H[S \cup X]$ is a rooted copy of $R(1;X)$.
	Let~$L$ be the $r$-graph on $V(P-X)$, whose edges consist of the $r$-sets $S$ supporting these copies.
	By \cref{thm:erd64} applied with $\mu 2^{-s^m}$ and~$\sqrt{\beta}$ playing the rôle of $\gamma$ and $\beta$, it follows that $L$ contains $\sqrt{\beta} n^{rb'}$ copies of the complete $r$-partite $r$-graph whose parts have size $b'$.
	Let $K$ be one of these copies.
	It suffices to show that $K$ contains a rooted copy of $R(b;X)$, since we can then double count to find that $H$ contains at least
	\begin{equation*}
		\sqrt{\beta} n^{rb'} \tbinom{n-rb}{rb'-rb}^{-1} \geq \sqrt{\beta}n^{rb'} \big(\tfrac{1}{e}\tfrac{rb'-rb}{n-rb}\big)^{rb'-rb}
		\geq \sqrt{\beta} n^{rb} \big(\tfrac{rb'-rb}{e}\big)^{rb'-rb}
		\geq {\beta} n^{rb}
	\end{equation*}
	rooted copies of~$R(b;X)$.
	
	It remains to show that $K$ contains a rooted copy of $R(b;X)$.
	We colour each $r$-edge $S$ of~$K$ by one of $s!$ colours, corresponding to the $s!$ possibilities in which the vertices of a rooted copy of $R(1;X)$ are mapped to $X$ and the parts of $K$.
	Applying \cref{thm:erd-original} to the colour with most edges, we obtain a complete $r$-partite subgraph $K' \subset K$ with parts of size $b$ whose edges correspond to rooted copies $R(1;X)$ that have their respective vertices in the same parts of $K$.
	Let us identify the parts of $K'$ with the vertices of $R$ and consider an edge $(V_1,\dots,V_m) \in R$.
	So~$H$ contains every edge $(v_1,\dots,v_m) \in V_1 \times \dots \times V_m$ with $v_i = v_j$ whenever $V_i = V_j$.
	In other words, $K'$ together with $X$ forms the desired rooted copy of $R(b;X)$ in $K$.
\end{proof}

It remains to show \cref{lem:absorber}.
Thanks to the setup of \cref{lem:local-blow-ups} the argument is equivalent to finding a suitable allocation in the context of a Blow-up Lemma.
Recall the definition of lattices and their completeness from \cref{sec:lattice-completeness}.

\begin{proof}[Proof of \cref{lem:absorber}]
	Introduce $q',\,b$ with $1/s,\,\mu \gg 1/q' \gg 1/q \gg 1/b \gg \beta$.
	In particular, we choose $q'$ large enough such that for every $(s-m)$-vertex $m$-digraph $Q$ and $\vecb b \in \cL(Q)$ with $||\vecb b||_1 \leq m^2$, we can write $\vecb b = \sum_{e \in Q} c_e \vn_e$ with $|c_e| \leq q'$.
	(Since there is a finite number of all the involved objects, we can take $q'$ simply to be the maximum of $|c_e|$.) 
	By \cref{lem:local-blow-ups} applied with~$m$, $\sqrt{\beta}$ and $\Del_{2m} \left( {\div}  \cap  {\cov}\right)$ playing the rôles of $m$, $\beta$ and $\P$, there is an  $s$-vertex $m$-digraph  $R \in \Del_{2m} \left( {\div}  \cap  {\cov}\right)$ such that $H$ contains at least $\sqrt{\beta} n^{(s-{m})b}$ rooted copies of $R(b;X)$.
	Let~$R^\ast$ be one of these copies.
	We will show that~$R^\ast$ contains an $X$-absorber of order~$q$.
	From this it follows by double counting (as in the proof of \cref{lem:local-blow-ups}) that $H$ contains at least $\sqrt{\beta} n^{(s-{m})b} / \binom{n - q}{(s-m) b -q}  \geq \beta n^{q}$ of the desired $X$-absorbers.
	
	Recall that an $X$-absorber in $R^\ast$ consists of two matchings $\cM_1,\cM_2 \subset R^\ast$ such that $V(\cM_1) = X \cup V(\cM_2)$ and $X \cap V(\cM_2) = \es$.
	We begin by constructing $\cM_1$ as follows.
	Select a matching $\cM_{\cover} \subset R^\ast$ on $m$ edges that covers~$X$.
	This is possible, since $H$ is fissile, $R \in \Del_{2m} \left(   {\cov}\right)$ and~ $b$ is large enough.
	Next, we pick a matching $\cM_{\res} \subset R^\ast - V(\cM_{\cover})$ that acts as a `reservoir' by selecting $q'$ disjoint $e$-preserving ({injective}) edges for every $e \in R$.
	Set $\cM_1 = \cM_{\cover} \cup \cM_{\res}.$
	Since $e(R) \leq s^m$, we have $v(\cM_1) \leq {m}^2 + mq'  s^m  \leq q+{m} \leq~b$.
	(This calculation also justifies why $\cM_{\res}$ can be selected in the first place.)
	We add some additional {injective} edges to $\cM_{\res}$ until $v(\cM_1)=q+{m}$, keeping the names for convenience.
	This concludes the construction of~$\cM_1$.
	
	It remains to construct a second matching $\cM_2 \subset R^\ast$ with $V(\cM_1) = X \cup V(\cM_2)$ and $X \cap V(\cM_2) = \es$.
	By definition, the vertex set of $\cM_1 - X$ consists of the vertices of~$\cM_{\res}$ and $\cM_{\cover}-X$.
	Our plan is to remove a special matching $\cM_{\ominus} \subset \cM_{\res}$ from the reservoir such that the vertices of $\cM_{\ominus} \cup (\cM_{\cover}-X)$ are spanned by another matching $\cM_{\oplus} \subset R^\ast$.
	Setting $\cM_2 = \cM_{\oplus} \cup (\cM_{\res} \sm \cM_{\ominus})$ gives the desired $X$-absorber.
	
	To formalise this, let $\mathbf {b} \in \INTS^{V(R)}$ count the number of vertices of $\cM_{\cover}-X$ in each cluster of $R^\ast$.
	By the definition of $\cM_{\cover}$, we have  $||\vecb b||_1 = m(m-1)$.
	Since  $R \in  \Del_{2m} \left( {\div} \right)$, we have $R-X \in \Del_{m} \left( {\div} \right)$.
	Hence \cref{obs:robust-int-matching-gives-lattices-completeness} tells us that the lattice $\cL(R-X)$ is complete.
	Moreover, $\sum_{v \in  V(R-X)} \mathbf {b}(v) = m(m-1)$ is divisible by $m$.
	Hence $\mathbf {b}\in \cL(R-X)$.
	By the definition of $q'$, there are integers $|c_e| \leq q'$ for every $e \in R-X$ such that $\mathbf {b} = \sum_{e \in R-X} c_e \vn_e$.
	Let $E_{\oplus}$ and $E_{\ominus}$ denote the sets of edges $e \in R-X$ for which~$c_e$ is positive and negative, respectively.
	We pick $\cM_{\ominus} \subset \cM_{\res}$ by selecting $|c_e|$ many $e$-preserving (injective) edges for each $e \in E_{\ominus}$.
	Note that
	\begin{equation*}
		\sum_{e \in E_{\oplus}} c_e \vn_e  = \mathbf {b} - \sum_{e \in E_{\ominus}} c_e \vn_e = \mathbf {b} + \sum_{e \in E_{\ominus}} |c_e|  \vn_e\,.
	\end{equation*}
	Owing to the fissility of $H$ and the fact that $R^\ast \subset H$ is a blow-up, we may then (greedily) pick a matching $\cM_{\oplus} \subset R^\ast$ on the vertices of $\cM_{\ominus} \cup \cM_{\cover}-X$ by adding~$c_e$ many $e$-preserving {(injective)} edges for every  $e \in   E_{\oplus}$.
	Setting $\cM_2 = \cM_{\oplus} \cup (\cM_{\res} \sm \cM_{\ominus})$, we obtain that 
	\begin{equation*}
		V(\cM_2) = V(\cM_{\oplus}) \cup V(\cM_{\res} \sm \cM_{\ominus}) = V(\cM_{\res}) \cup V(\cM_{\cover}-X) = V(\cM_1-X)
	\end{equation*}
	as desired.
\end{proof}

We conclude this section with a partite version of \cref{prp:absorption}.
Let $V$ be a set with a partition $\cU$.
We say that a set $S \subset V$ is \emph{$\cU$-proportional} if $|U \cap S|/|S| = |U|/|V|$ for every $U \in \cU$.
Analogously, a tuple $e \in V^m$ is \emph{$\cU$-proportional} if $\sum_{v \in U} \multi(v,e) / m = |U \cap V|/|V|$ for every $U \in \cU$.
An $m$-digraph $H$ is \emph{$\cU$-proportional} if $V(H) \subset V$ is $\cU$-proportional and all edges of~$H$ are $\cU$-proportional.
In this case, we can fix a representative edge $e \in H$, and denote for each $U \in \cU$, by  $\multi(U)$ the number of vertices that $e$ has in $U$.
The lattice of $H$ is then \emph{$\cU$-complete}, if it contains every $\mathbf {b} \in \INTS^{V(H)}$ for which there is $p \in \INTS$, such that for each $U \in \cU$, we have $\sum_{v \in U} \mathbf {b}(v) = p \cdot \multi(U)$.
We remark that $H$ is complete in the non-partite sense if $\cU = \{V(H)\}$.
The following is the partite analogue of \cref{obs:lattice-completeness}.

\begin{observation}\label{obs:lattice-completeness-partite}
	Let $\cU$ be a partition of a set $V$.
	Let $H$ be a $\cU$-proportional $m$-digraph.
	If $\cL(H)$ contains the {transferral} $\vn_v - \vn_u$ for all choices of $U \in \cU$ and $u,v \in V(H) \cap U$, then $\cL(H)$ is $\cU$-complete. 
\end{observation}

\begin{proof}
	Fix a `representative' vertex $w_U \in U$ for each $U \in \cU$.
	Consider  $\mathbf {b} \in \INTS^{V(H)}$ for which there is $p \in \INTS$, such that for each $U \in \cU$, we have $\sum_{v \in U} \mathbf {b}(v) = p \cdot \multi(U)$.
	Informally, we assemble $\vecb b$ from transferrals as follows.
	Fix an edge $e \in H$, and note that $\multi(U) = \sum_{v \in U} \multi(v,e)$.
	We begin with $p$ copies of $e$ and transfer their  weight to $w_U$ with each $U \in \cU$.
	Then, we distribute the weight of each~$w_U$ within $U$ according to the demands of $\vecb b$.
	Formally, 
	\begin{equation*}
		p  \vn_e +   p \sum_{U \in \cU} \sum_{v\in U}  \multi(v,e) (\vn_{w_U} - \vn_{v}) + \sum_{U \in \cU} \sum_{v \in U} \mathbf {b}(v) (\vn_v - \vn_{w_U}) = \mathbf {b}\,.
	\end{equation*}
	Since $\cL(H)$ contains all the involved transferrals, this implies $\mathbf {b} \in \cL(H)$.
\end{proof} 

Motivated by \cref{obs:robust-int-matching-gives-lattices-completeness}, we define the partite divisibility property directly in terms of lattice completeness.
More precisely, given a partition $\cU$ of a set $V$, we denote by \gls{div-partite} the set of all $\cU$-proportional uniform digraphs $H$ with $\cU$-complete lattice.
Now we are ready to state a partite version of \cref{prp:absorption}.

\begin{lemma}[Partite absorption]\label{prp:absorption-partite}
	For all $2\leq m \leq s$ and $\mu,\alpha >0$, there are $\eta,n_0 >0$ with the following property.
	Let $H$ be a fissile $\cU$-proportional $m$-digraph on $n \geq  n_0$ vertices, where $\cU$ is a partition of $V(H)$.
	Let $P \subset \PG{H}{\Del_m (  \pdiv_{}^{}(\cU) \cap \cov_{}^{})}{s}$ be the subgraph of $\cU$-proportional edges.
	Suppose that every $\cU$-proportional $m$-set $X \subset V(H)$ satisfies $\deg_P(X) \geq   \mu  n^{s-m}.$
	Then there is a set $A \subset V(H)$ with $|A| \leq \alpha n$ such that $H[A \cup L]$ has a perfect matching for every $\cU$-proportional set $L \subset V(H-A)$  with $|L| \leq \eta n$ divisible by $m$.
\end{lemma}

Note that the input conditions of \cref{prp:absorption-partite} correspond to the input of \cref{prp:absorption} after boosting them with \cref{lem:booster}.
It is stated in this way to avoid the inconvenience of defining (and bounding) minimum degree conditions for perfect matchings in the proportional setting.
The proof of \cref{prp:absorption-partite} follows mutatis mutandis along the lines of the proof of \cref{prp:absorption}.
The main difference is that we apply \cref{lem:local-blow-ups} with $\P = \Del_m(\pdiv_{}^{}(\cU)  \cap \cov_{}^{})$ and $P \subset \PG{H}{\P}{s}$ containing only the $\cU$-partite edges.
One then concludes using \cref{obs:lattice-completeness-partite} in place of \cref{obs:lattice-completeness}.
We omit the details.

\section{Applications}\label{sec:applications}

In what follows, we illustrate the use of \cref{thm:framework,thm:framework-exceptional} with applications for minimum degree and uniform density conditions in graphs and hypergraphs.
We also touch on variations of these problems in ordered and transversal settings.

\subsection{Dirac-type results for hypergraphs} \label{sec:dirac-hyper}

A classic theorem of Dirac determines the optimal minimum degree conditions that force a perfect matching in a graph.
The generalisation of this problem to perfect tilings was studied by Hajnal and Szemerédi~\cite{HS70}, Komlós~\cite{Kom00} and Kühn and Osthus~\cite{KO09} among many others. 
We refer the reader to the extensive survey of Simonovits and Szemerédi~\cite{SS19} for a more detailed history of the subject.

We discuss the results of this area in terms of the {minimum degree thresholds} $\th_d^{}(\til_F)$ defined in \cref{sec:introduction}.
Let us introduce two further pieces of notation.
An \emph{$\ell$-colouring} $\phi\colon V(F)  \to  [\ell]$ of a $k$-graph $F$ is \emph{proper} if no edge contains more than one vertex of the same colour.\footnote{Note that this is different from the classic definition of proper colourings in hypergraphs.}
The \emph{chromatic number} \gls{chromatic} is the least integer $\ell$ for which $F$ has a proper $\ell$-colouring.
(In the terminology of \cref{sec:preliminaries}, $F$ is $\cV$-partite if the parts of $\cV$ are the colour classes of a proper colouring.
Moreover, $\chi(F) \leq \ell$ if and only if $F$ is $\ell$-partite.)
We denote by \gls{tau-F} the relative size of a smallest colour class over all proper $\chi(F)$-colourings of $F$.
Thresholds for graphs are typically expressed in terms of the \emph{critical chromatic number}, which is defined by $\gls{critical-chromatic} = \tfrac{\chi(F)-1}{1-\tau(F)}$ for $\chi(F) \geq 2$ and $\chi_{\crit}(F) = 1$ otherwise.
Given this, we can state the following classic result of Komlós~\cite{Kom00}:

\begin{theorem}[Komlós]\label{thm:komlos}
	Let $F$ be a graph with $\chi(F) \geq 2$, $\mu > 0$ and $n$ be sufficiently large.
	Then every $n$-vertex graph $G$ with $\delta_1(G) \geq \big(1-1/\chi_{\crit}(F)+\mu\big)n$ contains an $F$-tiling on at least $(1-\mu)n$ vertices.
\end{theorem}

To extend this result to perfect tilings, we need to address the underlying divisibility issues.
Let $\cD(F) \subset~\NATS$ be the union of the integers $\big||\phi^{-1}(1)|-|\phi^{-1}(2)|\big|$ over all proper $\chi(F)$-colourings~$\phi$ of $F$.
We denote by \gls{gcd} the greatest common divisor of the integers in $\cD(F)\sm \{0\}$ with the convention that $\gcd(F)=\infty$ if $\cD(F)=\{0\}$.
Since this concept is somewhat cryptic, let us also highlight the following consequence, which follows from elementary number theory.
It was proved by Kühn and Osthus~\cite[Lemma 15]{KO09} for graphs and by Mycroft~\cite[Proposition 3.6]{Myc16} for hypergraphs with large \emph{codegrees} (when $d=k-1$).

\begin{lemma}\label{lem:downspin}
	Let $F$ be an ${m}$-vertex $k$-graph with $\gcd(F)=1$ and $\ell=\chi(F)$.
	Let $B$ be the complete $\ell$-partite $k$-graph with parts of size $b+1,\,b-1,\,b,\,\dots,\, b$, where $b$ is sufficiently large and divisible by $m$.
	Then $B$ contains a perfect $F$-tiling.
\end{lemma}

Given this, we can state the classic result for perfect tilings under minimum degree conditions in graphs due to Kühn and Osthus~\cite{KO09}:

\begin{theorem}[Kühn--Osthus]\label{thm:kuhn-osthus}
	Every graph $F$ with $\gcd(F) = 1$ and $ \chi(F) \geq 3$ satisfies $\th_1^{}(\til_F) = 1-1/\chi_{\crit}(F)$.
\end{theorem}

We note that the work of Kühn and Osthus also covers the cases $\gcd(F) \geq 2$ and $\chi(F) = 2$.
In particular, $\th_1^{}(\til_F) \leq 1/2$ if $\chi(F) = 2$.
(See also the discussion after \cref{pro:komlos-kuhn-osthus-thresholds}.)
Moreover, their bounds on the degree conditions are more precise (with an additive instead of multiplicative error).

Next, we turn to the hypergraph setting.
Here, much attention has been dedicated to the case of perfect matchings following the work of Rödl, Ruciński and Szemerédi~\cite{RRS09b} for {minimum codegree conditions}.
For (proper) hypergraph tilings, most results have been given in terms of codegree conditions~\cite{Kee18,SS19,Zha16}.
In particular, Mycroft~\cite{Myc16} determined asymptotically optimal bounds for complete $k$-partite $k$-graph tiles:

\begin{theorem}[Mycroft]\label{thm:mycroft}
	For $k \geq 3$, let $F$ be a complete $k$-partite $k$-graph with $\gcd(F)=~1$.
	Then $\th_{k-1}^{}(\til_F) =  \tau(F)$.
\end{theorem}

Beyond codegree conditions, Lo and Markström~\cite{LM15} provided several general estimates for upper bounds.
Moreover, Han, Zang and Zhao~\cite{HZZ17} determined the minimum vertex degree threshold for complete $3$-partite $3$-uniform tiles.

\subsubsection*{Results}

\cref{thm:minimum-degree-thresholds} allows us to recover the main results from the area and establish a few new bounds.
In particular, we recover \cref{thm:kuhn-osthus} by determining the thresholds of the properties $\SpaF{F}$, $\div_F$ and $\cov_F$, which turns out to be a quick exercise (proved in \cref{sec:dirac-komlos-kuhn-osthus}).

\begin{proposition}\label{pro:komlos-kuhn-osthus-thresholds}
	Every $2$-graph $F$ with $\gcd(F) = 1$ and $\chi(F) \geq 3$ satisfies   
	\begin{align*}
		\th_{1}^{}(\SpaF{F})  &=  1-1/\chi_{\crit}(F)\,, \\
		\th_{1}^{}(\div_F)   &= 1- {1}/{(\chi(F)-1)}\,,\\
		\th_{1}^{}(\cov_F)  &=  1- {1}/{(\chi(F)-1)}\,.
	\end{align*}
\end{proposition}

We also note that $\th_{1}^{}(\div_F)=1-{1}/{\chi(F)}$ when $\gcd(F) \geq 2$.
Moreover, for $\chi(F)=2$, we have $\th_{1}^{}(\div_F)=0$ if the greatest common divisor of $\{v(C)\colon C \subset F \text{ connected}\}$ is~$1$ and $\th_{1}^{}(\div_F)=1/2$ otherwise.
In conjunction with \cref{thm:minimum-degree-thresholds}, this recovers the complete picture for the graph case~\cite{KO09}.
The above observations can be easily derived from the tools developed in \cref{sec:completeness}.
But we omit the details.
Lastly, we remark that one can also recover \cref{thm:komlos} individually following the proof of \cref{thm:minimum-degree-thresholds} and using \cref{prp:almost-perfect-matching} instead of \cref{thm:framework}.

Turning to the hypergraph setting, we also recover \cref{thm:mycroft} by showing the following simple result (proved in \cref{sec:mycroft-thresholds}).

\begin{proposition}\label{pro:mycroft-thresholds}
	For $k \geq 3$, let $F$ be a complete $k$-partite $k$-graph with $\gcd(F)=1$.
	Then
		$\th_{k-1}^{}(\SpaF{F})  =  \tau(F),$  
		$\th_{k-1}^{}(\div_F)   = 0$ and 
		$\th_{k-1}^{}(\cov_F)  =  0$.
\end{proposition}

Lastly, we prove the following new result, which generalises the work of Mycroft to $(k-2)$-degree thresholds of complete $k$-partite $k$-graphs $F$ with $\gcd(F)=1$.
The case of $k=3$ was already shown by Han, Zang and Zhao~\cite{HZZ17}.
We say that a $k$-graph $F$ is a \emph{cone} if it has a proper $k$-colouring with a singleton colour class.

\begin{theorem}\label{thm:k-partite}
	For $k \geq 3$, let $F$ be a complete $k$-partite $k$-graph with $\gcd(F)=1$.
	Suppose that $\tau_1 \leq \tau_2$ are the relative sizes of the two smallest parts of $F$.
	Then
	\begin{align*}
		\th_{k-2}^{}(\SpaF{F}) &=  \max \left \{ 1-\left(1- \tau_1 \right)^{2},\, \left( {\tau_1 + \tau_2} \right)^2 \right\};
		\\
	\th_{k-2}^{}(\div_F)  &= \begin{cases}
				 {1}/{4} & \text{if $k = 3$\,,} \\0 & \text{if $k\geq 4$}\,;
			\end{cases}
		\\	
		\th_{k-2}^{}(\cov_F) &\leq \begin{cases}
				0 & \text{if $F$ is a cone,} \\ 2(\sqrt{2}-1)^2 \approx 0.34 & \text{otherwise}.
			\end{cases}
	\end{align*}
\end{theorem}
 
When $F$ is not a cone, the bound of $\th_{k-2}^{}(\cov_F)$ is matched from below by a construction when $k=3$.
It remains an open problem to determine $\th_{k-2}^{}(\cov_F)$ for $k\geq 4$.

The proofs of the upper bounds of \cref{pro:komlos-kuhn-osthus-thresholds,pro:mycroft-thresholds,thm:k-partite} can be found in \cref{sec:dirac-hyper-proofs}.
The constructions for the lower bounds, derived from the work of Han, Zang and Zhao~\cite{HZZ17}, are in \cref{sec:constructions}.

\subsection{Transversal tilings} \label{sec:transversal}

In the ordinary embedding problem any combination of edges can be used to embed the guest graph.
However, in many applications the host graph may contain certain conflicting edges, and our task is to find a conflict-free embedding (using at most one edge of any conflicting pair).
An interesting special case of this problem arises when the conflicts between edges form an equivalence relation.
If we identify each equivalence class with a \emph{colour}, the host graph is edge-coloured and we have to find a \emph{rainbow} copy of the guest graph, meaning that all its edges have distinct colours.
Apart from being a natural question on its own, this encodes many combinatorial problems such as the Ryser--Brualdi--Stein Conjecture on Latin squares~\cite{KPS+22} and Ringel's tree-packing conjecture~\cite{MPS21}.

In the following, we focus on a setting introduced by Aharoni et al.~\cite{ADH+20}.
Here each colour $1,\dots,{\ell}$ is represented by a hypergraph $G_1,\dots,G_{{\ell}}$ on the same vertex set.
The goal is to find a \emph{transversal} ${\ell}$-edge guest graph $J \subset G_1 \cup \dots \cup G_{\ell}$ whose edges $E(J) = \{e_1,\dots,e_{\ell}\}$ satisfy $e_j \in G_j$ for every $1 \leq j \leq {\ell}$.
For instance, Joos and Kim~\cite{JK20} showed that every choice of $n$-vertex graphs $G_1,\dots,G_{n/2}$ with $\deg(G_j) \geq n/2$  and $n\geq3$ even is guaranteed to contain a transversal perfect matching.
The research on spanning connected transversal structures has undergone a surge of activity in recent years, see for instance the recent survey of Sun, Wang and Wei~\cite{SWW24}.

We investigate these  questions for tilings.
The \emph{minimum $d$-degree threshold for transversal perfect $F$-tilings} of a $k$-graph $F$, written \gls{thransversal-threshold}, is defined as the infimum over all $\delta \in [0,1]$ such that for all $\mu>0$ and $n$ large enough, every collection of $k$-graphs $(G_1,\dots,G_{e(F)n})$ on the same $ v(F)n$ vertices with $\delta_d(G_i) \geq (\delta +\mu) \binom{v(F)n-d}{k-d}$ for every $1\leq i \leq e(F)n$ has a transversal perfect $F$-tiling.
Montgomery, M{\"u}yesser and Pehova~\cite{MMP22} determined this threshold for graphs.
A \emph{bridge} in a graph is an edge whose removal increases the number of connectivity components.

\begin{theorem}[Montgomery--M{\"u}yesser--Pehova]\label{thm:montgomery-muyesser-pehova}
	For every graph $F$, we have 
	\begin{align*}
		\th_1^{}(\rtil_F) = 
		\begin{cases}
			\th_1^{}(\til_F) & \text{if $\chi(F) \geq 3$ or $F$ is bridgeless,} \\
			1/2 & \text{if $\chi(F) = 2$ and $F$ has a bridge.}
		\end{cases}
	\end{align*}
\end{theorem}

Cheng, Han, Wang and Wang~\cite{CHW+23} confirmed this independently for clique-tilings and also showed a corresponding result for matchings in hypergraphs.

\begin{theorem}[Cheng--Han--Wang--Wang]\label{thm:cheng-han-wang-wang}
	We have $\th_d^{}(\rtil_F) =  \th_d^{}(\til_F)$ when $F$ is a $k$-uniform edge.
\end{theorem}

\subsubsection*{Results}

We give a common generalisation of these two results.
To motivate this, suppose we want to show that graphs $G_1,\dots,G_{\ell}$ contain a transversal copy of a graph $J$ whenever each~$G_i$ has relative minimum degree at least $\delta$.
Clearly, the parameter $\delta$ has to be large enough to force a {monochromatic} copy of $J \subset G_1$.
Otherwise, we could prevent a transversal copy of $J$ by taking $G_2,\dots,G_{\ell}$ to be copies of $G_1$.
This alone is however not sufficient.
Assume for instance that $J$ can be written as the union of (not necessarily disjoint) copies of some graph $F$.
It is then required that $\delta$ also forces a `{colour covering}' copy of~$F$, that is, a copy of $F$ with one edge of colour $1$ and all other edges of colour $2$.
Indeed, otherwise we can prevent a transversal copy of $J$ by taking $G_1$ as is and $G_3,\dots,G_{\ell}$ to be copies of $G_2$.
Moreover, this can be the bottleneck.
For instance, we have $\th_1^{}(\til_F) = 1/3$ for $F = K_{2,3} \cup K_{2,5}$~\cite{KO09}, while the above leads to a construction that gives $\th_1^{}(\rtil_F) \geq 1/2$. (Take $G_1 = K_{m,m}$ and $G_2,\dots,G_\ell$ as its complement.) 
As it turns out, constructions of this type are covering barriers in a reformulation of the problem in terms of partite hypergraphs (see~\cref{sec:transversal-proof}).
A tempting `meta question' is now whether these are the only two obstacles for the existence of a transversal copy of $J$.
The main result of this section confirms this asymptotically when $J$ is a perfect $F$-tiling.

To formalise this discussion, let $G_1$ and $G_2$ be $k$-graphs on the same vertex set $V$, and let~$F$ be a $k$-graph.
A \emph{colour covering homomorphism from $F$ to $(G_1,G_2)$} is a (not necessarily injective) function ${\phi\colon V(F) \to V}$ with an edge $f \in F$ such that $\phi(f) \in   G_1 $ and $\phi$ is a homomorphism from~$F'$ to~$G_2$, where~$F'$ is obtained from $F$ by deleting $f$ from its edge set.
We define $\gls{colour-cover}$ to be the infimum over all $\delta \in [0,1]$ such that for all $\mu>0$ and $n$ large enough any two $k$-graphs $G_1,G_2$ on the same $n$ vertices with $\delta_d(G_1),\delta_d(G_2) \geq (\delta +\mu) \binom{n-d}{k-d}$ admit a colour covering homomorphism from $F$ to $(G_1,G_2)$.
By the above discussion, $\th_d^{}(\rtil_F)$ is bounded from below by $\th_d^{}(\til_F)$ and $\th_d^{}(\rmix_F)$.
Our main contribution in this setting states that this is in fact tight.

\begin{theorem}\label{thm:transversal-tiling}
	For every $k$-graph $F$ and $1 \leq d < k$, we have
	\begin{equation*}
		\th_d^{}(\rtil_F) = \max\left\{\th_d^{}(\til_F),\,\th_d^{}(\rmix_F)\right\}\,. 
	\end{equation*}
\end{theorem}

The proof of \cref{thm:transversal-tiling} can be found in \cref{sec:transversal-proof}.
We conclude this section by briefly discussing the consequences of this result.

Note that when $F$ is a $k$-uniform edge, then $\th_d^{}(\rmix_F)  = 0$.
This is because for $k$-graphs $G_1,G_2$, every homomorphism from $F$ to $G_1$ is colour covering for $(G_1,G_2)$.
Hence \cref{thm:transversal-tiling} gives $\th_d^{}(\rtil_F) =  \th_d^{}(\til_F)$ as in \cref{thm:cheng-han-wang-wang}.
Similarly, we can also easily recover \cref{thm:montgomery-muyesser-pehova} from \cref{thm:kuhn-osthus} when $\gcd(F)=1$ and $\chi(F) \geq 3$ (see \cref{obs:transversal}).
The cases where $\gcd(F) \geq 2$ or $\chi(F) = 2$ follow similarly as discussed after \cref{pro:komlos-kuhn-osthus-thresholds}, but we omit details.

For hypergraphs, we may draw the following straightforward conclusion.
Let $G$ be a $3$-graph on $n$ vertices.
If $G$ has minimum codegree greater than $(2/3) n$, then every edge of $G$ is contained in a copy of a tetrahedron $F = K_4^{(3)}$.
It follows that $\th_2^{}(\rmix_F) \leq 2/3$.
On the other hand, it is known that $\th_2^{}(\til_F)=3/4$~\cite{KM15,LM15}.
In combination with \cref{thm:transversal-tiling}, this gives $\th_2^{}(\rtil_F)=3/4$.

As an extension of the above, we remark that if for a $k$-graph $F$ a relative minimum $d$-degree~$\delta$ asymptotically guarantees that every edge is contained in a copy of $F$, then $\th_d^{}(\rmix_F) \leq \max\{\delta,\,1/2\}$.
Indeed, given two $k$-graphs $G_1$ and $G_2$ of this degree, we first find a common edge $e \in G_1 \cap G_2$ (using the bound $1/2$) and then extend it to a copy of $F$ in $G_2$ (using the bound $\delta$).
 
Lastly, we highlight an application to transversal Hamilton cycles.
A classic result of Komlós, Sárközy and Szemerédi~\cite{KSS98} concerns optimal degree conditions for powers of Hamilton cycles, which resolves the so-called P\'osa--Seymour Conjecture.
The best known bounds for the analogous problem in hypergraphs are due to Pavez-Signé, Sanhueza-Matamala and Stein~\cite{PSS23}.
Recently, Gupta, Hamann, M{\"u}yesser, Parczyk and Sgueglia~\cite{GHM+23} developed a framework for transversal powers of Hamilton cycles and used this to obtain a transversal analogue of the P\'osa--Seymour Conjecture.
As it turns out, one can combine this framework with \cref{thm:transversal-tiling} to obtain a transversal version of the results of Pavez-Signé et al.
Since this paper focuses on tilings, we just give an outline of the argument.

For $2 \leq k \leq r$, the \emph{$(r-k+1)$st power of a Hamilton cycle} in a $k$-graph $G$ is an ordering of the vertices of $G$ such that every interval of $r$ consecutive vertices induces a clique.
Denote by $\th_{k-1}^r(\text{Ham})$ the \emph{minimum codegree threshold} for containing such a cycle.
It was proved by Pavez-Signé et al.~\cite{PSS23} that $\th_{k-1 }^r(\text{Ham}) \leq  \delta$ for $\delta = 1-\big(\binom{r-1}{k-1} + \binom{r-2}{k-2}\big)^{-1} $.
Let $F$ be a $k$-uniform $r$-clique.
We have $\th_{k-1}^{}(\til_F) \leq \delta$, since $(r-k+1)$st powers of Hamilton cycles contain perfect $r$-clique-tilings (as long as the host graph has order divisible by $r$).
Moreover, $\th_{k-1}^{}(\rmix_F) \leq \delta$, which is a simple generalisation of the above observation that $\th_2^{}(\rmix_F) \leq 2/3$.
By \cref{thm:transversal-tiling}, we thus obtain $\th_{k-1}^{}(\rtil_F) \leq \delta$.
One can then derive a transversal version of the results of Pavez-Signé et al.~\cite{PSS23} by following the steps of Gupta et al.~\cite[Section 5.1]{GHM+23} in their proof of a transversal analogue of the P\'osa--Seymour Conjecture.

\subsection{Ordered graphs}\label{sec:ordered-graphs}

Over the past years, many results from Turán and Ramsey theory have been transferred to (vertex-)ordered graphs~\cite{Tar19}.
The challenge in this setting is to find order-preserving embeddings.
A $k$-graph $F$ is \emph{(vertex) ordered} if its vertices are equipped with a linear ordering.
We typically assume that the vertices of $F$ form an interval in $\NATS$ with the natural ordering of the integers.
For another ordered $k$-graph $G$, a homomorphism $\phi \in \hom{F}{G}$ is \emph{order-preserving} if $\phi(u) \leq \phi(v)$ whenever $u < v$ for $u,v \in V(F)$.
Note that if $x < y < z$ are vertices of $F$, and $x,z$ are mapped to a vertex $v$ of $G$, then $y$ must also be mapped to $v$ under an order-preserving homomorphism.
Put differently, the preimages of vertices of $G$ are intervals.
An \emph{order-preserving copy} of $F$ in $G$ corresponds to an injective order-preserving homomorphism.
Naturally, an \emph{order-preserving $F$-tiling} of $G$ is a tiling of $G$ with order-preserving copies of $F$.
For $1\leq d < k$, we define the minimum $d$-degree threshold $\th_d^{}(\otil^{}_F)$ as in \cref{sec:dirac-hyper}, where the host  structure ranges over all ordered $k$-graphs.

The threshold for an almost perfect order-preserving $F$-tiling was established by Balogh, Li and Treglown~\cite{BLT22} for ($2$-uniform) graphs $F$.
Recently, Freschi and Treglown~\cite{FT22} were able to determine $\th_1^{}(\otil_F)$ for all graphs $F$ (\cref{thm:freschi-treglown}).
To state this result, we introduce some further notation.\footnote{Our exposition differs slightly from that of Freschi and Treglown. Instead of $\chi^<(F)$ and $\chi^<_{\crit}(F)$, they use the symbols $\chi_<(F)$ and $\chi_{\crit}^\ast(F)$, respectively.
Moreover, $F$ is an ordered cone in our terminology if and only if~$F$ does not have a `local barrier' in their terminology.}
Let $F$ be an ordered graph on~${m}$ vertices.
An \emph{interval $r$-colouring} of~$F$ is a partition of $V(F)$ into $r$ intervals $V_1,\dots,V_r$ such that no two vertices belonging to the same interval are adjacent in $F$.
So each $V_i$ consists of consecutive integers. Moreover, we implicitly assume that the vertices in $V_i$ are smaller than those in $V_j$ whenever $i<j$.
The \emph{interval chromatic number} $\chi^<(F)$ of an ordered graph $F$ is the least $r$ such that there exists an interval $r$-colouring of~$F$.

We say that a complete $k$-partite unordered graph $B$ is a \emph{bottle graph} for $F$ (a term introduced by Komlós~\cite{Kom00}), if for every permutation $\sigma \in \cS_k$, there exists a $b$ such that the blow-up $B(b)$ ordered by $\sigma$ contains a perfect $F$-tiling.\footnote{Formally, $B(b)$ is obtained by replacing each vertex of $B$ with a {part} of $b$ new vertices and each edge with a complete bipartite graph between the parts such that the ordering between the vertices of the parts is inherited from an ordering of the parts by $\sigma$.}
The \emph{ordered critical chromatic number} $\chi^<_{\crit} (F)$ of $F$ is defined as $\chi^<_{\crit} (F) = \min \{\chi_{\crit} (B)\colon \text{$B$ is a bottle graph for $F$}\}$.
Note that $\chi^<(F)-1 \leq \chi_{\crit}^<(F)$. 
(To see this note that every ($k$-partite) bottle graph $B$ satisfies $\chi^<(B) = k \geq \chi^<(F)$, since otherwise its blow-ups
contain no ordered copy of~$F$.
Moreover, $\chi_{\crit}(B) \geq k-1 = \chi^<(B)-1$.)
On the other hand and contrary to the unordered setting, we do not always have $\chi_{\crit}^<(F) \leq \chi^< (F)$~\cite{FT22}.

Finally, we say that $F$ is an \emph{ordered cone} if $r = \chi^<(F)\geq2$ and for all distinct $i, j \in \{1,\dots,r+1\}$, there is an interval $(r+1)$-colouring of $F$ with colour classes $V_1,\dots,V_{r+1}$ such that $V_i$ contains a single vertex and there is no edge between $V_i$ and $V_j$.
Given this, we are ready to formulate the main result in this area due to Freschi and Treglown~\cite{FT22}.

\begin{theorem}[Freschi--Treglown]\label{thm:freschi-treglown}
	Let $F$ be an ordered graph with $\chi^<(F) \geq 3$. Then
	\begin{align*}
		\th_{1}^{}(\otil_F^{}) = 
		\begin{cases}
			1 - {1}/{{\chi^<_{\crit}(F)}} & \text{if ${\chi^<_{\crit}(F)} \geq {\chi^<(F)}$,} \\
			1 - {1}/{{\chi^<(F)}} & \text{if ${\chi^<_{\crit}(F)} < {\chi^<(F)}$ and $F$ is not an ordered cone,} \\
			1 - {1}/{{\chi^<_{\crit}(F)}} & \text{if ${\chi^<_{\crit}(F)} < {\chi^<(F)}$ and $F$ is an ordered cone.}
		\end{cases}
	\end{align*}
\end{theorem}

\subsubsection*{Results}

Our main contribution characterises the barriers for perfect tilings in the ordered setting analogous to \cref{thm:minimum-degree-thresholds}.
From this we derive a short proof of \cref{thm:freschi-treglown}.
To connect the vertex-ordered setting with our main result, we introduce the corresponding auxiliary digraph.

\begin{definition}[Vertex-ordered setting]\label{def:hom-graph-ordered}
	Given ordered $k$-graphs $F$ and $G$ such that $V(F)=\{1,\dots,m\}$, let \gls{digraph-vertex-ordered} be the $m$-digraph on vertex set $V(G)$ with an edge $(\phi(1),\dots,\phi(m))$ for every order-preserving $\phi \in \hom{F}{G}$.
\end{definition}

For an ordered $k$-graph $F$ and $1\leq d < k$, let $\th_d^{}(\ospa^{}_F)$, $\th_d^{}(\odiv^{}_F)$ and $\th_d^{}(\ocov^{}_F)$ be the infimum over all $\delta \in [0,1]$ such that for all $\mu>0$ and $n$ large enough, every $n$-vertex ordered $k$-graph~$G$ with $\delta_d(G) \geq (\delta +\mu) \binom{n-d}{k-d}$ satisfies $H^{\Ord}(F;G) \in \spa_{\mu}^{}$, $\div$ and $\cov$, respectively.\footnote{In fact, our proof works with $\spa_{0}$, but we take this as an opportunity to illustrate the use of \cref{thm:framework-exceptional}.}
Our framework allows us to decompose the threshold for order-preserving perfect tilings as follows:

\begin{corollary}\label{cor:threshold-decomposition-ordered}
	For every ordered $k$-graph $F$ and $1 \leq d \leq k-1$, we have
	\begin{equation*}
		\th_d^{}(\otil^{}_F) = \max \big\{	\th_d^{}(\ospa^{}_F),\,	\th_d^{}(\odiv^{}_F),\,	\th_d^{}(\ocov^{}_F) \big\}\,.
	\end{equation*}
\end{corollary}

In light of this, the next result recovers the upper bounds of \cref{thm:freschi-treglown}.
For the lower bound constructions, we refer the reader to the work of Freschi and Treglown~\cite{FT22}.

\begin{proposition}\label{pro:thresholds-ordered}
	Let $F$ be an ordered graph with $\chi^<(F) \geq 3$.
	Then
	\begin{align*}
		\th_1^{}(\ospa^{}_F) &\leq  1- {1}/{{\chi^<_{\crit}(F)}}\,; 
		\\
		\th_1^{}(\odiv^{}_F)  &\leq \begin{cases}
			1- {1}/{({\chi^<(F)}-1)} & \text{if $ \chi^<_{\crit}(F)< {\chi^<(F)}$},
			\\ 1- {1}/{{\chi^<(F)}} & \text{otherwise};
		\end{cases}
		\\	
		\th_1^{}(\ocov^{}_F) &\leq \begin{cases}
			1- {1}/{({\chi^<(F)}-1)} & \text{if $F$ is an ordered cone,} 
			\\ 1- {1}/{{\chi^<(F)}} & \text{otherwise}.
		\end{cases}
	\end{align*}
\end{proposition}

\cref{cor:threshold-decomposition-ordered,pro:thresholds-ordered} are proved in \cref{sec:ordered-graphs-proofs}.
We remark that the bound on $\th_1(\odiv^{}_F)$ relies on an insight of Freschi and Treglown (\cref{lem:flexible}).
On the other hand, the bound on $\th_1(\ospa^{}_F)$ follows from a new short argument.

\subsection{Uniform density} \label{sec:quasirandomness}

Our last application concerns deterministic hypergraphs that exhibit typical properties of a random hypergraph.
For $\eps, d >0$, we say that an $n$-vertex $k$-graph $G$ is \emph{uniformly $(\eps,d)$-dense} if for all (not necessarily disjoint) $X_1,\dots, X_k \subset V(G)$, we have
\begin{equation*}
	e_G(X_1,\dots,X_k) \geq d\, |X_1|\cdots |X_k| - \eps n^k
\end{equation*}
where $e_G(X_1,\dots,X_k)$ is the number of edges $ v_1 \dots v_k \in G$ with $v_i \in X_i$ for each $i \in [k]$.
Let $\gls{uni-dense}$ contain, for every $k\geq 2$, all  $(\eps,d)$-uniformly dense $n$-vertex $k$-graphs~$G$.

Lenz and Mubayi~\cite{LM16} asked for which tiles $F$ every large enough uniformly dense hypergraph is guaranteed to contain a perfect $F$-tiling.
Since uniform density does not guarantee that every vertex is on an edge, we have to add a mild degree condition for any chances of success.
(Recall the definition of $\DegF{d}{\delta}$ in \cref{sec:example-application}.)
Formally, let $\cF({\til})$ contain, for every $k\geq 2$, the $k$-graphs $F$ such that for every $d,\mu >0$, there is an $\eps >0$  such that for every  
$n    $ large enough divisible by~$v(F)$, every $n$-vertex $k$-graph $G \in {\DenF{\eps}{d}^{}} \cap \DegF{1}{\mu}$ has a perfect $F$-tiling or, equivalently, $H(F;G) \in \mat$ (\cref{obs:digraph-property-equivalence}). 

Lenz and Mubayi~\cite{LM16} initiated the study of the class $\cF({\til})$ and showed that it contains all linear hypergraphs.
More recently, Ding, Han, Sun, Wang and Zhou~\cite{DHS+22} provided a characterisation of $k$-partite $k$-graphs and all $3$-graphs in $\cF({\til})$.
For further progress on these problems in the multipartite setting see also  references~\cite{DHS+23,Sun23}.

A crucial step towards the characterisation of Ding et al.~\cite{DHS+22} is to reduce the problem of finding a perfect tiling to the problem of covering all vertices.
Let $\cF({\cov})$ be defined analogously to $\cF({\til})$ with $H(F;G) \in \cov$ instead of $H(F;G) \in \mat$.

\begin{theorem}[Ding--Han--Sun--Wang--Zhou]\label{thm:Ding-Han-Sun-Wang-Zhou}
	We  have $\cF({\til}) = \cF({\cov})$.
\end{theorem}

\subsubsection*{Results}

We use our framework to give a short proof of \cref{thm:Ding-Han-Sun-Wang-Zhou}.
Define $\cF({\div})$ analogously to $\cF({\cov})$ with $H(F;G) \in \div$ instead of $H(F;G) \in \cov$.
Moreover, let $\cF({\spa})$ contain the $k$-graphs $F$ such that for every $d,\,\mu,\,\rho >0$, there are $\eps,n_0$  such that every $k$-graph $G \in {\DenF{\eps}{d}^{}} \cap \DegF{1}{\mu} $ on $n \geq n_0$ vertices satisfies $H(F;G) \in \spa_\rho$.
Since uniformly dense graphs are approximately closed under taking typical induced subgraphs of constant order (\cref{lem:inheritance-uniformly-dense}), our framework has the following implication:

\begin{corollary}	\label{cor:uniformly-dense}
	We have $\cF({\til}) =  \cF({\spa}) \cap \cF({\div}) \cap \cF({\cov})$.
\end{corollary}

It follows immediately from the definition that a large enough induced subgraph of a uniformly dense graph is still uniformly dense with somewhat weaker parameters.
To be precise, if $G$ is a uniformly $(\eps,d)$-dense $k$-graph on $n$ vertices, then every induced subgraph of $G$ on at least~$\rho n$ vertices is still uniformly $(\eps/\rho^k,d)$-dense.
So given a graph in $\cF({\cov})$, we may recursively remove copies of $F$ to arrive at an almost perfect $F$-tiling.
This shows $\cF({\cov}) \subset \cF({\spa})$.
We may therefore recover \cref{thm:Ding-Han-Sun-Wang-Zhou} from the next observation. 

\begin{proposition}\label{lem:quasirandom-divisibility}
	We have $\cF_{}^{}(\cov) \subset \cF({\div})$.
\end{proposition}

We remark that \cref{lem:quasirandom-divisibility} is implicitly proved in the  work of Ding et al.~\cite{DHS+22}.
In \cref{sec:quasirandomness-proofs}, we give a streamlined version of the argument and also prove \cref{cor:uniformly-dense}.

\section{Dirac-type results for hypergraphs -- Proofs} \label{sec:dirac-hyper-proofs}

Our first three applications take place in the setting of Dirac's theorem, which concerns optimal minimum degree conditions that ensure a perfect tiling.
Recall that \cref{thm:minimum-degree-thresholds} reduces this problem to determining the thresholds for space, divisibility and covering.
In the following two subsections, we develop a few general tools for satisfying the space and divisibility properties under degree conditions.
In the last three subsections, we give the proofs of \cref{pro:komlos-kuhn-osthus-thresholds,pro:mycroft-thresholds,thm:k-partite}.

\subsection{Fractional matching via link graphs}\label{sec:tilings-via-linkgraphs}

To study the threshold for space, we focus on the structure of the link graphs.
{For a $k$-graph $G$ and a $d$-set $D \subset V(G)$, the \emph{link graph} $L_G(D)$ is the $(k-d)$-graph on {$V(G-D)$} with an edge $Y$ whenever $D \cup Y \in G$.}
A common way to find a  perfect fractional matching in $G$ is to show that every link graph contains a large enough matching, an idea systematically studied by Alon, Frankl, Huang, Rödl, Ruci\'{n}ski and Sudakov~\cite{AFH+12}.
In the following, we generalise this approach to partite tilings in hypergraphs.
For graphs, similar results have been obtained by Hladk{\`y}, Hu and Piguet~\cite{HHP21}.\COMMENT{We remark that at this point it is crucial that fractional tilings are defined in terms of weighted homomorphisms and not just in terms of weighted copies.}

Recall that the size of a {fractional matching} $w \colon H \rightarrow [0,1]$ in an $m$-digraph $H$ on $n$ vertices is $\sum_{e \in H} w(e) \leq n/m$.
Let $F$ be an ${m}$-vertex $k$-graph and $0 \leq \beta \leq 1/m$.
We denote by \gls{edge-turan} the infimum over all $\delta \in [0,1]$ such that for every $\mu >0$, $n$ large enough and every $n$-vertex $k$-graph~$G$ with $\delta_0(G)=e(G)\geq (\delta +\mu) \binom{n}{k}$, the ${m}$-digraph $H(F;G)$ has a fractional matching of size at least $\beta n$.
For $U \subset V(F)$, we denote by $F \sm U$ the hypergraph obtained from $F$ by replacing each edge~$e$ with $e\sm U$ and deleting the vertices of $U$.
(Not to be confused with $F-U=F[V(F) \sm U]$.)
We use this notation only for partite graphs $F$ with $U$ containing all vertices of some parts, so that~$F \sm U$ continues to be uniform.

\begin{lemma}\label{lem:linkgraph-threshold}
	Let $F$ be an $m$-vertex $k$-partite $k$-graph.
	Let $U\subset V(F)$ contain the vertices of $d$ parts of $F$, and let $F'=F\sm U$.
	Then $\th^{}_d(\SpaF{F}) \leq \th_0^{}\left(\SpaF{F'}; 1/{m}\right)$.
\end{lemma}

We obtain \cref{lem:linkgraph-threshold} as a corollary of the following more general result, which is of independent interest.
Recall the connection between an $F$-tiling in a $k$-graph $G$ and a matching in the auxiliary $m$-digraph $H(F;G)$ with $m=v(F)$ (\cref{obs:digraph-property-equivalence}).
 
\begin{lemma}\label{lem:linkgraph}
	Let $F$ be an ${m}$-vertex $k$-partite $k$-graph.
	Let $U\subset V(F)$ contain the vertices of~$d$ parts, and let $F'=F\sm U$ with $m'=v(F')$.
	Let $G$ be an $n$-vertex $k$-graph.
	Suppose that for every $d$-set $D \subset V(G)$, the ${m'}$-digraph $H(F';L_G(D))$ contains a fractional matching of size at least $n/{m}$.
	Then $H(F;G)$ contains a perfect fractional matching.
\end{lemma}

The proof of \cref{lem:linkgraph} is based on linear programming duality following the ideas of Alon et al.~\cite{AFH+12}.
The dual program for fractional matchings in directed hypergraphs is defined as follows.
For an $m$-digraph~$H$, a \emph{fractional cover} is a function $\bvec{c}\colon V(H) \to \REALS$ such that ${\sum_{v \in V(H)} \bvec c(v) \multi(v,e) \geq 1}$ for every $e \in H$.
The \emph{size} of $\vecb c$ is $\sum_{v \in V(H)} \vecb c (v)$.
We denote by $\nu(H)$ the maximum size of a fractional matching and by $\lambda(H)$ the minimum size of a fractional cover of $H$.
By linear programming duality, we have $\nu(H)=\lambda(H)$.

\begin{proof}[Proof of \cref{lem:linkgraph}]
	For the sake of contradiction, suppose that $H = H(F;G)$ does not have a perfect fractional matching and thus $\nu(H) <  n/m$.
	By duality, there is a fractional cover $\bvec{c} \in \mathbb{R}^{V(H)}$ of size $\lambda(H) = \nu(H)  <  n/m$.
	Let $D$ be a set of $d$ vertices of lowest weight.
	We obtain a weighting $\bvec{c}'$ from $\bvec{c}$ by replacing the weights of the vertices in $D$ by their average weight under~$\bvec{c}$.
	Note that~$\bvec{c}'$ still has the same size as~$\bvec{c}$.
	
	Let $c_0 = \bvec{c}'(w)$ for some $w \in D$, and observe that $c_0 <1/m$.
	(Otherwise, $\bvec c'$ would have size $n/m$.)
	If $c_0 >0$, we apply  to  $\bvec{c}'$ the following linear transformation
	\begin{equation*}
		\bvec{c}''(v) = \frac{\bvec{c}'(v)-c_0}{1-m c_0}\,.
	\end{equation*}
	Note that the size of $\bvec{c}''$ is still less than $n/m$ as
	\begin{equation*}
		(1-m c_0) \sum_{v \in {V(H)}}  \vecb c''(v) =\sum_{v \in {V(H)}} (\bvec{c}'(v) - c_0) < \frac{n}{m} - c_0 n = (1- m c_0) \frac{n}{m}\,.
	\end{equation*}
	If $c_0 = 0$, we simply set $\bvec{c}'' = \bvec{c}'$.
	Note that in either case, $\bvec c''$ vanishes on~$D$.
	
	Let $L=L_G(D)$ be the link $(k-d)$-graph of $D$ in $G$.
	Set $H' = H(F';L)$.
	We can identify each edge $f \in H'$ with an edge $e_f \in H$ whose corresponding homomorphism in $\hom{F}{G}$ maps the vertices of each part of $U$ to one of the vertices of $D$.
	So for an arbitrary  $f \in H'$, we obtain 
	\begin{equation*}
		\sum_{v \in V(L)} \bvec c''(v) \multi(v,f)  = \sum_{v \in V(H)} \bvec c''(v) \multi(v,e_f) 
		=  \frac{\sum_{v \in V(H)} \bvec{c}'(v) \multi(v,e_f)  - mc_0}{1-m c_0} \geq 1\,.
	\end{equation*}
	Hence, $\bvec c''$ restricted to $V(H)\sm D$ is a fractional cover of $H'$.
	Since $\bvec c''$ vanishes on~$D$, we obtain $\nu(H') = \lambda(H') <  n/m$ in contradiction to the assumption.
\end{proof}

\begin{proof}[Proof of \cref{lem:linkgraph-threshold}]
	Set $\delta = \th_0^{}\left(\spa_{F'}; 1/{m}\right)$ and $m'=v(F')$.
	Given $\mu > 0$, consider an $n$-vertex $k$-graph $G$ with relative minimum $d$-degree at least $\delta +\mu$ where $n$ is large enough.
	We have to show that the $m$-digraph $H(F;G)$ contains a perfect fractional matching.
	By assumption, the link graph $L_G(D)$ has edge density at least $\delta+\mu/2$ for each $d$-set $D \subset V(G)$.
	By choice of~$\delta$, the $m'$-digraph $H(F';L_G(D))$ has a fractional matching of size at least $n/m$.
	(Technically, we just obtain a matching of size $(n-d)/m$. But we can boost this to $n/m$ by setting aside a few edges, and then using the choice of $\delta$ on the remainder.)
	By \cref{lem:linkgraph}, the $m$-digraph $H(F;G)$ contains a fractional matching of size at least $n/m$, which is equivalent to containing a perfect fractional matching.
\end{proof}

Finally, we prove the following variant of \cref{lem:linkgraph} for $2$-uniform graphs using essentially the same argument.

\begin{lemma}\label{lem:linkgraph-d=2}
	Let $F$ be an ${m}$-vertex complete $k$-partite $2$-graph and $U \subset V(F)$ be one of its parts.
	Let $F' = F - U$ with $m'=m-|U|$.
	Let $G$ be an $n$-vertex $2$-graph.
	Suppose that for every vertex $v \in V(G)$, the ${m'}$-digraph $H(F';G[N(v)])$ contains a fractional matching of size at least $n/{m}$.
	Then $H(F;G)$ contains a perfect fractional matching.
\end{lemma}

\begin{proof}
	For the sake of contradiction, suppose that $H = H(F;G)$ has $\nu(H) <  n/m$.
	By duality, there is a fractional cover $\bvec{c} \in \mathbb{R}^{V(H)}$ of size $\lambda(H) = \nu(H)  <  n/m$.
	Let $w$ be a vertex of lowest weight. 
	Set $c_0 = \bvec{c}(w)$, and observe that $c_0 <1/m$.
	(Otherwise, $\bvec c$ would have size $n/m$.)
	If $c_0 >0$, we apply  to  $\bvec{c}$ the following linear transformation
	\begin{equation*}
		\bvec{c}'(v) = \frac{\bvec{c}(v)-c_0}{1-m c_0}\,.
	\end{equation*}
	Note that the size of $\bvec{c}'$ is still less than $n/m$ as
	\begin{equation*}
		(1-m c_0) \sum_{v \in {V(H)}}  \vecb c'(v) =\sum_{v \in {V(H)}} (\bvec{c}(v) - c_0) < \frac{n}{m} - c_0 n = (1- m c_0) \frac{n}{m}\,.
	\end{equation*}
	If $c_0 = 0$, we simply set $\bvec{c}' = \bvec{c}$.
	Note that in either case, $\bvec c'$ vanishes on~$w$.
	
	Let $H' = H(F';G[N(w)])$.
	We can identify each edge $f \in H'$ with an edge $e_f \in H$ whose corresponding homomorphism in $\hom{F}{G}$ maps the vertices of $U$ to $w$.
	Note that $\multi(v,e_f) = 0$ if $v \notin N(w)$ and $\bvec c'(w)=0$.
	So for an arbitrary  $f \in H'$, we obtain 
	\begin{align*}
		\sum_{v \in N(w)} \bvec c'(v) \multi(v,f)  &= \sum_{v \in N(w)} \bvec c'(v) \multi(v,e_f)  
		= \sum_{v \in V(H)} \bvec c'(v) \multi(v,e_f) \\
		&=  \frac{\sum_{v \in V(H)} \bvec{c}(v) \multi(v,e_f)  - mc_0}{1-m c_0} \geq 1\,.
	\end{align*}
	Hence, $\bvec c'$ restricted to $N(w)$ is a fractional cover of $H'$.
	So $\nu(H') = \lambda(H') <  n/m$, which contradicts the assumption.
\end{proof}

\subsection{Lattice completeness via connectivity}\label{sec:completeness}

In order to obtain the divisibility condition for $k$-graphs $F$ with $\gcd(F)=1$, it is convenient to work with the following notion of connectivity.
Let $G$ be a $k$-graph and $\ell \geq k$.
We define  $K_{\ell}(G)$ as the ($2$-uniform) graph on $V(G)$ obtained by placing a $2$-uniform $\ell$-clique on the vertex set of every $k$-uniform $\ell$-clique of $G$.
We say that~$G$ is \emph{$\ell$-connected} if $K_{\ell}(G)$ is connected (in the usual sense).

A lack of connectivity  may pose an obstruction to perfect tilings.
We have already seen this in the introduction: the union of two odd cliques does not contain a perfect matching.
More generally, consider an ${m}$-vertex $k$-graph $F$ that is $\ell$-connected for $\ell = \chi(F) $.
(This is, for instance, the case when $F$ is complete $k$-partite.)
Now suppose that $G$ is a $k$-graph such that the $2$-graph $K_{\ell}(G)$ has two (connected) components $A$ and $B$.
It follows that the vertex set of each copy of $F$ in~$G$ must be contained in either $A$ or $B$.
So if both $|A|$ and~$|B|$ are not divisible by ${m}$, then $G$ cannot have a perfect $F$-tiling.

On the other hand, if $\gcd(F)=1$, then connectivity implies the divisibility property:

\begin{lemma}\label{lem:connectivity-to-completeness}
	Let $F$ be an $m$-vertex $k$-graph with $\gcd(F)=1$ and $\ell = \chi(F)$.
	For $n$ divisible by~$m$, let $G$ be an $n$-vertex $\ell$-connected $k$-graph.
	Then $ H(F;G)$ satisfies $\div$.
\end{lemma}

\begin{proof}
	Let $B$ be the complete $\ell$-partite $k$-graph with parts of size $b+1,\,b-1,\,b,\dots,b$ where $b$ is divisible by $(\ell-1)!\cdot m$ and large enough in terms of ${m}$ and $\ell$.
	So $B$ has a perfect $F$-tiling by \cref{lem:downspin}.
	Let $B'$ be the complete $\ell$-partite $k$-graph with parts of uniform size $b$.
	Then $B'$ also has a perfect $F$-tiling.
	(This is because the disjoint union of $\ell!$ copies of $F$ with permuted colour classes assigns exactly $(\ell-1)!m$ vertices to each part of $B'$.)
	
	Consider an arbitrary $\ell$-clique $K \subset G$.
	(Such a clique exists since~$G$ is $\ell$-connected.)
	Recall the definition of the $m$-digraph $H'=H(F;K)$ (\cref{def:hom-graph}) and the lattice $\cL(H')$ in \cref{sec:lattice-completeness}.
	Observe that $\cL(H')$ contains the element $(b+1,b-1,b,\dots,b)$ and each of its permutations.
	Moreover, $\cL(H')$ also contains the element $(b,\dots,b)$.
	So for $H=H(F;G)$, the lattice $\cL(H)$ contains all transferrals $\vn_v - \vn_u$ whose endpoints $u,v$ are in a common $\ell$-clique.
	By the definition of $\ell$-connectivity, we can generalise this to arbitrary pairs of vertices by navigating along the $\ell$-cliques of $G$.
	It follows by \cref{obs:lattice-completeness} that $\cL(H)$ is complete.
	So in particular, $H(F;G)$ has a perfect integral matching as required.
\end{proof}

For a $k$-graph $F$ with $\ell = \chi(F)$ and $1\leq d < k$, define $\th_d^{}(\con_F)$ to be the minimum $d$-degree threshold that forces $\ell$-connectivity in a $k$-uniform host graph.
To give a simple example, a $k$-partite $k$-graph $F$ with $k\geq 3$ satisfies $\th_{2}^{}(\con_F) = 0$, since every two vertices are on a common edge, which is also a $k$-clique in $K_k(G)$.
The next result is an immediate consequence of \cref{lem:connectivity-to-completeness}.

\begin{corollary}\label{cor:completeness-connectivity}
	For a $k$-graph $F$ with $\gcd(F)=1$ and $1 \leq d < k$, we have $\th_d^{}(\div_F) \leq\th_d^{}(\con_F)$.
\end{corollary}

\subsection{Minimum degree thresholds for graphs}\label{sec:dirac-komlos-kuhn-osthus}

In the following, we show \cref{pro:komlos-kuhn-osthus-thresholds}.
Consider a ($2$-uniform) graph $F$ with $k = \chi(F) \geq 3$ and $\gcd(F)=1$.
Let $G$ be an $n$-vertex graph with $\delta(G) > \left( 1- 1/(k-1)\right)n$.
So, in particular, every set of $k-1$ vertices has a common neighbour when $n$ is large enough.
For a given vertex $v_1$, this implies that there are $v_2,\dots,v_{k+1}$ such that both $\{v_1,\dots,v_k\}$ and $\{v_2,\dots,v_{k+1}\}$ induce a $k$-clique in $G$.
This immediately gives $\th_{1}^{}(\cov_F) \leq  1- 1/(k-1)$.
Similarly, every edge of $G$ is on a $k$-clique.
Combining this with the fact that $G$ is $2$-uniform, it follows that $K_k(G) = G$.
Lastly, note that $\delta(G) > n/2$ since $k \geq 3$, and hence $K_k(G) = G$ is connected. 
{By \cref{cor:completeness-connectivity} we can bound the divisibility threshold as required for \cref{pro:komlos-kuhn-osthus-thresholds}.}
It remains to bound the space threshold:

\begin{lemma}\label{lem:fractional-komlos}
	For any graph $F$, we have $\th_{1}^{}(\SpaF{F}) \leq 1- {1}/{\chi_{\crit}(F)}$.
\end{lemma}

We remark that \cref{lem:fractional-komlos} has already been proved by Hladk{\`y}, Hu and Piguet~\cite{HHP21}.
For the sake of completeness, we include the details.

\begin{proof}[Proof of \cref{lem:fractional-komlos}]
	Consider an $m$-vertex graph $F$ with $k = \chi(F)$.
	Trivially $\th_{1}^{}(\SpaF{F}) =~0$ when $k=1$.
	We proceed to show the lemma by induction on $k \geq 2$, assuming that the statement holds for smaller values of $k$.
	By assumption, $F$ has a proper $k$-colouring whose smallest colour class has relative size $\tau=\tau(F)$.
	Since we are aiming for a fractional matching in $H(F;G)$, we may assume that $F$ has one \emph{small} colour class of relative size $\tau$ and $k-1$ \emph{big} colour classes of relative size $(1-\tau)/(k-1)$.
	(Otherwise, replace $F$ with a suitable disjoint union of $(k-1)!$ copies of $F$.)
	Let $U$ be one of the big colour classes, and set $F'=F \sm U$ with $m' =v(F')$.
	Recalling that $\chi_{\crit}= ({k-1})/({1-\tau})$, this gives
	\begin{equation*}
		m' = \left(1-\frac{1-\tau }{k-1} \right) m = \left(1-\frac{1 }{\chi_{\crit}} \right) m \,.
	\end{equation*}
	It follows that $\chi(F')=k-1$, and 
	\begin{equation*}
		\tau' := \tau(F') =  \frac{\tau m}{m'}  = \frac{\tau}{1- \tfrac{1-\tau}{k-1}}   \,.
	\end{equation*}
	\begin{claim}\label{cla:fractional-komlos-computation}
		We have  $2  - \frac{1}{1-\frac{1-\tau }{k-1}}  = 1-\frac{1-\tau'}{k-2}$ if $k\geq 3$.
	\end{claim}
	\begin{proofclaim}
		Solving the equation $1-\frac{1-x}{k-2} = 2 -\frac{x}{\tau}$ for $x$ gives $x = \tau \frac{k-1}{k-2+\tau}$.
		Moreover, $\frac{k -2 + \tau}{k-1}  =  1- \tfrac{1-\tau}{k-1}$ by rearranging.
		Hence, $x = \frac{\tau}{1- \tfrac{1-\tau}{k-1}}$, as desired.
	\end{proofclaim}
	
	Now consider an $n$-vertex graph $G$ with $\delta(G) \geq ( 1- {1}/{\chi_{\crit}(F)})n$.
	Fix a vertex $v\in V(G)$.
	{By \cref{lem:linkgraph-d=2}, it suffices to show that there is a fractional matching of size at least $n/m$ in the $m'$-digraph $H(F';G[N(v)])$.}
	By assumption on~$G$ and \cref{cla:fractional-komlos-computation}, every vertex $w \in N(v)$ satisfies
	\begin{align*}
		\frac{\deg_{G[N(v)]}(w)}{|N(v)|}  &\geq \frac{|N(w)| - |V(G) \sm N(v)|}{|N(v)| }  
		  \geq  \frac{2\delta(G) - v(G)}{\delta(G)}  \\
		 &=   2  - \frac{v(G)}{\delta(G)} 
		   \geq   2  - \frac{1}{1-\frac{1-\tau }{k-1}}  = 1-\frac{1-\tau'}{k-2} = 1- \frac{1 }{\chi_{\crit}(F')}\,.
	\end{align*}
	So by the induction assumption, $H(F';G[N(v)])$ has a perfect fractional matching.
	This translates to the desired fractional matching of size
	\[
	\frac{|N(v)|}{m'} \geq \left(1-\frac{1 }{\chi_{\crit}} \right) \frac{n}{m'} = \frac{n}{m}\,. \qedhere
	\] 
\end{proof}
\urldef\myurlx\url{https://www.wolframalpha.com/input?i=2-1%2F%281-%281-a%29%2F%28k-1%29%29+-+%281-+%281-b%29%2F%28k-2%29%29+where+b%3D+a%2F%281-%281-a%29%2F%28k-1%29%29}
\COMMENT{\myurlx}

Lastly, we observe that \cref{lem:fractional-komlos} in conjunction with \cref{prp:almost-perfect-matching} and \cref{lem:inheritance-minimum-degree} also recovers \cref{thm:komlos}.

\subsection{Codegree thresholds}\label{sec:mycroft-thresholds}

Next, we show \cref{pro:mycroft-thresholds}.
For $k \geq 3$, consider a $k$-partite $k$-graph $F$ with $\gcd(F)=1$ and parts of size ${m}_1 \leq \dots \leq {m}_k$ with ${m}= {m}_1+\dots+{m}_k$.
We begin with the space property.
Let $U$ contain the $k-1$ largest parts of $F$, and set $F'=F\sm U$.
So $F'$ is a $1$-graph with $m_1$ vertices and $m_1$ edges.
By \cref{lem:linkgraph-threshold}, it follows that $\th^{}_{k-1}(\SpaF{F}) \leq \th_0^{}(\spa_{F'};1/{m}) = {m}_1/{m} = \tau(F)$.
The simple construction for the lower bound can be found in \cref{sec:constructions}.

The divisibility property follows by \cref{cor:completeness-connectivity} in conjunction with the following observation.
 
\begin{proposition}\label{pro:kruskal-katona}
	Let $F$ be a $k$-partite $k$-graph.
	Then $\th_{1}^{} (\con_F) = 2^{-k+1}$ and $\th_{d}^{} (\con_F) =~0$ for every $2 \leq d < k$.
\end{proposition}

\begin{proof}
	The second part of the statement is trivial as in this case any two vertices are on a common edge.
	It remains to show the first part.
	
	The lower bound follows by considering the union of two disjoint cliques of the same order.
	For the upper bound, let $n$ be sufficiently large with respect to $k$ and~$\mu > 0$.
	Suppose~$G$ is an $n$-vertex $k$-graph with $\delta_{1}(G) \geq (2^{-k+1} +\mu) \binom{n-1}{k-1}$.
	Note that a set of $n/2$ vertices may host at most $\binom{n/2}{k-1} \leq  (2^{-k+1}+\mu/2) \binom{n }{k-1}$ distinct $(k-1)$-sets. 
	It follows that the edges incident to each vertex span more than $n/2$ vertices.
	Hence, any two vertices are on intersecting edges, and $G$ must therefore be $k$-connected.
\end{proof}

Finally, recall that a $k$-graph $G$ satisfies $H(F;G) \in \cov$ if for every vertex $v$, there is a homomorphism $\phi \in \hom{F}{G}$ which covers $v$ exactly once (\cref{def:hom-graph,obs:digraph-property-equivalence}).
So for any $k$-partite $k$-graph~$F$, we trivially have $\th_{k-1}^{}(\cov_F)=0$, which finishes the proof of \cref{pro:mycroft-thresholds}.

\subsection{Beyond codegree}

We conclude this section by showing \cref{thm:k-partite}.
The constructions for the lower bounds are similar to the ones of Han, Zang and Zhao~\cite{HZZ17} and can be found in \cref{sec:constructions}.
In the following, we focus on the upper bounds.
We derive the space threshold of \cref{thm:k-partite} by applying a result of Grosu and Hladk{\`y}~\cite{GH12}.
Recall from \cref{sec:tilings-via-linkgraphs} that $\th_0^{}(\SpaF{F};\beta)$ is the edge density that forces an $F$-tiling of size $\beta$.

\begin{theorem}[Grosu--Hladk{\`y}]\label{thm:GH}
	Let $F$ be a bipartite graph with parts of size ${m}_1 \leq {m}_2$.
	Then for any $0< \beta \leq 1/({m}_1+{m}_2)$, we have $\th_0^{}(\SpaF{F};\beta) \leq \max\left\{1-(1-\beta {m}_1)^2,\,(\beta ({m}_1+{m}_2))^2 \right\}.$
\end{theorem}

\begin{lemma}
	For $k \geq 3$, let $F$ be an $m$-vertex $k$-partite $k$-graph with parts of size ${m_1\leq \dots \leq m_k}$.
	Then
	$\th_{k-2}^{}(\SpaF{F}) \leq  \max \left \{ 1-\left(1- m_1/m \right)^{2},\, \left(  (m_1 + m_2)/m\right)^2 \right\}$. 
\end{lemma}

\begin{proof}
	Let $U$ contain the $k-2$ largest parts of $F$, and set $F'=F\sm U$.
	So $F'$ is a bipartite graph with parts of size ${m}_1 \leq {m}_2$.
	Hence, the claim follows by combining \cref{lem:linkgraph-threshold,thm:GH} where the latter is applied with $F'$ and $1/m$ playing the rôle of $F$ and~$\beta$, respectively.
\end{proof}

The upper bound of the divisibility threshold in \cref{thm:k-partite} follows immediately from \cref{pro:kruskal-katona,cor:completeness-connectivity}.
It remains to show the upper bound for the cover property.
Recall from the discussion before \cref{thm:minimum-degree-thresholds} that a $k$-graph $G$ is {$\Hom(F)$-covered} if for every vertex $v \in V(G)$, there is a homomorphism $\phi \in \hom{F}{G}$ with $|\phi^{-1}(v)| =1$.
Moreover, $\cov_F$ denotes the set of $\Hom(F)$-covered hypergraphs and $\th^{}_{d}(\cov_F)$ is the corresponding minimum $d$-degree threshold.

Let $F$ be a $k$-partite cone, meaning that it has a proper $k$-colouring with a singleton colour class.
So in particular, $H(F;G)$ satisfies $\cov$ if every vertex of $G$ is on an edge.
This is trivially the case under any positive minimum $d$-degree.
Hence the cover threshold of \cref{thm:k-partite} follows for cones.
It remains to handle the case when $F$ is not a cone:

\begin{lemma}\label{lem:linkage-upper}
	For every $k$-partite $k$-graph $F$, we have $\th^{}_{k-2}(\cov_F) \leq 2(\sqrt{2}-1)^2$. 
\end{lemma}

We omit the standard, but somewhat tedious proof of the following optimisation problem.

\urldef\myurly\url{https://www.wolframalpha.com/input?i=maximise+d_1%2Bd_2%2Bd_3+for+d_1%2Bd_1%2Bd_2%2Bd_3+%3C%3D+%281-max%28sqrt%28%28d_2%2Bd_3%2F2%29%29%2Cd_1%2Bd_2%2Bd_3%29%29%5E2%2C+0%3C%3Dd_1%3C%3D1%2C+0%3C%3Dd_2%3C%3D1%2C+0%3C%3Dd_3%3C%3D1%2C+d_1%2Bd_2%2Bd_3%3C%3D1}
\COMMENT{\myurly}

\begin{fact}\label{fac:optimisation}
	Suppose $0 \leq \delta,\delta_1,\delta_2,\delta_3 \leq 1$ with $\delta=\delta_1 + \delta_2 + \delta_3$ satisfy 
	\begin{equation*}
		\delta_1 + \delta \leq \left(1-\max \left\{\sqrt{\delta_2 + \delta_3/2},\, \delta\right\}\right)^2. 
	\end{equation*}
	Then $\delta_1 + \delta_2 + \delta_3 \leq 2(\sqrt{2}-1)^2$.
\end{fact}

\begin{proof}[Proof of \cref{lem:linkage-upper}]
	Given a $k$-graph $F$ on $m$ vertices and $\mu>0$, let $\eps,\,n$ with $1/m\,,\mu \gg \eps \gg 1/n$.
	The purpose of $\eps$ is to bound a few errors along the argument.
	Let $G$ be a $k$-graph with $\delta_{k-2}(G) = (\delta +\mu) \binom{n-(k-2)}{2}$ for a $\delta \in [0,1]$.
	Suppose that $G$ is not {$\Hom(F)$-covered}, which means that there is a vertex $x \in V(G)$ that is covered at least twice by every homomorphism $\phi \in \hom{F}{G}$ that covers $x$ at least once.
	So, in particular, there are no edges $e,\,f \in G$ such that $x \in f$, $x\notin e$ and $|f \cap e|=k-1$.
	Our goal is to show that this implies $\delta \leq  2(\sqrt{2}-1)^2$.
	
	Let $S$ be a $(k-2)$-set that contains $x$.
	We obtain a subgraph $G' \subset G$ by deleting edges containing $S$ until the degree of $S$ is precisely $\delta \binom{n-(k-2)}{2}$.
	In other words, the $2$-graph $L(S) = L_{G'}(S)$ contains $ \delta \binom{n-(k-2)}{2}$ edges.
	Let $y$ be a vertex of maximum degree in $L(S)$.
	Denote its neighbourhood by $N = N_{L(S)}(y)$ and let $|N| = \nu n$.
	By averaging, we deduce that
	\begin{equation}\label{equ:linkage-simple-bound-N}
		\nu \geq  \delta -\eps \,.
	\end{equation}
	
	Let $S'$ be the $(k-2)$-set obtained from $S$ by removing $x$ and adding $y$.
	Note that among the edges of $G$ that have been deleted to obtain $G'$, there are at most $n-k+1$ edges that contain both $S$ and $y$.
	So for $L(S') = L_{G'}(S')$, we still have
	\begin{equation}\label{equ:LS'}
		e(L(S')) \geq (\delta+\mu/2) \binom{n}{2} \,.
	\end{equation}

	{We claim that the edge sets of $L(S)$ and $L(S')$ are disjoint.}
	Indeed, an edge $uw$ in both $L(S)$ and $L(S')$ would contradict the assumption on $x$ with $S \cup \{u,w\}$, $S' \cup \{u,w\}$ playing the rôles of $e,f$ respectively.
	A similar argument shows that
	\begin{equation}\label{equ:linkage-L(S')}
		\text{the edges of $L(S')$ do not intersect with $N$.}
	\end{equation}
	
	We partition the edge set of $L(S)$ into parts $L_1$, $L_2$ and $L_3$ by setting $L_1 = L(S) - N$ and $L_2 = (L(S))[N]$ such that $L_3$ contains the edges leaving $N$.
	Let $0 \leq \delta_1,\delta_2,\delta_3 \leq 1$ such that $e(L_i) = \delta_i \binom{n-(k-2)}{2}$ for $1 \leq i \leq 3$.
	Note that {$\delta   =\delta_1 +\delta_2+\delta_3.$}

	We now derive another lower bound for $\nu$.
	Averaging reveals a vertex $u \in N$ with $\deg_{L(S)}(u) \geq |N|^{-1} \sum_{w \in N} \deg_{L(S)} (w)$.
	So by maximality of $y$, it follows that
	\[
	|N| = |N_{L(S)}(y)| \geq |N|^{-1}\sum_{w \in N} \deg_{L(S)} (w)\,.
	\]
	Moreover $\sum_{w \in N} \deg_{L(S)} (w) \geq (2\delta_2 + \delta_3) \binom{n-(k-2)}{2}$.
	Putting this together, we obtain $|N|^2 \geq (2\delta_2 + \delta_3) \binom{n-(k-2)}{2}$, which translates to 
	\begin{equation}\label{equ:linkage-advanced-bound-N}
		\nu \geq  \sqrt{\delta_2 + \delta_3/2 - \eps}\,.
	\end{equation}
	
	Recall that the edge sets of $L(S)$ and $L(S')$ are disjoint and $|N|=\nu n$.
	Together with observations~\eqref{equ:LS'} and~\eqref{equ:linkage-L(S')}, this gives
	\[
		\delta_1  \binom{n}{2} + (\delta+\mu/2)\binom{n}{2} \leq e(L_1 \cup L(S'))  \leq \binom{n-|N|}{2} \leq (1-\nu)^2  \binom{n}{2}\,.
	\]	
	In combination with the bounds \eqref{equ:linkage-simple-bound-N} and \eqref{equ:linkage-advanced-bound-N}, we obtain
	\[
		\delta_1 + \delta + \mu/2 \leq (1-\nu)^2 \leq \left(1-\max \left\{\sqrt{\delta_2 + \delta_3/2 - \eps},\,\delta -\eps\right\}\right)^2\,.
	\]
	By the choice of $\eps$, it follows that
	\[
	\delta_1 + \delta \leq \left(1-\max \left\{\sqrt{\delta_2 + \delta_3/2},\,\delta\right\}\right)^2\,.
	\]
	So \cref{fac:optimisation} implies that $\delta \leq 2(\sqrt{2}-1)^2$, as desired.
\end{proof}

\section{Transversal tilings -- Proofs}\label{sec:transversal-proof}

We prove \cref{thm:transversal-tiling} by transforming the transversal tiling problem into a `bipartite' matching problem for directed hypergraphs and then applying our (partite) framework.
To formalise this transition, we introduce some ad-hoc notation.

Consider a $k$-graph $F$ with $V(F)=\{v_1,\dots,v_{m}\}$ and $E(F) = \{e_1,\dots,e_h\}$.
Let $\cG=\{G_1,\dots,G_{hn}\}$ be a family of $k$-graphs on the same ${m}n$ vertices $V$.
Let $F'$ be the $(k+1)$-graph on vertex set $V(F) \cup E(F)$ with an edge $f \cup \{f\}$ for each $f \in E(F)$.
Let~$G'$ be the $(k+1)$-graph on vertex set $V \cup [hn]$ with an edge $e \cup \{j\}$ whenever $e \in E(G_j)$ for colour $1 \leq j \leq hn$.
We define $H(F;\cG)$ as the $({m}+h)$-digraph with an $({m}+h)$-edge $(\phi(v_1),\dots,\phi(v_{m}),\phi(e_1),\dots,\phi(e_h))$ for every homomorphism $\phi \in \hom{F'}{G'}$ with $\phi(V(F)) \subset V $ and $\phi(E(F)) \subset [hn]$.
Fissility in this setting follows along the lines of the proof of \cref{obs:hom-graph-fissile}.
We include the details for the sake of completeness.

\begin{claim}\label{cla:hom-graph-transversal-fissile}
	$H(F;\cG)$ is fissile. 
\end{claim}

\begin{proof}
	Set $H = H(F;\cG)$, and consider a blow-up $R(\cV) \subset H$ with a family $\cV=\{V_x\}_{x \in V(R)}$ and an ample $m$-edge $f \in R$.
	We show that $H$ contains an $f$-preserving {injective} edge.
	Let $e$ be a strictly $f$-preserving edge in $R(\cV)$.
	So $e$ corresponds to a homomorphism $\phi_e \in \hom{F'}{G'}$ with $\phi_e(V(F)) \subset V $ and $\phi_e(E(F)) \subset [hn]$ by definition of $H$.
	Let $\phi \in \hom{F'}{R}$ be the corresponding homomorphism, which maps each vertex $u \in V(F')$ to the vertex $x \in V(R)$ for which $\phi_e(u) \in V_x$.
	
	Now consider an edge $u_1 \dots u_{k+1}  \in F'$ with $u_i \in V(F)$ for $i \in [k]$ and  $u_{k+1} =  u_1 \dots  u_{k} \in E(F)$, and let $\phi(u_i) = x_i$ for $i \in [k+1]$.
	We claim that every set $\{w_1,\dots,w_{k+1}\}$ with $w_i \in V_{x_i}$ is an edge in $G'$.
	To see this, define a function $\phi' \colon V(F') \to V(G')$ by setting $\phi'(u) = w_i$ if $\phi_e(u) \in V_{x_i}$ and $\phi'(u) = \phi_e(u)$ otherwise.
	Note that $\phi'$ is in fact a homomorphism, since its corresponding $m$-tuple is strictly $f$-preserving in $R$ and $R(\cV) \subset H$.
	Since homomorphisms map edges to edges, it follows that $w_1 \cdots w_{k+1} \in G'$ as claimed.
	
	Since $f$ is ample, we  can thus select for each $x = f(j)$ with $j \in [m+h]$, a number of $|\phi^{-1}(x)| = \multi(x,f)  \leq |V_{x}|$ vertices in each cluster $V_{x}$ to find a copy of $F'$ in $G'$ on exactly these vertices.
	This results in an $f$-preserving {injective} edge in $H$.
\end{proof}

Recall the definition of proportionality from the discussion before \cref{obs:lattice-completeness}, and note that $H(F;\cG)$ is $\cU$-proportional for $\cU = \{V,[hn]\}$.
We remark that, crucially, every transversal perfect $F$-tiling of $\cG$ corresponds to a  perfect matching in $H(F;\cG)$ and vice versa.

To prove \cref{thm:transversal-tiling}, we define the following auxiliary thresholds.
For a $k$-graph $F$ on ${m}$ vertices and $h$ edges, let $\th_d^{}(\rspa_F)$, $\th_d^{}(\rdiv_F)$ and $\th_d^{}(\rcov_F)$ be the infimum over all $\delta \in [0,1]$ such that for all $\mu>0$ and $n$ large enough and every family of $k$-graphs $\cG=\{G_1,\dots,G_{hn}\}$ on the same ${m}n$ vertices $V$ with $\delta_d(G_j) \geq (\delta +\mu) \binom{mn-d}{k-d}$ for each $j \in [hn]$, the $({m}+h)$-digraph $H(F;\cG)$  satisfies $\spa$, $\pdiv(\cU)$ and $\cov$, respectively, where $\cU = \{V,[hn]\}$.
As before, we may use our framework to derive a threshold decomposition.

\begin{corollary}\label{cor:transversal}
	For every $k$-graph $F$ and $1\leq d < k$, we have
	\begin{equation*}
		\th_d^{}(\rtil_F) = \max \left\{	\th_d^{}(\rspa_F),\,	\th_d^{}(\rdiv_F),\,	\th_d^{}(\rcov_F) \right\}\,.
	\end{equation*}
\end{corollary}

Given this, \cref{thm:transversal-tiling} follows from the next proposition and \cref{thm:minimum-degree-thresholds}.

\begin{proposition}\label{pro:transversal-thresholds}
	For every $k$-graph $F$ and $1 \leq d < k$, we have
	\begin{align*}
		\th_d^{}(\rspa_F) &= \th_d^{}(\SpaF{F})\,, \\ 
		\th_d^{}(\rdiv_F) &= \max\left\{\th_d^{}(\div_F),\, \th_d^{}(\rmix_F)\right\}\,, \\ 
		\th_d^{}(\rcov_F) &= \max\left\{\th_d^{}(\cov_F),\, \th_d^{}(\rmix_F)\right\}\,.
	\end{align*}
\end{proposition}

The proof of \cref{cor:transversal} is fairly similar to the proof of \cref{thm:framework}.
Since the setting is somewhat unfamiliar, we spell out the details.

\begin{proof}[Proof of \cref{cor:transversal}] 
	For $2 \leq k \leq m$, let $F$ be a $k$-graph on $m$ vertices and $h$ edges.
	Set  $\delta = \max \left\{	\th_d^{}(\rspa_F),\,	\th_d^{}(\rdiv_F),\,	\th_d^{}(\rcov_F) \right\}$.
	Consider $\mu, \alpha,s',\eta, n$ with \[
	\frac{1}{m},\, \mu \gg  \alpha,\, \frac{1}{s'}  \gg \eta \gg \frac{1}{n'}\,.
	\]
	Set $m' = m+h$, $n' = m' n$ and $s' = m' s$.
	Let $\cG=\{G_1,\dots,G_{hn}\}$ be a family of $k$-graphs on the same ${m}n$ vertices $V$ with $\delta_d(G_j) \geq (\delta +\mu) \binom{{m}n-d}{k-d}$.
	Set $\cU = \{V,[hn]\}$.
	Our goal is to find a transversal a perfect $F$-tiling in $\cG$, which is equivalent to finding a perfect matching in $H = H(F;\cG)$.
	
	We begin by selecting a suitable absorption structure.
	For each $j \in [hn]$, let $P_j = \PG{G_j}{\P}{ms}$ be the property $ms$-graph for the $\P = \DegF{d}{\delta+\mu/2}$.
	So $P_j$ contains an edge $S$ if $\delta_d(G_j[S]) \geq (\delta+\mu/2) \binom{ms-d}{k-d}$.
	It follows by \cref{lem:inheritance-minimum-degree} applied with $m$, $mn$ playing the rôles of $q$, $n$ that $\delta_{m}(P_j) \geq  (1-e^{-\sqrt{ms}})  \tbinom{mn-m}{ms-m}$.
	Let $P$ be the $s'$-graph on vertex set $V(H)$, which contains a $\cU$-proportional $s'$-edge $S$ if $S \cap V \in P_j$ for every colour $j \in S \cap [hn]$.
	Consider a  $\cU$-proportional $m'$-set $X \subset V(H)$.
	So $X$ has $m$ vertices in $V$ and $h$ vertices in $[hn]$.
	By taking the intersection of the $ms$-graphs $P_j$ with $j \in X \cap [hn]$, we find that $\deg_P(X) \geq   (1-hs e^{-\sqrt{ms}})  \binom{n'-m'}{s'-m'}  \geq \mu  {n'}^{s'-m'}$.
 
	Consider an $S \in P$, and let $S' \subset S$ be obtained by deleting $m'$ vertices.
	Since $\mu$ is much larger than $1/s$, we have $\delta_d(G_j[S' \cap V]) \geq (\delta +\mu/4) \binom{{m}s-d}{k-d}$ for every $j \in S' \cap [hn]$.
	So $H[S']$ satisfies $  \div(\cU) \cap \cov $ by definition of $\delta$.
	It follows that~$H[S]$ satisfies $\Del_m(\div(\cU) \cap \cov)$.
	This allows us to apply \cref{prp:absorption-partite} with $n'$, $m'$, $2\eta$ playing the rôles of $n$, $m$, $2\eta$, which provides a set $A \subset V(H)$ with $|A| \leq \alpha n'$ such that $H[A \cup L]$ has a perfect matching for every $\cU$-proportional set $L \subset V(H-A)$ with $|L| \leq 2\eta n'$ divisible by~$m'$.
	Note that $A$ is itself $\cU$-proportional, which can be seen by taking $L=\es$.

	Next, we cover most of the vertices of $H-A$ with a matching.
	By the same argument as above (taking $hs$ intersections now), we deduce that $\delta_1(P) \geq   (1- hs e^{-\sqrt{ms}})  \binom{n'-1}{s'-1}$.
	Let $\cV$ be a random refinement of the partition~$\{V \sm A,[hn] \sm A\}$ into $s'$ parts of size $n''$.
	(In case there are divisibility issues, we temporarily add some dummy vertices.)	
	Let $P' \subset P$ contain the $\cV$-partite edges of $P$.
	A standard concentration argument shows that $\delta_1(P') \geq (1-1/s'+\mu) n''^{s'-1}$.
	Recall that $H[S]$ satisfies $\spa$ for each $S \in P$ by definition of $P$ and~$\delta$.
	We may therefore apply \cref{prp:almost-perfect-matching-partite} with $\rho=0$ to find a matching $M$ in $H$ covering all but $2\eta n'$ vertices (after accounting for the possible removal of dummy vertices). 
	
	It follows that the set of uncovered vertices $L = V(H) \sm (V(M) \cup A)$ is $\cU$-proportional and of size  $|L| \leq 2 \eta n'$ divisible by $m'$.
 	Owing to the absorption property of $A$, we may extend~$M$ to a perfect matching of $H$.
\end{proof}

It remains to show \cref{pro:transversal-thresholds}.

\begin{proof}[Proof of \cref{pro:transversal-thresholds}]
	Let $F$ be a $k$-graph on ${m}$ vertices and $h$ edges.
	For $\delta$, $\mu>0$ and $n$ large enough, consider a family $\cG=\{G_1,\dots,G_{hn}\}$ of $k$-graphs on the same $mn$ vertices~$V$ with $\delta_d(G_j) \geq (\delta +\mu) \binom{{m}n-d}{k-d}$ for every $j \in [hn]$.
	Let $H = H(F;\cG)$ and  $H_j = H(F;G_j)$ for each $j \in [hn]$.
	By the definition of homomorphisms,
  	\begin{equation}\label{obs:stepping-up}
		\begin{minipage}[c]{0.9\textwidth}
			$H$ contains $(v_1,\dots,v_{m},j,\dots,j)$ for every colour $j\in [hn]$ and edge $(v_1,\dots,v_{m})$ in $H_j$.
		\end{minipage}\ignorespacesafterend 
 	\end{equation}
	
	\subsubsection*{Space}
	
	We have to show that $H$ contains a perfect fractional matching under the assumption $\delta \geq \th_d({\spa_F})$.
	In this case each $H_j$ has a perfect fractional matching $w_j$.
	For an edge $f=(v_1,\dots,v_{m})$ in $H_j$, we set $e_f = (v_1,\dots,v_{m},j,\dots,j)$ and note that $e_f \in  H $ by observation~\eqref{obs:stepping-up}.
	It follows that for distinct colours $j,\ell \in [hn]$   and $v \in V$, we have $\sum_{f \in H_j} w_j(f) \multi(j,e_f) =hn$, $\sum_{f \in H_j} w_j(f) \multi(\ell,e_f) = 0$ and $\sum_{f \in H_j} w_j(f) \multi(v,e_f) =~1$.
	We can therefore define a perfect fractional matching $w' \colon H \to [0,1]$ by setting for each $e \in  H $, the weight $w'(e)$ to be the sum of $w_j(f)/(hn)$ over all  $j \in [hn]$ and $f \in  H_j $ with  $e=e_f$.
	
	\subsubsection*{Divisibility}
	
	We obtain the divisibility property via lattice completeness (\cref{sec:lattice-completeness}).
	Our goal is to show that the lattice $\cL (H)  \subset \INTS^{V \cup [hn]}$ is $\{V,[hn]\}$-complete given that $\delta$ is at least $\max\{\th_d^{}(\div_F),\, \th_d^{}(\rmix_F)\}$.
	By \cref{obs:lattice-completeness-partite}, it suffices to show that $\cL(H)$ contains the {transferral} $\vn_v - \vn_u$ for all elements $u,v$ coming from the same part of $\{V,[hn]\}$.
	We begin with the following observation.
	\begin{claim}\label{cla:transversal-divisibility}
		For every colour $j \in [hn]$ and $\mathbf {b} \in \cL(H_j) \subset~\INTS^{V}$ with $\sum_{v \in V} \mathbf {b}(v)=0$, there is  $\mathbf {b}' \in \cL(H) \subset \INTS^{V \cup [hn]}$ with $\mathbf {b'}(v)=\mathbf {b}(v)$ for all $v \in V$ and  $\mathbf {b}' (i)=0$ for each colour $i \in [hn]$.
	\end{claim}
	
	\begin{proofclaim}
		Without loss of generality, suppose that $j=1$.
		Write $\mathbf {b}= \sum_{e \in H_1}  c_e  \vn_{e}$.
		For $e=(v_1,\dots,v_m) \in H_1$, let $e' = (v_1,\dots,v_m,1,\dots,1)$ be obtained from $e$ by concatenating~$h$ times colour~$1$.
		So $ \vn_{e'}(1)  = h$.
		Given this, set $\mathbf {b}'= \sum_{e \in H_1}  c_e \vn_{e'}$.
		Note that $\mathbf{b}' \in \cL(H) $ by observation~\eqref{obs:stepping-up}.
		Clearly, $\mathbf {b}'(v)=\mathbf {b}(v)$ for all $v \in V$.
		Moreover, $\mathbf {b}'(i)=0$ for every colour $i \neq 1$.
		It remains to show that  $\mathbf {b}'(1)  = 0$.
		To this end, note that $\mathbf {b}'(1) =  \sum_{e \in H_1}  c_e  \vn_{e'}(1)  = \sum_{e \in H_1} c_e  h$.
		Moreover,
		\begin{equation*}
			\sum_{e \in H_1}  c_e  m = \sum_{e \in H_1}  c_e \sum_{v \in V} \vn_{e}(v)  = \sum_{v \in V} \sum_{e \in H_1}  c_e   \vn_{e}(v)  = \sum_{v \in V}    \mathbf {b}(v) = 0 \,.
		\end{equation*}
		Together, this gives $\mathbf {b}'(1)  = 0$ as desired.
	\end{proofclaim}
	 
	Now consider any two vertices $u,v \in V$.
	We argue that lattice $\cL(H_1) \subset  \INTS^{V}$ is complete.
	By \cref{obs:robust-int-matching-gives-lattices-completeness}, it suffices to show that $H_1 - D \in \div$ for every $m$-set $D \subset V(H_1)$.
	This follows by definition of $\delta \geq  \th_d^{}(\div_F)$, as $G_1 - D$ still has minimum $d$-degree at least $(\delta +\mu/2) \binom{{m}(n-1)-d}{k-d}$. 
	So  $\cL(H_1)$ is indeed complete and contains, in particular, the transferral $\mathbf b = \vn_v - \vn_u \in \INTS^{V}$.
	By \cref{cla:transversal-divisibility}, it follows that $\cL(H)$ therefore also contains the transferral $\mathbf b' =  \vn_v - \vn_u \in \INTS^{V \cup [hn]}$, as desired.
		
	To finish, we have to show that $\cL(H)$ contains the transferral between two arbitrary colours.
	Without loss of generality, it suffices to check that  $\vn_1 - \vn_{2} \in \cL (H)$.
	By the definition of $\th_d^{}(\rmix_F)$, there is an $({m}+h)$-edge $e_1 \in H$ with $\multi(1,e_1)=1$ and $\multi(3,e_1)=h-1$.
	Let~$e_2$  be the analogous $({m}+h)$-edge for colours $2$ and $3$.
	Then $\mathbf {b} ' = \vn_{e_1} - \vn_{e_2} \in \cL(H)$ satisfies $\vecb b'(1) = 1$, $\vecb b'(2) = -1$ and $\vecb b'(j) = 0$ for all $j \in [hn] \sm [2]$.
	Let $\vecb b   \in \INTS^{V}$ be the restriction of $\vecb b'$ to~$V$.
	So $\vecb b = \vn_{\phi_1(V(F))} - \vn_{\phi_2(V(F))}$, where $\phi_i$ is the homomorphism corresponding to $\vecb b_i$.
	Note that the lattice $\cL(H_3)$ is complete for the same reason that $\cL(H_1)$ is complete.
	It follows that $\vecb {b} \in \cL(H_3)$.
	Note that
	\begin{equation*}
	 	\sum_{v \in V} \vecb b (v) = \sum_{v \in V} \vecb b'(v)= \sum_{v \in V}  \vn_{e_1}(v) - \sum_{v \in V} \vn_{e_2}(v) = m - m  = 0\, .
	\end{equation*}
	Hence \cref{cla:transversal-divisibility} allows us to find a $\vecb b'' \in \cL(H)$ with $\mathbf {b}''(v)=\mathbf {b}(v) =\mathbf {b}'(v)$ for all $v \in V$ and  $\mathbf {b}'' (i)=0$ for each colour $i \in [hn]$.
	It follows that $\vn_1 - \vn_2 = \vecb b' - \vecb b''$ as desired.
	
	\subsubsection*{Cover}
	
	Finally, we verify the cover property for $\delta \geq \max\{\th_d^{}(\cov_F),\, \th_d^{}(\rmix_F)\}$.
	Consider first a vertex $v \in V.$
	Since $H_1$ is covered, there is an $m$-edge $f \in H_1$ with $\multi(v,f)=1$.
	By observation~\eqref{obs:stepping-up}, the $(m+h)$-edge $e_f$ satisfies $\multi(v,e_f)=1$.
	Next, consider an arbitrary colour, say colour $1$.
	By the definition of $\th_d^{}(\rmix_F)$, there is an $(m+h)$-edge $e_1 \in H$ with $\multi(1,e_1)=1$ and $\multi(2,e_1)=h-1$.
	So we are done.
\end{proof}

Recall from \cref{pro:komlos-kuhn-osthus-thresholds} that $\th_{1}^{}(\cov_F) \geq 1-{1}/(k-1)$.
Moreover, we have ${\th_{1}^{}(\til_F)\leq1/2}$ for $\chi(F)=2$ as mentioned after \cref{thm:kuhn-osthus}.
Hence, the following observation recovers \cref{thm:montgomery-muyesser-pehova} from \cref{thm:transversal-tiling}.

\begin{proposition}\label{obs:transversal}
	Let $F$ be a graph with $k =\chi(F) \geq 2$.
	\begin{itemize}
		\item If $F$ is bridgeless, then $\th_1^{}(\rmix_F) = \max\{1-1/(k-1),\, 1/2\}$.
		\item If $k=2$ and $F$ has a bridge, then $\th_1^{}(\rmix_F)=0$.
	\end{itemize}
\end{proposition}

\begin{proof}
	For $\delta \in [0,1]$, consider $n$-vertex graphs $G_j$ with minimum degree $\delta(G_j) \geq \delta n + 2$ for $j \in [2]$.
	First note that when $k=2$ and $F$ has a bridge $f$, we can easily map $f$ to an edge $uv$ of $G_1$ and the rest of $F$ into two edges of $G_2$, one attached to $u$ and one attached to $v$.
	Hence $\th_1^{}(\rmix_F)=0$.
	
	When $F$ has no bridge, we can assume that $\delta = \max\{1-1/(k-1),\, 1/2\}$.
	Note that in this case $G_1$ and $G_2$ must in particular share an edge $e$.
	(This is because every set of $k-1$ vertices has a common neighbour.)
	Moreover, it is not hard to see that every edge of $G_2$ is contained in a $k$-clique of $G_2$.
	This gives the upper bounds.
	
	For the lower bounds, it suffices to take $G_1,G_2,\dots$ as the Turán graph for~$k\geq 3$.
	For $k=2$, we can take the construction from \cref{sec:transversal}.
	More precisely, let $G_1=K_{n/2,n/2}$, and let the graphs $G_2,G_3,\dots$ be its complement.
	Then any copy of $F$ with an edge in $G_1$ must have at least two edges of $G_1$ as $F$ has no bridges, which makes a transversal perfect $F$-tiling impossible.
\end{proof}

\section{Ordered graphs -- Proofs}\label{sec:ordered-graphs-proofs}

In this section, we derive \cref{cor:threshold-decomposition-ordered} and use this to recover the work of Freschi and Treglown on perfect tilings in ordered graphs by deriving \cref{pro:thresholds-ordered}.
Recall from \cref{def:hom-graph-ordered} that $H^{\Ord}(F;G)$ is the subgraph corresponding to order-preserving homomorphisms, and thus a subgraph of the (standard) auxiliary digraph $H(F;G)$ from \cref{def:hom-graph}.
We begin with the following analogue of \cref{obs:hom-graph-fissile}, whose proof is almost identical.

\begin{observation}\label{obs:ordered-hom-graph-fissile}
	$H^{\Ord}(F;G)$ is fissile for ordered $k$-graphs $F$ and $G$.
\end{observation}

\begin{proof}
	Set $H = H^{\Ord}(F;G)$, and consider a blow-up $R(\cV) \subset H$ with a family $\cV=\{V_x\}_{x \in V(R)}$ and an ample $m$-edge $f \in R$.
	We show that $H$ contains an $f$-preserving {injective} edge.
	Let $e$ be a strictly $f$-preserving edge in $R(\cV)$.
	So $e$ corresponds to an order-preserving homomorphism $\phi_e \in \hom{F}{G}$ by definition of $H$.
	Let $\phi \in \hom{F}{R}$ be the corresponding order-preserving homomorphism, which maps each vertex $v \in V(F)$ to the vertex $x \in V(R)$ for which $\phi_e(v) \in V_x$.
	
	Now consider an edge $v_1\cdots v_k \in F$, and let $\phi(v_i) = x_i$ for $1 \leq i \leq k$.
	We claim that every set $\{w_1,\dots,w_k\}$ with $w_i \in V_{x_i}$ is an edge in $G$.
	To see this, define a function $\phi' \colon V(F) \to V(G)$ by setting $\phi'(v) = w_i$ if $\phi_e(v) \in V_{x_i}$ and $\phi'(v) = \phi_e(v)$ otherwise.
	Note that $\phi'$ is in fact an order-preserving homomorphism, since its corresponding $m$-tuple is strictly $f$-preserving in $R$ and $R(\cV) \subset H$.
	Since homomorphisms map edges to edges, it follows that $w_1 \cdots w_k \in G$ as claimed.
	
	Since $f$ is ample, we  can thus select for each $x = f(j)$ with $1 \leq j \leq m$, a number of $|\phi^{-1}(x)| = \multi(x,f)  \leq |V_{x}|$ vertices in each cluster $V_{x}$ to find a copy of $F$ in $G$ on exactly these vertices.
	This results in an $f$-preserving {injective} edge in $H$.
\end{proof}

We can now derive \cref{cor:threshold-decomposition-ordered} from \cref{thm:framework-exceptional} in the same way that we deduced \cref{thm:minimum-degree-thresholds} from \cref{thm:framework}.

\begin{proof}[Proof of \cref{cor:threshold-decomposition-ordered}]
	Given a $k$-graph $F$ on $m$ vertices and $1 \leq d \leq k-1$, set
	\[\delta = \max \big\{	\th_d^{}(\ospa_F),\,	\th_d^{}(\odiv_F),\,	\th_d^{}(\ocov_F)\,\big\}\,.\]
	Introduce constants with $1/{m},\,\mu \gg 1/r \gg \rho \gg 1/s \gg 1/n$ such that $r$ and $n$ are divisible by $m$.
	Set $\alpha = \mu/4$.
	Let $G$ be a $k$-graph on $n$ vertices with $\delta_d(G) \geq (\delta+\mu) \binom{n-d}{k-d}$.
	Set $H=H^{\Ord}(F;G)$, and note that $H$ is fissile by \cref{obs:ordered-hom-graph-fissile}.
	
	It follows by \cref{lem:inheritance-minimum-degree} that the property $r$-graph $P_r=\PG{G}{\DegF{d}{\delta+\mu/2}}{r}$ satisfies $\delta_1(P_r) \geq (1-1/r^2) \tbinom{n-1}{r-1}$.
	By the definition of $\delta$, the $m$-digraph $H[S]$ satisfies $\div \cap \cov$ for every $S \in P_r$.
	So, in particular, $\delta_1\left(\PG{H}{\div \cap \cov}{r}\right) \geq  \left(1-1/r^2 \right)  \binom{n-1}{r-1}.$
	
	Next, let $A \subset V(G)$ have size at most $\alpha n$.
	Let $P_s =\PG{H-A}{\P}{s}$ be the property $s$-graph with $\P = \DegF{d}{\delta+\mu/4}$.
	Since $\alpha = \mu /4$, we still have $\delta_d(G-A) \geq (\delta+\mu/2) \tbinom{\bar n-d}{k-d}$, where $\bar n = n-|A|$.
	By applying \cref{lem:inheritance-minimum-degree} to $P_s$, we thus obtain $\delta_1\left(\PG{H-A}{\SpaF{\rho}}{s}\right) \geq  \left( 1-1/s^2 \right)  \binom{\bar n-1}{s-1}$ in the same way as before. 
	So $G$ contains a perfect $F$-tiling by \cref{thm:framework-exceptional}.
\end{proof}

It remains to show \cref{pro:thresholds-ordered}.
We begin with the threshold for space.
The next lemma was shown by Balogh, Li and Treglown~\cite{BLT22}.
Here we give a simple proof avoiding the use of a Regularity Lemma.

\begin{lemma}\label{lem:ordered-space-threshold}
	Every ordered graph  $F$ satisfies $\th_{1}^{}(\ospa_F^{}) \leq  1-{1}/{{\chi^<_{\crit}(F)}}$.
\end{lemma}
 
\begin{proof}
	Introduce constants with $\mu \gg 1/b' \gg 1/b \gg 1/n$.
	Let $G$ be an $n$-vertex graph with $\delta_1(G) \geq (1- {1}/{{\chi^<_{\crit}(F)}} + \mu) n$.
	Let $B$ be an (unordered) $k$-partite bottle graph of $F$ with $\chi_{\crit} (B) = \chi^<_{\crit} (F)$.
	So $B$ is a complete $k$-partite unordered graph such that for every permutation $\sigma \in \cS_k$ the blow-up $B(b)$ ordered by $\sigma$ contains a perfect $F$-tiling.
	By  \cref{thm:komlos} we may cover all but $(\mu/2) n$ vertices of $G$ with copies of the blow-up $B(b)$.
	Let $B^\ast$ be one of these copies on vertex set $U$, whose partition we denote by $B_1,\dots,B_k$.
	We have to cover all but $(\mu/2)|U|$ vertices of $B^\ast$.\footnote{We may also directly find a fractional perfect matching in $H^{\Ord}(F;B)$ by a brief meditation on the definitions.
	}
	Colour the $k$-cliques induced in $B^\ast$ with one of~$k!$ colours depending on the ordering of their vertices.
	Using \cref{thm:erd64} on the most common colour, we can find sets $V_i \subset B_i$ with $|V_i|/|V| = |B_i|/|U|$ for $i \in [k]$ and $V=V_1 \cup \dots V_k$ such that $G[V]$ contains a copy of $B({b'})$ in which all $k$-cliques have the same colour.
	By assumption, we can tile this copy of $B({b'})$ with $F$.
	We continue in this fashion until \cref{thm:erd64} no longer applies.
	For sufficiently large $b$, this leaves at most $(\mu/2) |U|$ vertices of~$B^\ast$ uncovered.
\end{proof}

For the divisibility threshold of \cref{pro:thresholds-ordered}, we borrow the notion of `flexibility' introduced by Freschi and Treglown~\cite{FT22}.
Let $F$ be an ordered graph with $r = \chi^<(F)$. 
For $j \in \{0,1\}$, we say that $F$ is \emph{$j$-flexible} if, for every $i \in [r + j - 1]$, there exists an interval $(r + j)$-colouring $V_1, \dots , V_{r+j}$ of $F$ with a vertex $v \in V_i$ such that $V_1,\, \dots,\, {V_i \sm \{v\}},\, {V_{i+1}\cup\{v\}} ,\, \dots ,\,  V_{r+j}$ is an interval $(r+j)$-colouring of $F$.
It is easy to see that we are always guaranteed $1$-flexibility.

\begin{observation}\label{obs:flexible-simple}
	Every ordered graph has a $1$-flexible colouring.
\end{observation}

\begin{proof}
	Let $F$ be an ordered graph with $r = \chi^<(F)$, and fix $i \in [r]$.
	By assumption, $F$ has an interval $(r+1)$-colouring $\psi$ with colour classes $V_1,\dots,V_{r+1}$, where $V_{i+1}$ is empty.
	Let $v$ be the largest vertex in $V_i$.
	Then $V_1 ,\, \dots ,\, V_i \sm \{v\} ,\, V_{i+1} \cup \{v\} ,\, \dots ,\, V_{r+1}$ also is an interval $(r+1)$-colouring of $F$, as desired.
\end{proof}

We can use this to obtain the following analogue of \cref{lem:connectivity-to-completeness}.
Recall the definition of $\ell$-connectivity from \cref{sec:completeness} and the definition of lattice completeness (\cref{sec:lattice-completeness}).

\begin{observation}\label{obs:connectivity-to-completeness-connected-simple}
	Let $F$ be a graph with $r = \chi^<(F)$, and  let $G$ be an $(r+1)$-connected ordered graph.
	Then $H^{\Ord}(F;G)$ has an integral perfect matching.
\end{observation}

\begin{proof}
	Suppose without loss of generality that $[r+1]$ induces a clique in $G$.
	Fix an $i \in [r]$.
	By \cref{obs:flexible-simple}, we may find edges $e,e' \in H^{\Ord}(F;G)$ such that $\vn_{e} - \vn_{e'} = \vn_{i} - \vn_{i+1}$.
	So the lattice $\cL =\cL(H^{\Ord}(F;G))$ contains $\vn_{i} - \vn_{i+1}$, and by extension $\vn_{i} - \vn_{j}$ for every $j \in [r+1]$.
	Since~$G$ is $(r+1)$-connected, we may navigate along the $(r+1)$-cliques to discover that $\cL$ contains, in fact, all transferrals and is thus complete by \cref{obs:lattice-completeness}.
	It follows in particular that $H^{\Ord}(F;G)$ has an integral perfect matching.
\end{proof}

The next lemma due to Freschi and Treglown~\cite{FT22} gives a criterion for $0$-flexibility and is key to showing the divisibility part of \cref{pro:thresholds-ordered}.
Its proof is short and self-contained, so there is no need to reproduce it here.

\begin{lemma}[{Ordered Downspin Lemma~\cite[Corollary 8.6]{FT22}}]\label{lem:flexible}
	Let $F$ be an ordered graph.
	If $\chi^<_{\crit} (F) < \chi_{\crit} (F)$, then $F$ has a $0$-flexible colouring.
\end{lemma}

Given this, we may derive the following result following along the lines of the proof of \cref{obs:connectivity-to-completeness-connected-simple}.
We omit the details.

\begin{lemma}\label{lem:connectivity-to-completeness-connected}
	Suppose $F$ is an ordered graph with $r=\chi^<(F)$ and a $0$-flexible colouring, and suppose $G$ is an $r$-connected ordered graph.
	Then $H^{\Ord}(F;G)$ has an integral perfect matching.
\end{lemma}

Now we are ready to bound the threshold for divisibility defined just after \cref{def:hom-graph-ordered}.

\begin{lemma}
	Let $F$ be an ordered graph with $k = \chi^<(F)\geq 3$.
	Then
	\begin{align*} 
		\th_{1}^{}(\odiv_F)  &\leq \begin{cases}
			1- {1}/{(k-1)} & \text{if $ \chi^<_{\crit}(F)< \chi^<(F)$}\,,
			\\ 1- {1}/{k} & \text{otherwise}. 
		\end{cases}
	\end{align*}
\end{lemma}

\begin{proof}
	Suppose that $G$ is an $n$-vertex ordered graph with $\delta_1(G) > (1-1/r)n$ for an integer $r \geq 2$ and $n$ large enough.
	Consider an arbitrary edge $ij$ in $G$.
	It is easy to see that $ij$ is contained in a $K_{r+1}$-clique.
	(Every set of $r$ vertices has a common neighbour.)
	Since $r\geq2$, this implies in particular that $G$ is $(r+1)$-connected.
	Now if $r=\chi^<(F)-1$ and $\chi^<_{\crit}(F)< \chi^<(F)$, we are done by \cref{lem:flexible,lem:connectivity-to-completeness-connected}.
	On the other hand, if $r=\chi^<(F)$, then we are done by \cref{obs:connectivity-to-completeness-connected-simple}.
\end{proof}

Lastly, we bound the threshold for the covering part of \cref{pro:thresholds-ordered}.

\begin{lemma}
	Let $F$ be an ordered graph.
	Then
	\begin{align*}
		\th_{1}^{}(\ocov_F) &\leq \begin{cases}
			1-\frac{1}{\chi^<(F)-1} & \text{if $F$ is an ordered cone,} 
			\\ 1-\frac{1}{\chi^<(F)} & \text{otherwise}.
		\end{cases}
	\end{align*}
\end{lemma}

\begin{proof}
	Given $\mu > 0$ and $r\geq 2$, suppose that $G$ is an $n$-vertex ordered graph with $\delta_1(G) \geq (1-1/r +\mu)n$  for $n$ large enough.
	Fix a vertex $i \in [n]$.
	Our goal is to find an order-preserving homomorphism $\phi \in \hom{F}{G}$ that covers $i$ exactly once.
	By the minimum degree condition, every set of $r$ vertices must have a common neighbour.
	Hence, there is a vertex $j \neq i$ and an $(r+2)$-set $S \subset V(G)$ with $i,j \in S$ such that $G[S]$ contains all edges but (possibly) $ij$.
	Without loss of generality, we can assume that $S$ equals $[r+2]$.
	
	We first handle the case where $F$ is an ordered cone and $r=\chi^<(F)-1$.
	So there is an interval $(r+2)$-colouring of $F$ with colour classes $V_1,\dots,V_{r+2}$ such that $V_i$ is a singleton class, and there is no edge between $V_i$ and $V_j$ in $F$.
	Hence we can define $\phi$ by sending the vertices of~$V_k$ to $k$ for each $k \in [r+2]$.
	
	Now suppose that $r=\chi^<(F)$.
	In this case, we actually do not use the vertex $j$.
	Without loss of generality, we may therefore assume that $j=r+2$.
	By the definition of $r$, the graph $F$ has an interval $(r+1)$-colouring with colour classes $V_1,\dots,V_{r+1}$, where $V_i$ is empty.
	We obtain a new interval colouring by taking one vertex from $V_{i+1}$ (or $V_{i-1}$ if $i=r+1$) and adding it to $V_i$.
	Hence we can define $\phi$ by sending the vertices of $V_k$ to $k$ for each~$k \in [r+1]$.
\end{proof}

\section{Uniform density -- Proofs} \label{sec:quasirandomness-proofs}
 
The goal of this section is to show \cref{thm:Ding-Han-Sun-Wang-Zhou}.
We begin by showing \cref{cor:uniformly-dense}.
To apply our framework, we require a suitable Inheritance Lemma.
For stronger forms of uniform density, such lemmas have been studied by Mubayi and Rödl~\cite{MR04}, Alon, Fernandez de la Vega, Kannan and Karpinski~\cite{AFKK02} as well as Czygrinow and Nagle~\cite{CN11}.
Related considerations can be found in the work of Conlon, H{\`a}n, Person and Schacht~\cite{CHP+12} and Kohayakawa, Nagle, R{\"o}dl and Schacht~\cite{KNR+10}.
In combination with a standard application of the Weak Hypergraph Regularity Lemma, we obtain the following Inheritance Lemma for uniform density.
Since it is not clear how widely known this fact is~\cite{DHS+22}, we present a proof in \cref{sec:inheritance-lemma-uniformly-dense}.

\begin{lemma}[Inheritance Lemma for uniform density]\label{lem:inheritance-uniformly-dense}
	For $1/k,\,1/{q},\,d \gg \eps' \gg \eps,\, 1/s \gg 1/n$, let $G$ be a $k$-graph on $n$ vertices satisfying $\DenF{\eps}{d } $.
	Then the property $s$-graph $P=\PG{G}{\DenF{ \eps'}{d}}{s}$  satisfies $\delta_{q}(P) \geq  (1- e^{- s^{1/(2k)}}  )  \tbinom{n-q}{s-q}$.
\end{lemma}

In combination with \cref{thm:framework-exceptional}, we may derive a property decomposition for perfect tilings under uniform density.

\begin{proof}[Proof of \cref{cor:uniformly-dense}]
	Consider a $k$-graph $F \in \cF({\spa}) \cap \cF({\div}) \cap \cF({\cov})$ with $m = v(F) \geq k$.
    Introduce constants with $1/{m},\,d,\,\mu \gg \eps'' \gg 1/r \gg  \rho \gg \eps' \gg \eps \gg  1/s \gg 1/n$ with~$r$ and $n$ divisible by $m$ as well as $\alpha = \mu/4$, let $G$ be an $n$-vertex $k$-graph in $\DenF{\eps}{d} \cap \DegF{1}{\mu}$.
    Set $P_r =\PG{G}{\DenF{\eps''}{d} \cap \DegF{1}{\mu/2}}{r}$.
   	We apply \cref{lem:inheritance-minimum-degree,lem:inheritance-uniformly-dense} with $m$ playing the rôle of $q$ to find that $\delta_m(P_r) \geq  (1- 2e^{- r^{1/(2k)}} ) \tbinom{n-m}{r-m}$.
   	Set $H = H(F;G)$.
	By the definition of  $\cF({\cov})$ and $\cF({\div})$, the $m$-digraph $H[S]$ satisfies $\div \cap \cov$ for every edge $S \in  P_r$.
	We deduce that $\delta_1\left(\PG{H}{\div \cap \cov}{r}\right) \geq  \left( 1-1/r^2 \right)  \binom{n-1}{r-1}$ by \cref{fct:monotone-degrees}.
	
	Next, let $A \subset V(G)$ have size at most $\alpha n$, and set $P_s = \PG{G-A}{\DenF{\eps'}{d} \cap \DegF{1}{\mu/4}}{s}$.
	Note that for all $X_1,\dots, X_k \subset V(G-A)$, we have
	\[
	e_{G-A}(X_1,\dots,X_k) \geq d\, |X_1|\cdots |X_k| - \eps n^k \geq d \, |X_1|\cdots |X_k| - 2\eps ( n-|A|)^k\,.
	\]
	It follows that $G-A$ {is in $ {\DenF{2\eps}{d} \cap \DegF{1}{\mu/2}}$.}
	Thus \cref{lem:inheritance-minimum-degree,lem:inheritance-uniformly-dense} applied with $m$ playing the rôles of $s$ reveal that $\delta_m(P_s) \geq (1- 2e^{- s^{1/(2k)}} ) \tbinom{n-m}{s-m}$.
	 So $\delta_1\left(\PG{H}{\SpaF{\rho}}{s}\right) \geq  \left(1-1/s^2\right)  \binom{n-1}{s-1}$. 
	It follows that $G$ contains a perfect $F$-tiling by \cref{thm:framework-exceptional}.
\end{proof}

Given this, we recover \cref{thm:Ding-Han-Sun-Wang-Zhou} by showing \cref{lem:quasirandom-divisibility} following the ideas of Ding et al.~\cite{DHS+22}.

\begin{proof}[Proof of \cref{lem:quasirandom-divisibility}]
	Consider a $k$-graph $F \in \cF({\cov})$ with $m= v(F)$.
	Let $1/m,\, d,\, \mu \gg \rho \gg \eps \gg 1/n$.
	Let $G$ be an $n$-vertex $(\eps,d)$-dense $k$-graph with $\delta_1(G) \geq \mu \, \binom{n-1}{k-1}$.
	Set $H=H(F;G)$.
	Recall the definition of lattice completeness in \cref{sec:lattice-completeness}.
	By \cref{obs:lattice-completeness}, it suffices to show that the lattice $\cL  = \cL(H)$ contains all transferrals $\vn_{v} - \vn_{u}$ with $u,v \in V(H)=V(G)$.
	We write $N(v)$ for the collection of $(k-1)$-sets $X$ such that $X\cup\{v\} \in G$.
	
	\begin{claim}\label{cla:quasirandomness-joint-neighbourhood}
		If $u, v \in V(G)$ satisfy $|N(u) \cap N(v)| \geq \rho n^{k-1}$, then $\cL$ contains $\vn_{v} - \vn_{u}$.
	\end{claim}

	\begin{proofclaim}
		Let $G'$ be obtained from $G$ by deleting $v$ and all edges $Y \cup \{u\}$ with $Y \in N(u) \sm N(v)$.
		Note that for all $X_1,\dots, X_k \subset V(G')$, we have
		\[
		e_{G'}(X_1,\dots,X_k) \geq d\, |X_1|\cdots |X_k| - \eps n^k - n^{k-1} \geq d\, |X_1|\cdots |X_k| - 2\eps (n-1)^k  \,.
		\]
		So $G'$ is a $(2\eps,d)$-dense $k$-graph with $\delta_1(G') \geq \rho \binom{n-1}{k-1}$.
		Set $H' = H(F;G')$.
		Since $F \in \cF({\cov})$, there is an $e \in H' \subset H$ with $\multi(u,e)=1$.
		By the definition of $G'$, we may obtain  $f \in H$ with $\multi(v,f)=1$ by replacing $u$ with $v$ in $e$.
		Since $\vn_{v} - \vn_{u} = \vn_{f} - \vn_{e} \in \cL(H)$, we are done.
	\end{proofclaim}
	 
	Define an auxiliary $2$-graph $J$ on $V(G)$ by adding an edge $uv$ whenever $\vn_{v} - \vn_{u} \in \cL$.
	Denote by $\alpha(J)$ the independence number of $J$, that is, the size of a maximum set of pairwise non-adjacent vertices.
	By \cref{cla:quasirandomness-joint-neighbourhood} and the assumption of  $\delta_1(G) \geq \mu \, \binom{n-1}{k-1}$, we have $\alpha(J) \leq 1/\rho$.
	We may therefore find a set $W \subset V(J)$ with $|W|\leq  (2\rho^2/\rho) n \leq 2\rho n$ such that $\delta_1(J-W) \geq \rho^2 n$.
	(Iteratively remove a low-degree vertex with at most $\rho^2 n$ neighbours in the {remainder of $J$ together} with its neighbourhood.
	Since the low-degree vertices form an independent set, this process stops after $\alpha(J)$ steps.)
 
	Now consider $v_1,v_2 \in V(G - W)$.
	We continue to show that $\vn_{v_1} - \vn_{v_2} \in \cL$.
	For $i \in [2]$, let $U_i$ contain the vertices $u \neq v_i$ such that $\vn_{v_i} - \vn_{u}  \in \cL$, and note that $|U_i| \geq \rho^2 n$ by definition of $W$.
	
	\begin{claim}\label{cla:quasirandom-divisibility-vx-U_i}
		For each $i\in [2]$, there is a vertex $u \in U_i$ and an edge $e$ in  $H(F; G[U_1 \cup \{u\}])$  with $\multi(u,e)=1$.
	\end{claim}

	\begin{proofclaim}
		Set 
		\begin{equation*}
			X = \left\{v \in V(G)\colon \deg_{G[U_1 \cup \{v\}]}(v) < \frac{d}{2} \frac{|U_1|^{k-1}}{(k-1)!}\right\}.
		\end{equation*}
		Since $G$ is $(\eps,d)$-dense, we obtain 
		\[\tfrac{d}{2} \, |X|\,|U_1|^{k-1} > e_G(X,U_1,\dots,U_1) \geq {d\, |X|\,|U_1|^{k-1} - \eps n^k}.\]
		So $\eps n^k \geq (d/2) |X|\,|U_1|^{k-1} \geq |X| \rho^{2k-2}n^{k-1}$, where we used that $|U_1| \geq \rho^2 n$ by definition of $W$.
		It follows that $|X| \leq \sqrt{\eps} |U_1|$.

		For $i\in [2]$ and $U_i' = U_i \sm X$, we have $|U_i'| \geq (\rho^2/2) n$.
		Fix an arbitrary $u \in U_1' \cup U_2'$, and let $G' = G[U_1' \cup \{u\}]$ have order $n'$.
		Set $\eps ' = \eps/(\rho^2/2)^k$.
		Note that for all $X_1,\dots, X_k \subset V(G')$, we have
		\[
		e_{G'}(X_1,\dots,X_k) \geq d\, |X_1|\cdots |X_k| - \eps n^k - n^{k-1} \geq d\, |X_1|\cdots |X_k| - \eps' n'^k\,.
		\] 
		So $G'$ is a $(\eps',d)$-dense $k$-graph with $\delta_1(G') \geq (d/4) \binom{n'-1}{k-1}$.
		Since $F \in \cF({\cov})$, it follows that there is an edge~$e$ in~$H(F; G')$  with $\multi(u,e)=1$.
	\end{proofclaim}
	
	Fix $u \in U_2$.
	By \cref{cla:quasirandom-divisibility-vx-U_i}, there is an edge $e$ in $H(F; G[U_1 \cup \{u\}])$  with $\multi(u,e)=1$.
	We may assume that $V(F) = [m]$ and $e(m)=u$.
	Informally, one can generate $(m-1)\vn_{v_1} + \vn_{v_2}$ in $\cL$ as follows.
	We begin with a surplus at every vertex of $e$.
	Then we send the weight of~$u$ to $v_2$ and the weight of the other vertices of $e$ to $v_1$.
	Formalising this (in reverse order) gives
	\begin{equation*}
		(m-1)\vn_{v_1} + \vn_{v_2} = \vn_{e} +  (\vn_{v_2} - \vn_{u} ) +  \sum_{i=1}^{m-1}   \vn_{v_1} - \vn_{e(i)}  \in \cL\,.
	\end{equation*}
	By \cref{cla:quasirandom-divisibility-vx-U_i}, there is also an edge $f$ in $H(F; G[U_1])$.
	Hence, we obtain
	\begin{equation*}
		m\vn_{v_1} = \vn_{f} +  \sum_{i=1}^{m}  \vn_{v_1} - \vn_{f(i)}   \in \cL \,.
 	\end{equation*}
 	Subtracting one from the other gives $\vn_{v_1}  - \vn_{v_2} \in \cL$ as desired.
	
	Finally, we have to handle the vertices of $W$.
	Consider an element $w \in W$.
	It suffices to show that $\vn_v - \vn_w \in \cL$.
	To see this, let $G'=G-(W \sm \{w\})$.
	Since $|W| \leq 2\rho n$, it follows that $G'$ is still a $(2\eps,d)$-dense $k$-graph with $\delta_1(G') \geq (\mu/2) \binom{n-1}{k-1}$.
	Since $F \in \cF({\cov})$, there is an edge $g$ in $H(F; G')$  with $\multi(w,g)=1$.
	For convenience, we assume that $g(m)=w$.
	For any $v \in V(G-W)$, we may now write
	\begin{equation*}
	 \vn_w +	(m-1)\vn_v    =  \vn_g + \sum_{i=1}^{m-1} \vn_v  - \vn_{g(i)}   \in \cL\,.
	\end{equation*}
	Moreover, by the above we have $m\vn_{v} \in \cL$, which after subtraction yields $\vn_{v}  - \vn_{w} \in \cL$ as desired.
\end{proof}

\section{Conclusion}\label{sec:conclusion}

The framework presented in this paper decomposes perfect tiling problems for dense hypergraphs into three independent subproblems concerning space, divisibility and covering conditions.
While this sheds light on the general nature of the problem, many open questions remain.
In the following, we discuss a few of them together with some remarks.

\subsection*{Minimum degree conditions}

In \cref{thm:k-partite}, we determined the minimum $(k-2)$-degree thresholds of complete $k$-partite $k$-graphs $F$ with $\gcd(F)=1$ for space and divisibility property, and gave an upper bound for the covering property.
A natural open question is to determine the threshold for covering completely.
It would also be interesting to extend these results to non-complete $k$-partite $k$-graphs and $k$-graphs $F$ with $\gcd(F) \geq 2$.
For $k=3$, the latter was achieved by Han, Zang and Zhao~\cite{HZZ17} and a generalisation to $k\geq4$ seems not out of reach.
Solving the problem for non-complete graphs could be more challenging.
Beyond this, it would be important to understand the behaviour of perfect fractional tilings under degree conditions in more depth.
It is conceivable that the bounds arising from the space barriers $G_{n,i,\beta}$ of \cref{cons:space} are tight.
In light of \cref{lem:linkgraph-threshold} one could approach this problem via the following `lopsided' version of the (fractional) Erd\H{o}s Matching Conjecture~\cite[Conjecture 1.4]{AFH+12}.

\begin{conjecture}\label{con:lEMC}
	For every complete $k$-partite $k$-graph $F$ and $0 \leq \beta \leq 1/v(F)$, we have
	\begin{equation*}
		\th_0(\spa_F;\beta) = \max_{i\in [k]}  {\limsup_{n} e(G_{n,i,\beta})/\tbinom{n}{k}}{}\,.
	\end{equation*}
\end{conjecture}

Note that \cref{thm:GH} of Grosu and Hladk{\`y}~\cite{GH12} confirms this for $k=2$.
Moreover, Gan, Han, Sun and Wang \cite{GHS+23} conjectured the same for $k$-graphs $F$ consisting of two edges intersecting in~$b$ vertices, which was subsequently resolved for $k=3$ and $b=2$ by Han, Sun and Wang~\cite{HSW24}.
Chen,  Liu, Sun and Wang~\cite{NXL+25} gave a counterexample to a stronger variant of \cref{con:lEMC} for non-complete graphs that appeared in an earlier version of this paper.
For further progress towards \cref{con:lEMC} and similar problems, see also the work of Hou, Hu, Li, Liu, Yang and Zhang~\cite{HHL+25}.

\subsection*{Transversal setting}

Perfect transversal tilings are determined by the maximum thresholds $\th_d^{}(\til_F)$ and $\th_d^{}(\rmix_F)$ as detailed in \cref{thm:transversal-tiling}.
Under which conditions is one always larger than the other?
For graphs, this is simple to answer (\cref{obs:transversal}).
But what about hypergraphs?
It is easy to see that $\th_{k-1}^{}(\rmix_F)=0$ when $F$ is $k$-partite and $k$-uniform.
Together with \cref{thm:transversal-tiling}, this results in a transversal version of \cref{thm:mycroft}.
A next step would be to determine $\th_{k-2}^{}(\rmix_F)$, which seems more difficult, but not impossible.
The general question is related to the study of covering problems, such as the work of Falgas-Ravry, Markström and Zhao~\cite{FMZ21}, and it is likely that the techniques of this field could apply here as well.
Finally, we note that Bowtell, Kathapurkar, Morrison and Mycroft~\cite{BKM+25} and Liu, Nie, Yang and Zhang~\cite{LNY+25} used \cref{thm:transversal-tiling} to determine minimum codegree thresholds for transversal perfect tilings with generalised triangles.

\subsection*{Ordered graphs}

A natural question is whether one can extend the results for ordered graphs to the hypergraph setting.
In particular, is it possible to generalise \cref{thm:freschi-treglown} to $k$-graphs under minimum codegree conditions?
Another possible application for our framework could be the problem of edge-ordered tilings recently investigated by Araujo, Piga, Treglown and Xiang~\cite{APT+24}.

\subsection*{Dependencies}

The Regularity Lemma of Szemerédi~\cite{Sze76} is the basis of many results in graph and hypergraph tiling theory.
In particular, the original proofs of \cref{thm:komlos,thm:kuhn-osthus,thm:mycroft,thm:montgomery-muyesser-pehova,thm:freschi-treglown,thm:Ding-Han-Sun-Wang-Zhou}  each rely on either the Weak or the Strong Hypergraph Regularity Lemma, often in conjunction with a Blow-up Lemma.
This leads to serious dependencies on the order of the host graphs and, in some cases, rather involved technicalities.
It is therefore noteworthy that except for \cref{thm:Ding-Han-Sun-Wang-Zhou}, our proofs completely avoid these powerful instruments, lending further support to the principle of ``laziness paying off''~\cite{Sze13}.

That being said, the proof of \cref{lem:blow-up-matching} as it is written results in rather poor (tower-type) bounds.
One can overcome this limitation by proving a variant of \cref{lem:larger-matching} in which the complete $s$-partite tuples of constant size are replaced by weakly regular tuples of linear size (defined in \cref{sec:inheritance-lemma-uniformly-dense}).
Such tuples can be found via a density increment argument.
This allows us to cover most vertices of the $s$-graph $P$ in \cref{lem:blow-up-matching} with weakly regular $s$-partite tuples of linear size.
We then finish by applying \cref{thm:erd64} to find an almost perfect $K_s^{(s)}(b)$-tiling in each of these tuples.
Overall this results in somewhat less severe requirements on the order of $P$, but still double-exponential in $s$.
For a formal implementation of this approach see the work of Lang and Sanhueza-Matamala~\cite{LS24a}.

Finally, we mentioned that \rf{lem:covering-with-blow-ups} is of independent interest when combined with the inheritance principle.
The following result illustrates this by the example of minimum degrees (\cref{lem:inheritance-minimum-degree}). 

\begin{corollary}\label{cor:almost-perfect-blow-up-tiling}
	For $ 0 \leq d < k$,  $\mu \gg 1/s \gg 1/b \gg 1/n$, let $G$ be an $n$-vertex $k$-graph with $\delta_d(G)\geq (\delta + \mu) \binom{n-d}{k-d}$.
	Then all but $\mu n$ vertices of $G$ can be covered with pairwise vertex-disjoint blow-ups $R_1(b),\dots,R_\ell(b)$, where each $R_i$ is an $s$-vertex $k$-graph with $\delta_d(R_i)\geq (\delta + \mu/2) \binom{s-d}{k-d}$.
\end{corollary}

It would be interesting to understand whether such a result can be extended to cover all but $o(n^k)$ edges with edge-disjoint blow-ups.

\subsection*{Uniformly dense graphs}

Related to the above discussion, recall that \cref{thm:Ding-Han-Sun-Wang-Zhou} reduces the tiling problem to the covering problem in the setting of uniformly dense graphs.
To show this, we used an Inheritance Lemma for Uniform Density (\cref{lem:inheritance-uniformly-dense}), whose proof relies on the Weak Hypergraph Regularity Lemma.
Can this dependency be removed?
This is true if we restrict the assumption to a stronger type of uniform density in which the density is bounded from both sides (see \cref{sec:inheritance-lemma-uniformly-dense}).
It would be interesting to know whether this could be relaxed.

\subsection*{Connectivity}

Perfect tilings are vertex-spanning, but inherently unconnected structures.
In subsequent work, the ideas introduced in this paper have been applied to obtain vertex-spanning cyclical structures.
Lang and Sanhueza-Matamala~\cite{LS24b,LS24a,LS25} developed a framework for embedding Hamilton cycles (and more complex structures) into hypergraphs.
Illingworth, Lang, Müyesser, Parczyk and Sgueglia~\cite{ILM+25} used \cref{lem:covering-with-blow-ups} to find $k$-uniform spanning spheres in $k$-graphs with large supported codegree.
Finally, Letzter and Ranganathan~\cite{LR25} developed the methods presented here to find exact bounds for Hamilton cycles in hypergraphs under large supported codegree.

\subsection*{Absorption}

{It is possible to prove \cref{thm:framework} without absorption using the strategy outlined at the beginning of \cref{sec:proof-main-result}.
For a detailed implementation of this approach, see the work of Lang and Sanhueza-Matamala~\cite{LS24a}.
It appears however that the strategy runs into difficulties when proving stronger statements such as \cref{thm:framework-exceptional}, and it is therefore not clear whether one can avoid absorption completely.}

\subsection*{Extensions of the framework}

There are several natural directions to extend or modify \cref{thm:framework}.
For instance, it seems plausible that one can still find a  matching that misses only a constant number of vertices after dropping the divisibility condition.
It would also be interesting to better understand the algorithmic aspects of \cref{thm:framework}.
Assuming that an $m$-digraph robustly admits a perfect fractional matching, under which circumstances can we decide in polynomial time whether there exists a perfect matching?
In the setting of minimum degree conditions, this was solved by Han and Treglown~\cite{HT20a} and Gan and Han~\cite{GH22}.
Lastly, we mention that the inheritance principle as presented in \cref{lem:inheritance-minimum-degree} has also found use in the settings of designs~\cite{DHLP25,HP25}, which could indicate further applications in this area.

\section*{Acknowledgments}

The author would like to thank Nicolás Sanhueza-Matamala and Alp Müyesser for many valuable comments on an earlier draft of this paper.
In addition, the author is greatly indebted to two anonymous referees whose comments helped to improve the presentation and correctness of the paper.
The research was supported by DFG (450397222), FAPESP (21/11020-9) and H2020-MSCA (101018431).

\bibliographystyle{abbrv}
\bibliography{/home/richard/Documents/Mathematics/Bibliography/bibliography.bib}

\phantomsection
\addcontentsline{toc}{chapter}{Glossary}
\printnoidxglossary[
type=symbols,
title={Notation},
style=long3col,        
]
\markboth{}{}     

\appendix

\section{Inheritance Lemma for minimum degrees}\label{sec:inheritance-lemma-minimum-degree}

For a $k$-graph $G$ on $n$ vertices, we denote the \emph{relative degree} of a $d$-set $D \subset V(G)$ by $\rdeg_G(D) = \deg_G(D) / \binom{n-d}{k-d}$.
The next lemma implies \cref{lem:inheritance-minimum-degree}.

\begin{lemma}[Degree inheritance] \label{lem:inheritance-degree}
	For every $0 \leq d \leq k$, $\eps >0$, $r\geq 0$, there is $s_0$ such that the following holds for every $s \geq s_0$.
	Suppose $G$ is a $k$-graph on $n\geq s$ vertices.
	Let $P$ be the $s$-graph on $V(G)$ with an edge $S$ whenever every $d$-set $D \subset S$ satisfies $\rdeg_{G[S]}(D) \geq \rdeg_G(D) - \eps$.
	Then $\delta_{r}(P) \geq 1 - \binom{s}{d} 2 \exp(- (\eps/8k)^2 s)$.
\end{lemma}

We show \cref{lem:inheritance-degree} using the following corollary of the Azuma--Hoeffding inequality (see Frieze and Pittel~\cite[Appendix B]{FP04}).
\COMMENT{Our proof follows the exposition of Frieze and Pittel~\cite[Appendix B]{FP04}.}

\begin{lemma}\label{lem:concentration}
	Let $V$ be an $n$-set with a function $h$ from the $s$-sets of $V$ to $\REALS$.
	Suppose that there exists $K \geq 0$ such that $|{h(S)-h(S')}| \le K$
	for any $s$-sets $S, S' \subset V$ with $|S \cap S'| = s-1$.
	Let $S \subset V$ be an $s$-set chosen uniformly at random.
	Then, for any $\ell >0$,
	\begin{equation*}
		{\rm Pr}( |h(S) - \Exp [h(S)]| \geq \ell) \leq 2 \exp\left( -\frac{ \ell^2}{2\min\{s,n-s\} K^2}\right).
	\end{equation*}
\end{lemma}

\COMMENT{
	\begin{proof}
		Consider a random permutation $\omega$ of $V$.
		Let $S \subset V$ consist of the first $s$ elements of $\omega$, and set $X(\omega) = h(S)$.
		We define the corresponding Doob martingale as follows.
		For a fixed permutation $(x_1,\dots,x_n)$ of $V$ and $0 \leq i \leq n$, let
		\[X_i(x_1,\dots,x_i) = \Exp(X \mid \omega_j = x_j,~ 1\leq j \leq i)\,.\]
		We claim that
		\begin{equation}\label{itm:AH}
			|X_i(x_1,\dots,x_i) - X_i(x_1,\dots,x_{i-1},x_i')| \leq K
		\end{equation}
		for all $i$-tuples $(x_1,\dots,x_i)$ and $(x_1,\dots,x_{i-1},x_i')$ with (internally) distinct entries.
		Indeed, in this case we can conclude by the Azuma--Hoeffding inequality~\cite[Theorem 2.25]{JLR01}.
		Otherwise, we repeat the above argument using the bijection between sets and their complements.
		It remains to show the above claim.
		Consider
		\begin{equation*}
			\Omega_1 = \{\omega \in \Omega \colon \omega_j = x_j,~1\leq j \leq i\}
		\end{equation*}
		and
		\begin{equation*}
			\Omega_1' = \{\omega \in \Omega \colon \omega_j = x_j,~1\leq j \leq i-1,~\omega_i = x_i'\}.
		\end{equation*}
		Define a map $f\colon \Omega_1 \to \Omega_1'$ as follows.
		For $\omega = x_1x_2\ldots x_{i-1}x_iy_{i+1}\ldots y_n$ and $y_j = x'_i$, set
		\[
		f(\omega) = x_1x_2\ldots x_{i-1}x'_iy_{i+1}\ldots y_{j-1}x_iy_{j+1}\ldots y_n.
		\]
		Since $f$ is a bijection, we have
		\[
		|X_i(x_1, x_2, \ldots, x_i) - X_i(x_1, x_2, \ldots, x'_i)| = \left|\frac{\sum_{y_{i+1},\ldots,y_n} (X(\omega) - X(f(\omega)))}{(n-i)!}\right| \leq K,
		\]
		as desired.
	\end{proof}
}

\begin{proof}[Proof of \cref{lem:inheritance-degree}]
	Set $V = V(G)$.
	Let $R \subset V$ be an $r$-set, let $T$ be drawn uniformly among all $(s-r)$-subsets of $V \setminus R$, and set $S = R \cup T$.
	We say that a $d$-set $D \subset S$ is \emph{bad for $S$} if $\rdeg_{G[S]}(D) < \rdeg_G(D) - \eps$.
	Moreover, $S$ is \emph{good} if no $d$-set $D \subset S$ is bad for $S$, and $P$ is the $s$-graph of good sets.
	It suffices to show that $\Pr(\text{$D$ bad} \mid D \subset S) \leq 2\exp(- (\eps/8k)^2 s)$ for each $d$-set $D \subset V$.
	Indeed, since $D \subset S$ if and only if $D \setminus R \subset T$,
	\begin{align*}
		\Pr(\text{$S$ not good}) &\leq \sum_{D \in \binom{V}{d}} \Pr(\text{$D$ bad} \mid D \subset S) \cdot \Pr(D \subset S)
		\\&\leq \binom{s}{d} 2 \exp(- (\eps/8k)^2 s )\,.
	\end{align*}
	
	To obtain the above bound, fix a $d$-set $D \subset V$ with $D \cap R = \emptyset$ and $\delta = \rdeg_G(D) - \eps \geq 0$.
	(When $D \cap R \neq \emptyset$, the same argument applies with $d$ replaced by $|D \setminus R| < d$, giving a stronger bound.)
	Conditioning on $D \subset T$, write $T = D \cup T'$ where $T'$ is a uniform $(s-r-d)$-subset of $V \setminus (R \cup D)$, so that $S = R \cup D \cup T'$.
	Let $h(T') = \deg_{G[S]}(D)$.
	For each edge $e \in G$ with $e \cap (R \cup D) = D$, we have $e \setminus D \subset V \setminus (R \cup D)$, and the probability that $e \setminus D \subset T'$ is
	\begin{equation*}
		\frac{(s-r-d)_{k-d}}{(n-r-d)_{k-d}}
		= \binom{s-r-d}{k-d}\binom{n-r-d}{k-d}^{-1}
		\geq (1 - \eps/8)\binom{s-d}{k-d}\binom{n-d}{k-d}^{-1}.
	\end{equation*}
	There are at least $(\delta+\eps)\binom{n-d}{k-d} - r\binom{n-d-1}{k-d-1} \geq (\delta + 7\eps/8)\binom{n-d}{k-d}$ edges $e \in G$ with $e \cap (R \cup D) = D$, where the error term counts at most $r\binom{n-d-1}{k-d-1}$ edges that also intersect with $R$.
	It follows that $\Exp[h(T')] \geq (\delta + 3\eps/4)\binom{s-d}{k-d}$.
	
	Moreover, for any two $(s-r-d)$-sets $T', T'' \subset V \setminus (R \cup D)$ with $|T' \cap T''| = s-r-d-1$, we have $|h(T') - h(T'')| \leq \binom{s-d-1}{k-d-1} \leq \frac{k}{s}\binom{s-d}{k-d}$.
	So by \cref{lem:concentration} applied with $s-r-d$, $(\eps/4)\binom{s-d}{k-d}$, $\frac{k}{s}\binom{s-d}{k-d}$ playing the roles of $s$, $\ell$ and $K$, we have
	\begin{align*}
		\Pr \left(    h(T') <   \delta\binom{s-d}{k-d} \right)
		&\leq \Pr \left(   \Exp[h(T')] - h(T')  \geq    (\eps/4) \binom{s-d}{k-d} \right) \\
		&\leq 2 \exp \left(-\frac{ (\eps/4)^2 s^2 }{ 2k^2 (s-r-d )}\right)
		\leq 2 \exp(- (\eps/8k)^2 s) \,. \qedhere
	\end{align*}
\end{proof}

\section{Inheritance Lemma for uniform density}\label{sec:inheritance-lemma-uniformly-dense}

To show \cref{lem:inheritance-uniformly-dense}, we introduce the notion of (weak) hypergraph regularity.
For  a $k$-graph  $G$ with pairwise disjoint $V_1,\dots,V_k \subset V(G)$, the \emph{density} of $(V_1,\dots,V_k)$ is 
\begin{equation*}
	d_G(V_1,\dots,V_k) = \frac{e_G(V_1,\dots,V_k)}{|V_1|\dots|V_k|}\,.
\end{equation*}
We say that $(V_1,\dots,V_k)$ is \emph{$\eps$-regular} if for every $X_1\subset V_1,$ $\dots,$ $X_k \subset V_k$ satisfying $|X_1| \geq \eps |V_1|,\,\dots,\,|X_k| \geq \eps |V_k|$, we have
$|	d_G(X_1,\dots,X_k) - 	d_G(V_1,\dots,V_k) | \leq \eps$.
The following Inheritance Lemma was proved by Mubayi and Rödl~\cite{MR04}.

\begin{lemma}\label{lem:inheritance-regular}	
	For $1/k,\, \eps',\,d  \gg 1/s,\, \eps, 1/n$ and $1/r \gg 1/n$.
	Let $G$ be an $n$-vertex $k$-graph with an $\eps$-regular tuple $(V_1,\dots,V_k)$ of density $d$ with $|V_i| \geq n/r$ for each $i \in [k]$.
	Let $S \subset V(G)$ be an $s$-set chosen uniformly at random.
	Then $(V_1 \cap S,\dots,V_k \cap S)$ is $\eps'$-regular with density at least $d-\eps'$ with probability at least $1-\exp(-s^{1/k}/20)$.
\end{lemma}

The original proof of \cref{lem:inheritance-regular} relies on the Weak Hypergraph Regularity Lemma.
Alternatively, one may also combine the work of Conlon, H{\`a}n, Person and Schacht~\cite[Fact 7 and Lemma 10]{CHP+12} and Kohayakawa, Nagle, R{\"o}dl and Schacht~\cite[Lemma 10]{KNR+10}.\footnote{These results show that $\eps$-regularity is equivalent to the number of certain small subgraphs.
	Since the number of small subgraphs is robustly inherited by taking a typical induced subgraph of constant order, one can work out an Inheritance Lemma without applying a Regularity Lemma.
	As a by-product, this improves the error to $\exp(-\gamma s)$ for some $1/k \gg \gamma > 0$.}
This can be used to prove a variant of \cref{lem:inheritance-uniformly-dense} for a stricter notion of uniform density, where the density is bounded from both sides.

A partition $\cV=\{V_1,\dots,V_r\}$ of an $n$-vertex $k$-graph $G$ is \emph{$\eps$-regular} if all but at most $\eps r^k$  tuples in $\cV^k$ are $\eps$-regular.
Moreover, $\cV$ is an \emph{$r$-equipartition} if $|V_1| \leq \dots \leq |V_r| \leq |V_1|+1$.
The next result was  proved implicitly by Szemerédi~\cite{Sze76}.

\begin{theorem}[Weak Hypergraph Regularity Lemma]\label{thm:weak-hypergraph-regularity-lemma}
	For all $k,r_0$ and $\eps >0$, there are~$r_1,n_0$ such that every $k$-graph $G$ on $n\geq n_0$ vertices admits a vertex partition $\cV$ that is $r$-equipartite and $\eps$-regular with $r_0 \leq r \leq r_1$.
\end{theorem}

\begin{proof}[Proof of \cref{lem:inheritance-uniformly-dense}]
	Introduce $\eps_1,\,\eps_2,\,r_0,\, r_1$ and $d,\, 1/k,\, \eps'  \gg \eps_1 \gg \eps_2 \gg 1/r_0 \gg 1/r_1 \gg 1/s$.
	Let $G$ be an $n$-vertex  $(\eps,d)$-uniformly dense graph.
	We show that the property $s$-graph $P =\PG{G}{\DenF{ \eps'}{d}}{s}$ satisfies $\delta_{q}(P) \geq \left(1-\exp(- s^{1/(2k)} ) \right)  \tbinom{n-q}{s-q}$.
	For convenience, we focus on the case of $q=0$.
	The other cases follow in the same way.
	By \cref{thm:weak-hypergraph-regularity-lemma}, there is an $\eps_2$-regular $r$-equipartition  $\cV = \{V_1,\dots,V_r\}$ of $V(G)$ with $r_0 \leq r \leq r_1$.
	Let $R$ be the weighted $k$-graph whose vertices are the indices $[r]$ of $\cV$ with an edge of weight $d'$ if the corresponding parts form an $\eps_2$-regular tuple of density $d'$.
	
	Now consider a uniformly chosen $s$-set $S \subset V(G)$.
	Let $\cV'= \{V_1',\dots,V_r'\}$ with $V_i' = V_i \cap S$ for every $i\in[r]$.
	It follows by Chernoff's bound that with probability at least $1-\exp(- \eps'  s)$, we have $|V_i'|/s = (1 \pm \eps_1)|V_i|/n$ a fixed $i \in V(R)$.
	Moreover, for every edge $\{i_1,\dots,i_k\} \in R$ of weight $d'$, it follows by \cref{lem:inheritance-regular} that $(V_{i_1}',\dots,V_{i_k}')$ is $\eps_1$-regular with density $d' \pm \eps_1$ with probability at least $1-\exp(-  s^{1/k}/20)$.
	Fix an $S$, which satisfies these properties.
	Since $R$ has $[r]$ vertices and at most $r^k$ edges, the above statements hold for all vertices and edges of $R$ simultaneously with probability at least $1 - r  \exp(-  \eps' s)  -  r^k \exp(-  s^{1/k}/20) \geq 1- \exp(- s^{1/(2k)} )$.
	We may therefore finish by showing that $G' = G[S]$ is $(\eps',d)$-uniformly dense.

	To this end, consider $X_1,\dots, X_k \subset S$.
	We have to show that
	\[e_{G'}(X_1,\dots,X_k) \geq d |X_1|\dots|X_k|- \eps's^k\,.\]
	This is trivial if one of $|X_1|,\dots,|X_k|$ is smaller than $\eps' s$.
	So let us assume otherwise.
	For every $j\in [k]$, let $Y_j \subset V(G)$ be a set such that $ {|Y_j \cap V_{i}|}/{|V_i|} =  {|X_j \cap V_{i}'|}/{|V_i'| \pm 1/|V_i|}  $ for every $i \in [r]$.
	It follows by $\eps_1$-regularity that
	\begin{equation*}
		d_{G'}(X_1 \cap V_{i_1}', \dots, X_k \cap V_{i_k}') \geq (1-2\eps_1)\, d_G(Y_1 \cap V_{i_1}, \dots, Y_k \cap V_{i_k})
	\end{equation*}
	for every $\{i_1,\dots,i_k\} \in R$.
	Therefore, we have
	\begin{align*}
		d_{G'}(X_1,\dots,X_k)
		&\geq (1- 4\eps_1) \sum_{(i_1,\dots,i_k)\colon \{i_1,\dots,i_k\} \in R} d_G(X_1 \cap V_{i_1}',\dots,X_k\cap V_{i_k}')   \\
		&\geq (1- 8\eps_1) \sum_{(i_1,\dots,i_k)\colon \{i_1,\dots,i_k\} \in R} d_G(Y_1 \cap V_{i_1},\dots,Y_k\cap V_{i_k}) \\
		&\geq  (1- 16\eps_1) \, d_{G}(Y_1,\dots,Y_k)\,.
	\end{align*}
	So in particular,
		\begin{align*}
		e_{G'}(X_1,\dots,X_k) 
		&\geq (1- 16\eps_1) \, d_{G}(Y_1,\dots,Y_k) \, |X_1|\dots|X_k| \\
		&\geq (1- 16\eps_1) \, \left(d- \frac{ \eps n^k}{|Y_1|\dots|Y_k|}\right) \, |X_1|\dots|X_k| \\
		&\geq d \, |X_1|\dots|X_k| - \eps' s^k
	\end{align*}
	 as desired.
\end{proof}

\section{Constructions}\label{sec:constructions}

In the following, we present three constructions adapted from the work of Han, Zang and Zhao~\cite{HZZ17}, which show the lower bounds of \cref{thm:mycroft,thm:k-partite}.

\begin{construction}[Cover barrier]\label{constr:local}
	Suppose $\{A_1,A_2,B,\{v\}\}$ is a partition of an $n$-set $V$ with $|A_1| = |A_2| = \lceil (\sqrt{2}-1) n \rceil$.
	Consider the $3$-graph $G$ on $V(G)=V$ whose edges are formed by
	\begin{itemize}[-]
		\item all triples $v a_1 a_2$ with $a_1 \in A_1$ and $a_2 \in A_2$,
		\item all triples $a c c'$ with $a\in A_i$ and $c,\,c' \in A_i \cup B$ for $i \in [2]$.
	\end{itemize}
\end{construction}

Let $G$ be a $3$-graph as in \cref{constr:local} for a large $n$.
To obtain the lower bound of \cref{thm:k-partite}, consider a complete $3$-partite $F$ that is not a cone.
We claim that there is no (directed) edge $g \in H(F;G)$ with $\multi(v,g)=1$.
Indeed, suppose otherwise.
By assumption on~$F$, there must be edges $e,\,f \in G$ with $v \in f$, $v \notin e$ an $|e \cap f| = 2$.
By construction, $f = v a_1 a_2 $ for some $a_1 \in A_1$ and $a_2 \in A_2$.
But then $e =  a_1 a_2 w $ with $w \in V(G - v)$ in violation of the construction.

To compute the minimum $1$-degree of $G$, let $\alpha = \sqrt{2}-1$.
Since  $2 \alpha^2 = (1-\alpha)^2$, we have 
\begin{itemize}[-]
	\item $\deg(v) =  |A_1||A_2| \approx 2\alpha^2 \binom{n}{2} = (1-\alpha)^2 \binom{n}{2}$,
	\item $\deg(a) \approx \binom{|A_i \cup B|}{2}\approx (1-\alpha)^2 \binom{n}{2}$ for $a\in A_i$ with $i\in [2]$,
	\item $\deg(b) \approx |A_1||A_1 \cup B| + |A_2||A_2 \cup B| \approx 4\alpha (1-\alpha) \binom{n}{2}$ for $b \in B$.
\end{itemize}
So the first two types of vertices have approximately the same degree, and the degree of the last type is much larger anyway.
It follows that ${\deg_{1}(G) \geq 2(\sqrt{2}-1)^2 \binom{n}{2} - o(n^2)}$.

\begin{construction}[Space barrier]\label{cons:space}
	Let $1 \leq {m}_1 \leq \dots \leq {m}_k$ be integers with ${m} = {{m}_1 + \dots + {m}_k}$.
	Let $V$ be a set of size $n$.
	For $1 \leq i \leq k$ and $0 \leq \beta \leq 1/{m}$, let $A \subset V$ be a set of size $\lfloor \beta ({m}_1+\dots+{m}_i) n \rfloor -1$.
	Let $G_{n,\,i,\,\beta}$ be the $n$-vertex $k$-graph on $V$ whose edges consist of all $k$-sets with at least $i$ vertices in $A$.
\end{construction}

Let $F$ be the complete $k$-partite $k$-graph  with parts of size ${m}_1,\dots,{m}_k$.
We claim that~$G_{n,\,i,\,\beta}$ does not contain an $F$-tiling of size $\beta n$.
To see this, observe that for each copy of $F$ in $G_{n,\,i,\,\beta}$, at least $i$ parts of $F$ are subsets of $A$.
Thus $F$ has at least $m_1 + \dots +m_i$ vertices in $A$.
Since $|A| < \beta(m_1 + \dots +m_i)n$, no $F$-tiling may have size $\beta n$ or larger.

Note that the case $\beta=1/m$ corresponds to a perfect tiling.
A straightforward calculation shows that
\begin{align*}
	\delta_{k-1}(G_{n,\,1,\,1/m}) &\geq  \frac{m_1}{m} n -o(n)\,,
	\\ \delta_{k-2}(G_{n,\,1,\,1/m}) &\geq \left(1-\left(1- \frac{m_1}{m} \right)^{2}\right) \binom{n}{2} -o(n^2)\,,
	\\ \delta_{k-2}(G_{n,\,2,\,1/m}) &\geq \left(\frac{m_1 + m_2}{m}\right)^2 \binom{n}{2}-o(n^2)\,.
\end{align*}
This implies the lower bound of the space threshold of \cref{pro:mycroft-thresholds,thm:k-partite}.

\begin{construction}[Divisibility barrier]\label{const:divisibility}
	Let $A$ and $B$ be sets of size $\lfloor n/2\rfloor+1$ and~$n-|A|$, respectively.
	Let $G$ be a $k$-graph formed by the union of two cliques on $A$ and $B$.
\end{construction}

To see how this implies the lower bound of the divisibility threshold of \cref{thm:k-partite}, note that the $k$-graph $G$ in \cref{const:divisibility} contains no  perfect $F$-tiling when $n$ is divisible by $v(F)$.
Moreover, $\delta_1(G) \approx 2^{-(k-1)} \binom{n-1}{k-1}$.

\end{document}